\documentclass[11pt, reqno]{amsart}

\usepackage[top=1in,bottom=1in,left=1.2in,right=1.2in]{geometry}
\usepackage{amsmath,amssymb,amsthm,amsfonts,enumerate, mathrsfs, mathtools}
\usepackage{relsize}
\usepackage{stmaryrd}
\usepackage{esint}
\usepackage{xcolor}
\usepackage{graphicx}

\usepackage{tikz}
\usetikzlibrary{patterns,positioning,arrows,shapes,trees,calc}
\def\vdist{1.5cm}  
\def\hdist{2cm}  
\tikzset{
  myStyle/.style={
    node distance=\vdist and \hdist, 
    every node/.style={draw, align=center, rounded corners}, 
    every path/.style={draw, thick, -stealth}, 
  }
}
\usepackage[
colorlinks=true,
linkcolor=blue,
filecolor=blue,
urlcolor=red,
citecolor=
blue]{hyperref}
\usepackage{cleveref}

\newtheorem{theorem}{Theorem}[section]
\newtheorem{lemma}[theorem]{Lemma}
\newtheorem{proposition}[theorem]{Proposition}
\newtheorem{corollary}[theorem]{Corollary}
\newtheorem{remark}[theorem]{Remark}
\newtheorem{definition}{Definition}[section]

\newcounter{counterConstant} 
\newcommand{\const}[1]{	\addtocounter{counterConstant}{1}
\edef#1{\arabic{counterConstant}}
}

\numberwithin{equation}{section}
\def\theequation{\arabic{section}.\arabic{equation}}

\def\cB{{\mathcal B}}
\def\cC{{\mathcal C}}

\def\cE{{\mathcal E}}
\def\cF{{\mathcal F}}

\def\cH{{\mathcal H}}

\def\cM{{\mathcal M}}

\def\cP{{\mathcal P}}
\def\cQ{{\mathcal Q}}

\def\cT{{\mathcal T}}

\def\mE{{\mathbb E}}

\def\mN{{\mathbb N}}

\def\mP{{\mathbb P}}

\def\mR{{\mathbb R}}

\def\bE{{\mathbf E}}

\def\bP{{\mathbf P}}
\def\bQ{{\mathbf Q}}

\def\sF{{\mathscr F}}

\def\sI{{\mathscr I}}

\def\sS{{\mathscr S}}
\def\sT{{\mathscr T}}

\def\l{\left}
\def\r{\right}
\def\<{\langle}
\def\>{\rangle}

\def\geq{\geqslant}
\def\leq{\leqslant}

\def\1{{\mathbf{1}}}
\def\p{\partial}
\def\d{\text{\rm{d}}}
\def\e{\mathrm{e}}
\def\eps{\varepsilon}

\def\div{\mathord{{\rm div}}}
\def\osc{\mathop{{\rm Osc}}}
\def\sgn{\textit{\rm sgn}}

\newcommand{\avert}{\Vert\hspace{-0.64 em}{\mathsmaller{-}}}
\newcommand{\norm}[1]{{\left\vert\kern-0.25ex\left\vert\kern-0.25ex\left\vert #1 
		\right\vert\kern-0.25ex\right\vert\kern-0.25ex\right\vert}}

\allowdisplaybreaks

\usepackage{fancyhdr}
{} 

\begin{document}
	
\title{McKean-Vlasov equations and nonlinear Fokker-Planck equations with critical singular Lorentz kernels}
	
\author{Michael Röckner, Deng Zhang and Guohuan Zhao}

\address{Faculty of Mathematics, Bielefeld University, 33615 Bielefeld, Germany; Academy of Mathematics and System Sciences, CAS, Beijing, 100190, China}
\email{roeckner@math.uni-bielefeld.de}
 
\address{School of Mathematical Sciences , Shanghai Jiao Tong University, Shanghai, 200240, China}
\email{dzhang@sjtu.edu.cn}
	
\address{State Key Laboratory of Mathematical Sciences, Academy of Mathematics and Systems Science, CAS, Beijing, 100190, China}
\email{gzhao@amss.ac.cn}

\begin{abstract} 
    We prove the existence and conditional uniqueness in the Krylov class for SDEs  with singular divergence-free  drifts in the endpoint critical Lorentz space $L^\infty(0,T; L^{d,\infty}(\mathbb{R}^d))$, $d\geq 2$, which particularly includes the $2$D Biot-Savart law. The uniqueness result is shown to be optimal in dimensions $d \geq 3$,  by constructing different martingale solutions in the case of supercritical Lorentz drifts. As a consequence, the well-posedness of McKean-Vlasov equations and nonlinear Fokker-Planck equations with critical singular kernels is derived. In particular, this yields the uniqueness of the 2D vorticity Navier-Stokes equations even in certain supercritical-scaling spaces. Furthermore, we prove that the path laws of solutions to McKean–Vlasov equations with critical singular kernel  form  a nonlinear Markov process in the sense of McKean.  
\end{abstract}
	
\maketitle

\tableofcontents
	
	
	
\section{Introduction and main results}  
\label{Sec-Intro}

Nonlinear Markov evolutions arise in various non-equilibrium statistical models, such as kinetic equations of Vlasov, Boltzmann and Landau. In the seminal paper  \cite{mckean1966class}, McKean proposed a deep connection  between nonlinear Fokker-Planck equations 
with a wide class of nonlinear Markov processes. This kind of relationship enables one to study nonlinear parabolic PDEs through underlying stochastic mechanisms governed by SDEs, now referred to as McKean-Vlasov equations, and vice versa.    

The purpose of this paper is to make progress, in the spirit of McKean, towards the understanding of McKean-Vlasov equations and nonlinear Fokker-Planck equations, particularly, with divergence-free drifts/{\it singular kernels} in the {\it endpoint critical} case for dimensions $d\geq 2$. 

One primary model of our nonlinear Fokker-Planck equations is the  2D vorticity Navier-Stokes equation \eqref{Eq:NS} 
\begin{equation}\label{Eq:NS}
    \partial_t \rho = \Delta \rho - \div(u\rho), \quad \rho|_{t=0}=\zeta, \tag{NSE}
\end{equation}
where the velocity field \(u\) can be reconstructed from \(\rho\) by the Biot-Savart law 
\begin{equation} \label{eq:BS}
    u(t,x)= \int_{\mR^2}K_{\mathrm{BS}}(x-y)\rho(t,\d y), \quad K_{\mathrm{BS}}(x)= \frac{1}{2\pi} \frac{(-x_2,x_1)}{x_1^2+x_2^2}. 
\end{equation} 
In view of It\^o's formula, \eqref{Eq:NS} corresponds to the following general McKean-Vlasov equation 
\begin{equation}\label{Eq:MV}
    \left\{
    \begin{aligned}
         &\mathrm{d} X_t = \int_{\mathbb{R}^d} K(t, X_t - y) \rho(t,\mathrm{d} y) \, \mathrm{d} t + \sqrt{2} \, \mathrm{d} W_t\\
         &\rho(t) = \mathrm{law}(X_t), \quad \rho|_{t=0} = \zeta  \in \mathcal{P}(\mathbb{R}^d)
    \end{aligned} 
    \right.  \tag{MVE}
\end{equation} 
with \(K(t,x)=K_{\mathrm{BS}}(x)\) given by \eqref{eq:BS}. The marginal law of McKean-Vlasov solutions satisfies the $2$D vorticity \eqref{Eq:NS}, and thus, \eqref{eq:BS}-\eqref{Eq:MV} provide a {probabilistic interpretation} of \eqref{Eq:NS}. We note that, rather than in the usual space $L^2_{loc}(\mR^2)$, the Biot-Savart law belongs to the {\it critical Lorentz space} $L^{2,\infty}(\mR^2)$.  
Hence, the nonlinearity in \eqref{Eq:NS} involves a singular kernel. Therefore, achieving the McKean picture for the 2D vorticity NSE, or more general nonlinear Fokker-Planck equations  
\begin{equation}\label{Eq:NFP}
\begin{cases}
    \p_t \rho  = \Delta \rho - \div \l((K*\rho) \rho\r)\\ \rho|_{t=0} = \zeta, 
\end{cases}
\tag{NFPE}
\end{equation}  
where the interaction kernel $K$ can be singular and time-dependent,  
requires the solvability theory for McKean-Vlasov equations \eqref{Eq:MV} with singular kernels. Although 
the well-posedness of the $2$D vorticity NSE has been extensively studied 
(see,  
e.g., \cite{FHM14, Isabelle2005uniqueness, gallagher2005uniqueness,giga1988two, hao2023second, barbu2023ns}), 
the singularity and criticality of kernels pose a major challenge to solve singular McKean-Vlasov equations. As a matter of fact,  very few results are known in the case of {\it critical} singular kernels in dimensions $d\geq 2$. 

In contrast to this, the solvability of McKean-Vlasov equations with {\it subcritical} kernels has been extensively studied, see, e.g., \cite{de2022wasserstein, mishura2016existence,rockner2021DDSDE, zhao2024existence}.  One efficient approach is to analyze the corresponding linearized version of \eqref{Eq:MV}, leading to SDEs with  subcritical drifts that have been well explored in the literature. We refer to \cite{krylov2005strong, mohammed2015sobolev, veretennikov1980strong2, zhang2011stochastic} and the references therein. 

The well-posedness of SDEs with {\it critical} drifts is much more challenging. Recently, significant progress has been made by Krylov in a series of papers \cite{krylov2021stochastic2, krylov2023diffusion, krylov2023strong, krylov2025strong}  in dimensions $d\geq 3$, where the strong well-posedness of SDEs with non-endpoint critical Morrey-type drifts was proved. In the endpoint critical case,  
the case of drifts in the Lorentz space \( L^\infty_t L^{d,\infty}_x\) under smallness conditions was addressed in \cite{kinzebulatov2025strong} and \cite{krylov2025strong}. For arbitrary divergence-free drifts in \( L^\infty_t L^{d,\infty}_x\) in high dimensions $d\geq 3$, the weak well-posedness of the SDEs was proved by the first and third authors \cite{rockner2023weak}.  

The remaining $2$D case is more subtle than the higher dimensional case. 
Actually, even in the simple case of SDEs with time-independent drifts, 
the standard weak sector condition is unclear for the corresponding Dirichlet form 
in dimension two, while it is valid in dimensions $d\geq 3$. Thus, solving singular SDEs with the {\it endpoint} critical \(L^\infty_tL^{2,\infty}_x\)-Lorentz drifts 
in dimension two, closely related to the Biot-Savart law, remains a challenging open problem.  

In the present paper we address this problem for singular SDEs in dimension two. Furthermore, we achieve the  picture  of McKean \cite{mckean1966class} 
for more general McKean-Vlasov equations and nonlinear Fokker-Planck equations with critical singular Lorentz kernels in all dimensions $d\geq 2$. The main results can be summarized as follows:

\begin{enumerate}[(i)]
    \item
    Well-posedness of SDEs with divergence-free  drifts in the 2D endpoint critical Lorentz space \(L^{\infty}_tL^{2,\infty}_x\).
    \item Well-posedness of McKean-Vlasov equations and nonlinear Fokker-Planck equations with singular kernels in the endpoint critical Lorentz space \(L^{\infty}_tL^{d,\infty}_x\) for dimensions $d\geq 2$. In particular, uniqueness holds in the {\it Krylov class}. 
    \item Nonlinear Markov process, in the sense of McKean, formed by the path laws of solutions to McKean-Vlasov equations with critical kernels. 
\end{enumerate} 

It should be mentioned that uniqueness is in general more difficult than existence of weak solutions to McKean-Vlasov equations and nonlinear Fokker-Planck equations. We  refer to \cite{barbu2023uniqueness} for  uniqueness in the case of  Nemytskii-type nonlinearities. For non-Nemytskii type nonlinearities like the Biot-Savart law in the 2D vorticity NSE, in the recent work \cite{barbu2023ns}, which much motivated the present work, uniqueness was obtained in a certain (sub)critical regime, and the nonlinear Markov property of the path laws of solutions to the corresponding McKean-Vlasov equations was proved in the class \(L^4_t(L^4_x\cap L^{{4}/{3}}_x)\cap L^\infty_t H^{-1}_x\), in which strong existence and uniqueness is also proved for good initial conditions. In the present work, our general uniqueness results also apply to  the 2D vorticity NSE 
model and render the following new contributions:
\begin{enumerate}[(a)]
     \item
     The uniqueness is proved for the 2D vorticity NSE in certain {\it Krylov class}, which includes a {\it supercritical} scaling space $ L^{\mathfrak{q}'}_t L^{\mathfrak{p}'}_x$ with $1/\mathfrak{p}'+ 1/\mathfrak{q}'>1$.  
     \item
     The present proof is of probabilistic nature, which is different from the analytic approach based on  the superposition principle in \cite{barbu2023ns}. In particular, it permits to derive the integration representation formula for solutions to \eqref{Eq:NS} and associated \eqref{Eq:MV} in the Krylov class,  with kernels satisfying the Aroson-type estimate, which is important in the derivation of uniqueness. 

    \item
    The probabilistic proof is also stable under small perturbations of integrable exponents,  which reveals that the unique solution class has a non-empty open interior of $(\mathfrak{p}', \mathfrak{q}')$ in the supercritical regime. 

    \item
    The nonlinear Markov property of solutions is derived in the Krylov class $ L^{\mathfrak{q}'}_t L^{\mathfrak{p}'}_x$ with $1/\mathfrak{p}'+ 1/\mathfrak{q}'>1$ for \eqref{Eq:MV} associated to the 2D vorticity NSE. 
\end{enumerate}

Let us mention that, in view of the Lady\v{z}enskaja-Prodi-Serrin  (LPS) criteria,  weak solutions to NSE in the (sub)critical-scaling regime are unique and automatically in the Leray-Hopf class \cite[Theorem 1.3]{CL22} for \(d\geq 2\). 
In the supercritical-scaling regime,  
it is usually expected that solutions are not unique. 
This has been confirmed near one endpoint of the LPS criteria, 
that is, in the space $C_tL^p_x$ with $p<2$, 
for the 2D NSE, 
recently proved by Cheskidov and Luo \cite{cheskidov2023nonuniqueness}.  
This phenomenon also occurs near another LPS endpoint in the space $L^p_tL^\infty_x$ with $p<2$
for the 3D NSE \cite{CL22}, 
and near two LPS endpoints 
for the hyper-dissipative NSE with viscosity beyond the Lions exponent $5/4$ \cite{LQZZ24}.  
The non-uniqueness of weak solutions is also expected in 
the remaining supercritical regime  \cite{CL22,LQZZ24}. 
Our uniqueness result in the Krylov class 
reveals that, 
even in the supercritical regime,  
there would exist certain unique solution class 
for the 2D vorticity NSE, as well as for the 2D NSE, 
see Theorem \ref{Thm:NS} and \Cref{Cor:NS} below.  
 
\subsection{Main results} 

\subsubsection{SDEs with endpoint critical drifts} 
We start  with  SDEs with endpoint critical drifts, which will serve as the linearization of McKean-Vlasov equations \eqref{Eq:MV} when the drift is independent of the marginal law $\rho(t)$.  
    
To be precise, we consider the following SDE in $\mR^d$: 
    \begin{equation}\label{Eq:SDE}
        X_{t}^{s,x} = x+ \int_s^t b(r, X_{r}^{s,x})\,\d r+ \sqrt{2} (W_t-W_s), \quad 0\leq s<t\leq T,   \tag{{SDE}}
    \end{equation}
where $b$ is a divergence-free vector field and belongs to the critical 
Lorentz space
    \begin{equation}\label{Eq:Ab}
        \|b\|^d_{L^\infty_t L_x^{d,\infty}} := \sup_{t\in [0,T]}\sup_{\lambda>0} \lambda^d\l|\l\{x: |b(t,x)|>\lambda\r\}\r|<\infty  \ \mbox{ and }\  \div b=0. \tag{$A_b$}
    \end{equation} 
The criticality can be seen from scaling arguments, 
see Subsubsection \ref{Subsub-SDE-(sub)critical} below for detailed explanations. 

\medskip

Let $h$ denote the heat kernel for $\Delta$ in $\mR^d$ 
 \begin{equation}\label{Eq:H}
     h(t, x):= (4\pi t)^{-\frac{d}{2}} \exp\l( -{|x|^2}/{4t} \r)
 \end{equation} 
 and 
\[
    \sI := \l\{(p,q)\in (1,\infty)^2: \frac{d}{p}+\frac{2}{q}<2\r\}. 
\]

The main results for SDEs  with endpoint critical drifts in dimension two are formulated in Theorem \ref{Thm:SDE} below. 
  
\begin{theorem}[Well-posedness of 2D critical SDE]\label{Thm:SDE}
    Let $d=2$. Assume that the drift $b$ satisfies \eqref{Eq:Ab}.  Then the following conclusions hold: 
    \begin{enumerate}[(i)]
	\item {\bf Existence:} For any $0\leq s<T$ and $x\in \mR^2$, there exists a weak solution $(X_t^{s,x})_{t\in [s,T]}$ to \eqref{Eq:SDE}, such that its density $ p_{s,t}(x,\cdot)$ satisfies the Aronson-type estimate: for any \(0\leq s\leq t\leq T\) and  \(x,y\in \mR^2\),  
	\begin{equation}\label{Eq:AE}
        \begin{aligned}
            \frac{1}{C} h\l({(t-s)}/{C}, x-y\r)&\leq  p_{s,t}(x,y) \leq C h\l(C(t-s), x-y\r), 
        \end{aligned}\tag{{AE}} 
	\end{equation}
         where  $C$ is a constant only depending on $\|b\|_{L^\infty_t L_x^{2,\infty}}$.  
        In particular, for any $(p,q) \in  \sI$,  the following Krylov-type estimate holds: 
\begin{equation}\label{Eq:Kry}
    \bE \int_{s}^T f(t,X_{t}^{s,x}) \d t \leq C \|f\|_{L^q(s,T; L_x^p)}, \tag{{KE}}
\end{equation}
where $C$ is a constant depending only on $p, q, T$ and $\|b\|_{L^\infty_t L_x^{2,\infty}}$. 
	
    \item {\bf Conditional Uniqueness:} There exists a pair \((\mathfrak{p}, \mathfrak{q}) \in \sI\), depending only on \(\|b\|_{L^\infty_t L^{2,\infty}_x}\), such that the law of the solution to \eqref{Eq:SDE} is unique within the class of all processes satisfying the following 
    generalized Krylov-type estimate 
    \begin{equation}\label{Eq:Kry'}
	    \bE \int_{s+\delta}^T f(t, X_{t}^{s,x}) \d t \leq C_\delta \|f\|_{L^{\mathfrak{q}}(s+\delta,T; L_x^\mathfrak{p})}, 
        \quad 
        \forall \delta\in (0,T-s), \tag{{KE'}}
    \end{equation}
    where $C_\delta$  depends on \(\delta\) as well as $p, q, T, \|b\|_{L^\infty_t L_x^{2,\infty}}$. Moreover, the collection of distributions \(\mP_{s,x}\) of  \((X^{s,x})\) forms a strong (linear) Markov process. 
    \end{enumerate}
\end{theorem}

In the sequel, we say that a process \((X_t)_{t\in [s,T]}\) 
(or equivalently, a probability measure $\mP$ on 
$C([s,T];\mR^d)$) belongs to the {\em Krylov class} if $X$ 
(or the canonical process under $\mP$) satisfies \eqref{Eq:Kry'} for some \((\mathfrak{p},\mathfrak{q})\in \sI\). With slight abuse of terminology, we also say a time-space function \(f: (s,T)\times \mR^d \to \mR\) belongs to the Krylov class, if \(f\in L^{\frac{\mathfrak{q}}{\mathfrak{q-1}}}(s+\delta, T; L^{\frac{\mathfrak{p}}{\mathfrak{p}-1}}_x)\) for any \(\delta\in (0, T-s)\) and some \((\mathfrak{p},\mathfrak{q})\in \sI\).

Let us mention that in the high dimensional case, where $d\geq 3$, the existence of weak solutions and the conditional uniqueness have been proved for \eqref{Eq:SDE} with the critical $L^\infty_t L^{d,\infty}_x$-Lorentz drifts by the first and third named authors \cite{rockner2023weak}. The Aronson-type estimate when $d\geq 3$ is  implied by the previous work by Qian-Xi \cite{qian2019parabolic} and Kinzebulatov-Sem{\"e}nov \cite{kinzebulatov2022heat}. 
As a consequence, together with \Cref{Thm:SDE} in the remaining case where $d=2$, 
one has  the well-posedness and the Aronson-type estimate for all dimensions $d\geq 2$, 
as formulated in Theorem \ref{Thm:SDE'} below. 

\begin{theorem}\label{Thm:SDE'}
    For any $d \geq 2$, the results in \Cref{Thm:SDE} hold with the exponents  $\mathfrak{p}$ and  $\mathfrak{q}$ potentially depending on $d$ and $\|b\|_{L^\infty_t L_x^{d,\infty}}$. 
\end{theorem}

\subsubsection{Optimality of well-posedness}  

Theorem \ref{Thm:SDE'} is optimal 
for dimensions $d\geq 3$ 
as revealed by the following non-uniqueness result.

\begin{theorem}[Non-uniqueness of supercritical SDEs] \label{Thm:example} 
    Let $d\geq 3$. 
    Then, for any $p\in (d/2, d)$, there exists a supercritical divergence-free vector field $b\in L^{p,\infty}_x$, such that there are  two distinct weak solutions to \eqref{Eq:SDE} starting from the origin, and both of them satisfy 
    the Krylov estimate \eqref{Eq:Kry} for any $(p,q)\in \sI$. 
\end{theorem}

As a consequence, 
we also obtain the non-uniqueness 
for linear Fokker-Planck equations 
with supercritical drifts. 

\begin{corollary}[Non-uniqueness of supercritical FPE] \label{Cor-LinFPK-Nonuniq}
    Let \( d \geq 3 \), and \(b\) be the same divergence-free vector field as in \Cref{Thm:example}. Then, there exist at least two distinct solutions to \eqref{Eq:LFP} in \(C_t\cP_x\) with initial data \(\delta_0 \), where $C_t \mathcal{P}_x$ denotes the set of weakly continuous probability measure-valued curves on $[0, T]$. 
\end{corollary}

\begin{remark}
    \begin{enumerate}[(i)]
        \item As mentioned before, the well-posedness of \eqref{Eq:SDE} in dimension two is more difficult than the higher dimensional case. 
        The detailed explanations of their distinctions are contained in Subsubsection \ref{Subsub-d2-d3} below.  

        \item 
        The uniqueness result in Theorem \ref{Thm:SDE} is important to establish the uniqueness of solutions to McKean-Vlasov equations and nonlinear Fokker-Planck equations with 
        endpoint critical kernels. 
        In particular, 
        the Krylov class in Theorem \ref{Thm:SDE} enables to derive the uniqueness of solutions to the $2$D vorticity NSE in certain supercritical spaces. 
        See Theorem \ref{Thm:NS} below. 

        \item 
        We also prove the uniqueness for the linear Fokker-Planck equations, see \Cref{Prop:LFP} below.  
        We mention that 
        the uniqueness for \eqref{Eq:SDE} in general cannot imply that of the corresponding linear Fokker-Planck equation. 
        For more details, 
        we refer to the monograph \cite{bogachev2015fokker}, 
in particular to Section 9.2 which contains examples for non-uniqueness.

        \item It is worth noting that the divergence-free condition  is necessary in Theorems \ref{Thm:SDE'} and \ref{Thm:example} to guarantee both the existence of weak solutions to SDEs with critical drifts and the uniqueness in the Krylov class.  Actually, if the drift \(b\) is 
        not divergence free, 
        it was shown in \cite[Example 7.4]{beck2019stochastic} that \eqref{Eq:SDE} 
        may not have weak solutions 
        in the case of critical drifts in $L^{d,\infty}(\mR^d)$, $d\geq 2$. 
        The divergence-free condition of drifts also matches naturally  the incompressibility  condition of velocity fields in the NSEs. 
    \end{enumerate}
\end{remark}

\subsubsection{McKean-Vlasov equations and nonlinear Fokker-Planck equations with critical kernels}

By virtue of \Cref{Thm:SDE'}, we obtain the well-posedness results for both the McKean-Vlasov equations \eqref{Eq:MV} and nonlinear Fokker-Planck equations \eqref{Eq:NFP} with critical singular kernels in all dimensions $d\geq 2$. 

Let \(\cM(\mathbb{R}^d)\) denote the space of all real measures on \(\mathbb{R}^d\) with finite total variation, and \(\cP(\mathbb{R}^d)\) the space of all probability measures on \(\mathbb{R}^d\). We denote by $|\zeta |(\mathbb{R}^d)$ the total variation of $\zeta\in \cM(\mathbb{R}^d)$.  
Any \(\zeta \in \cM(\mR^d)\) 
can be decomposed uniquely as 
\[
    \zeta = \zeta_c+\zeta_a, 
\]
where $\zeta_c$ is the continuous part, i.e.,  $\zeta_c(\{x\})=0$ for all $x\in \mR^d$, and $\zeta_a=\sum_{i}c_i \delta_{x_i}$ is the purely atomic part. 
Set 
\[
    \cP_\eps(\mR^d):=\l\{\zeta \in \cP(\mR^d): \zeta_a(\mR^d)\leq \eps \r\}. 
\]

We say a map $\mu:[0,T]\to \cM(\mR^d)$ is narrowly continuous if $t\mapsto \int f \d \mu_t$ is continuous for all 
$f\in C_b(\mR^d)$.

    \begin{theorem} [Well-posedness of McKean-Vlasov equations]  \label{Thm:MV}
        Let $d\geq 2$. Assume $K\in L^\infty_t L_x^{d,\infty}$ and $\div K=0$.  Then, the following holds: 
        \begin{enumerate}[(a)]
            \item[(i)] For any $\zeta \in\cP(\mR^d)$, there exists at least one weak solution \( (X_t) \) to \eqref{Eq:MV}. Moreover, for each \(t\in (0,T]\) the distribution of \(X_t\) admits a smooth bounded density. 
            \item[(ii)] There exist $(\mathfrak{p},\mathfrak{q})\in\sI$ and  $\eps_0>0$, depending on $d$ and $\|K\|_{L^\infty_tL^{d,\infty}_x}$,  such that if $\zeta \in \cP_{\eps_0}(\mR^d)$, then the weak solution to equation \eqref{Eq:MV} is unique in  the class of processes $(Y_t)$ satisfying the following property: for any $\delta\in(0,T)$, there is a constant $C_\delta$ such that 
            \begin{equation}\label{eq:kry_d}
                \bE \int_\delta^T f(t, Y_t) \d t\leq C_\delta \|f\|_{L^{\mathfrak{q}}(\delta, T; L^{\mathfrak{p}}_x)}. 
            \end{equation}
        \end{enumerate}
    \end{theorem}

    Parallel to the above results of \eqref{Eq:MV}, we have  the well-posedness of \eqref{Eq:NFP}.

    \begin{theorem}[Well-posedness of nonlinear Fokker-Planck equations] \label{Thm:NFP}
        Let $d\geq 2$. Assume that $K\in L^\infty_t L_x^{d,\infty}$ and $\div K=0$. 
        Then, the following holds: 
        \begin{enumerate}[(i)]
            \item For any $\zeta \in\cM(\mR^d)$, there exists at least one weak solution \(\rho\) to \eqref{Eq:NFP}. Moreover, for each \(t\in (0,T]\), \(\rho(t)\) is smooth and bounded.
            \item There exist 
$(\mathfrak{p},\mathfrak{q})\in\sI$ and   $\eps_0>0$, 
depending on $d$ and $\|K\|_{L^\infty_t L_x^{d,\infty}}$,  
such that if 
$\zeta_a \in \mathcal{M}(\mathbb{R}^d)$ 
with $|\zeta_a|(\mR^d) \leq \eps_0$, then the narrowly continuous weak solution to equation \eqref{Eq:NFP} is unique in the class of functions $\rho$ satisfying the following property: for each $\delta \in (0,T)$, 
		$$
		    \rho\in L^{\mathfrak{q}'}(\delta, T; L^{\mathfrak{p}'}_x),
		$$
            where $(\mathfrak{p}',\mathfrak{q}')=(\mathfrak{p}/(\mathfrak{p}-1), \mathfrak{q}/(\mathfrak{q}-1))$.  
        \end{enumerate}
    \end{theorem}  
    
    In the specific case where the singular kernel $K$ is the Biot-Savart law, we have the following uniqueness result for the 2D voricity NSE in certain scaling-supercritical spaces.

    \begin{theorem}[Well-posedness of 2D vorticity NSE] \label{Thm:NS}
        For any $\zeta \in \cM(\mR^2)$, the narrowly continuous weak solution to equation \eqref{Eq:NS} is unique in the class of functions $\rho$ satisfying the following property: for each $\delta \in (0,T)$, 
	\begin{align} \label{uniq-supercritical}
		\rho\in L^{\mathfrak{q}'}(\delta, T; L^{\mathfrak{p}'}_x),
	\end{align} 
	where $\mathfrak{p}'$ and $\mathfrak{q}'$ are the same exponents as in Theorem \ref{Thm:NFP} satisfying ${1}/{\mathfrak{p}'}+{1}/{\mathfrak{q}'}>1$.  
    \end{theorem}

\begin{remark}
\begin{enumerate}[(i)]
   \item  
   The uniqueness problem of the 2D vorticity NSE has been extensively studied in literature. 
    In the remarkable paper \cite{Isabelle2005uniqueness},  
    the uniqueness was proved for solutions 
    \( w \in C((0,T); L^1 \cap L^\infty) \)  
    satisfying $\|w(t)\|_{L^1} \leq C$ for all $t\in (0,T)$ 
    and 
    $w(t) \rightharpoonup \mu \in \mathcal{M}(\mR^2)$  
    as $t\to 0$. 
    In \cite{FHM14},  
    the uniqueness   was proved 
for weak solutions $w\in C([0,T]; \mathcal{D}') \cap L^\infty(0,T; \mathcal{P})$ 
    such that $\nabla w \in L^{2q/(3q-2)}(0,T; L^q)$, 
    $\forall q\in [1,2)$.  
    Very recently, 
the uniqueness in the class $L^\infty_t H^{-1}_x \cap L^4_t(L^4_x\cap L^{{4}/{3}}_x)$ was proved in \cite{barbu2023ns}.

    In Theorem \ref{Thm:NFP},  
     the uniqueness is derived in the Krylov class 
    for general nonlinear Fokker-Planck equations with critical kernels, 
    which include the 2D vorticity NSE. 

    \item     
    We remark that the Krylov class allows for obtaining 
    uniqueness of solutions to  \eqref{Eq:NS} in certain scaling-supercritical spaces \( L^{\mathfrak{q}'}_t L^{\mathfrak{p}'}_x \) 
    with $(\mathfrak{p},\mathfrak{q})\in\sI$. 
 
    Actually, 
    the $2$D vorticity NSE has the invariant scaling  
    \[ 
        \rho_\lambda(t,x) := \lambda^2 \rho(\lambda^2 t, \lambda x), \quad K_\lambda (t,x) := \lambda K(\lambda^2 t, \lambda x),\ \ \lambda>0. 
    \]
    Note that, the Biot-Savart kernel 
    is invariant under the scaling, namely, $K_\lambda(t,x) = K(t,x)$, $\forall \lambda>0$. 
    The above scaling leaves the mixed Lebesgue space $L^q(\mathbb{R}_+; L^{p}_x)$ invariant if  $ {1}/{p} + {1}/{q} = 1$. For the exponents $(\mathfrak{p}',\mathfrak{q}')$  
      in \eqref{uniq-supercritical},  
      since ${1}/{\mathfrak{p}'}  + {1}/{\mathfrak{q}'}  >1$,  
      one has 
      \begin{align*}
        \|\rho_\lambda\|_{L^{\mathfrak{q}'}_tL_x^{\mathfrak{p}'}} 
        = \lambda^{2(1-\frac{1}{\mathfrak{q}'} - \frac{1}{\mathfrak{p}'})} 
        \to \infty
      \end{align*} 
      as $\lambda \to 0^+$. 
      This justifies to say that 
       \( L^{\mathfrak{q}'}_t L^{\mathfrak{p}'}_x \) is a scaling-supercitical space for \eqref{Eq:NS}.

    It is somewhat surprising to find the uniqueness solution class for \eqref{Eq:NS} in a supercritical regime, because  usually weak solutions exhibit non-uniqueness in the supercritical regime. See, e.g., \cite{BCV22, BV19, CL22, cheskidov2023nonuniqueness,LQZZ24} for NSE in supercritical spaces with respect to the LPS criterion. In contrast to this,   Theorem \ref{Thm:NS} reveals that, even in the supercritical regime, there exist certain uniqueness class for the 2D vorticity NSE. 

    \item 
    In Theorems \ref{Thm:MV}-\ref{Thm:NS}, 
    the uniqueness results in the Krylov class for McKean-Vlasov equations and nonlinear Fokker-Planck equations also hold for an open set of the index \((\mathfrak{p}, \mathfrak{q}) \in \mathscr{I}\), thanks to the stability of our probabilistic arguments under small perturbations of the indices in $\mathscr{I}$. 
     
    The present probabilistic proof also avoids the explicit estimates used in \cite[(4.19)]{giga1988two} and \cite[Theorem 6.1]{hao2023second} for proving the uniqueness of (fractional) NSE. 
\end{enumerate} 
\end{remark}

\subsubsection{Nonlinear Markov processes}  
An important consequence of 
the above results is the following 
nonlinear Markov property for 
the path laws of solutions to \eqref{Eq:MV} in the Krylov class.    
We refer to \Cref{Sec:NMP} below for the precise definition of nonlinear Markov process.

\begin{theorem} [Nonlinear Markov process] \label{Thm:NMP}
   Assume that $K \in L^\infty_t L^{d,\infty}_x$ and $\div K = 0$. Then there exists a nonlinear Markov process $\{ \mathbb{P}_{s,\zeta} \}_{(s,\zeta) \in [0,T] \times \mathcal{P}(\mathbb{R}^d)}$ in the sense of \Cref{Def:NMP}, such that for each $(s, \zeta) \in [0,T] \times \mathcal{P}(\mathbb{R}^d)$, the measure $\mathbb{P}_{s,\zeta}$ is the law of a weak solution to \eqref{Eq:MV} with the initial condition $(s, \zeta)$, and \(\mP_{s,\zeta}\) lies in the Krylov class with the same index $(\mathfrak{p}', \mathfrak{q}')$ as in \Cref{Thm:NFP}. Moreover, the following uniqueness results hold: 
    \begin{enumerate}[(i)]
        \item \(\mP_{s,\zeta}, s\in [0,T],\zeta\in \cP_{\eps_0}\), is unique in the Krylov class, where \(\eps_0>0\) is the small number as in \Cref{Thm:MV}. 
        \item \(\mP_{s,\zeta}, s\in [0,T],\zeta\in \cP(\mR^2)\), is unique in the Krylov class, if \(d=2\) and \(K\) is the Biot-Savart kernel. 
    \end{enumerate}
\end{theorem} 

\subsection{Related literature}\label{sec:comment}

\subsubsection{
SDEs with critical and subcritical drifts}  
\label{Subsub-SDE-(sub)critical}

The study of strong well-posedness of SDEs with irregular drifts dates back to the pioneering works by Zvonkin \cite{zvonkin1974transformation} and Veretennikov \cite{veretennikov1980strong2}. 
The criticality of drifts of SDE 
can be seen by scaling arguments.  
More precisely, let $X(t)$ be a solution to \eqref{Eq:SDE} and $X_{\lambda}(t)=\lambda^{-1} X(\lambda^2 t)$ 
    for $\lambda>0$. Then, 
\[
    \d X_{\lambda}(t) = b_\lambda (t, X_{\lambda}(t)) \d t + \sqrt{2} \d W_{\lambda}(t), 
\]
where $b_\lambda (t,x)=\lambda b(\lambda^2 t, \lambda x)$ and, by the scaling invariance,  $W_{\lambda}(t)=\lambda^{-1} W(\lambda^2 t)$ is another standard Brownian motion. One has 
\[
    \left\|b_\lambda\right\|_{L^q_tL^{p,r}_x(\widetilde{Q}_R)}=\lambda^{1-\frac{d}{p}-\frac{2}{q}}\|b\|_{L^q_tL^{p,r}_x(\widetilde{Q}_{\lambda R})} 
\]
with $\widetilde{Q}_R= \left\{z=(t,y): s<t<s+R^2, |y-x|<R \right\}$, 
which remains invariant 
if the LPS condition holds 
\begin{align} \label{LPS} 
      \frac{2}{q} + \frac{d}{p} =1  \mbox{ with } p,q \in [2,\infty] \mbox{ and } r\in (0,\infty]
\end{align}
(see also \cite{beck2019stochastic}). 
For the small-time behavior of the process, 
it is more reasonable to consider the local norm $\|b_\lambda\|_{L^q_tL^{p,r}_x(\widetilde{Q}_R)}$ with $R<\infty$, rather than the global one. 
Note that $\|b_\lambda \|_{L^q_tL^{p,r}_x(\widetilde{Q}_R)} \to 0$ as $\lambda\to 0^+$, except the endpoint case $(p,q)=(d,\infty)$. 
This intuitively indicates that the {endpoint} critical case is more difficult to analyze than the non-endpoint critical and subcritical cases.

In the literature, there is a large number of  results 
for SDEs with subcritical drifts, 
see, for instance, 
\cite{krylov2005strong, mohammed2015sobolev,  
zhang2011stochastic} and the references therein. 
One standard strategy 
in the subcritical case is to construct a homeomorphism, 
e.g. the Zvonkin transform, to 
deal with the irregular drifts by solving the corresponding Kolmogorov equations. 
But it seems not applicable well in the critical case.

Very recently, 
the progress in dimensions $d\geq 3$ has been obtained by Krylov  \cite{krylov2021strong, krylov2023diffusion,krylov2023strong,krylov2025strong} regarding the existence of strong solutions to SDEs with non-endpoint critical and supercritical drifts. These developments are based on an analytical criterion for the existence of strong solutions initially applied in \cite{krylov2021strong}. Related results obtained via Malliavin calculus for \( d \geq 3 \) can be found in \cite{kinzebulatov2025strong} and \cite{rockner2025strong}. 

Lastly, 
we also would like to refer to  \cite{galeati2025almost,grafner2024weak, hao2023sdes, lu2025non, zhang2021stochastic} 
for SDEs in the supercritical case.

\subsubsection{Distinctions 
of critical SDEs in dimensions $d=2$ and $d\geq 3$}  
\label{Subsub-d2-d3} 

The subtleness in dimension two  to solve 
SDEs with critical $L^{d,\infty}_x$-Lorentz
drifts 
can be seen as follows. 

From the perspective of the theory of Dirichlet forms, 
the Dirichlet form corresponding to the operator \( \Delta + b \cdot \nabla \), 
where \( b \) is time-independent, 
is given by:
    \[ 
        \cE (u, v) = (\nabla u, \nabla v) - (b \cdot \nabla u, v), \quad \forall u, v \in H^1(\mathbb{R}^d).  
    \] 
    In dimensions \( d \geq 3 \), 
    it is a regular Dirichlet form and satisfies the standard  sector condition   
    \begin{align*}
	    \cE (u, v) \leq&  \sqrt{\cE (u, u)} \sqrt{\cE (v, v)}+ \sqrt{\cE (u, u)} \|bv\|_{L^2}\\
	    \leq & \sqrt{\cE (u, u)} \l( \sqrt{\cE (v, v)}+C\|b\|_{L^{d,\infty}}\|v\|_{L^{\frac{2d}{d-2},2}} \r)\\
        \leq& \l(1+C\|b\|_{L^{d,\infty}}\r) \sqrt{\cE (u, u)} \sqrt{\cE (v, v)}, \quad \forall u,v\in H^1(\mR^d), 
    \end{align*} 
    where the last step is due to estimate \eqref{Eq:Sob} below. 
    Hence, by virtue of the theory of Dirichlet forms (see  \cite{ma1992dirichletform}),  
    there exists a unique Hunt process associated with the Dirichlet form \((\cE, H^1(\mathbb{R}^d))\).  
    But 
    the sector condition is still unclear 
    in the low dimension 
    case $d=2$. 
    
   Another delicate point can be seen from the viewpoint  of martingale problems,  where a key step is to find a solution to the associated linear backward Kolmogorov equation \eqref{Eq:BKE1} below in the space  
   \begin{align} \label{H-2pq}
        \cH^{2,\mathfrak{p}}_{\mathfrak{q}}:=\left\{u\in L^{\mathfrak{q}}_tL^{\mathfrak{p}}_x: \p_t u, \nabla^2 u \in L^{\mathfrak{q}}_tL^{\mathfrak{p}}_x\right\}  \mbox{ with } (\mathfrak{p},\mathfrak{q})\in \sI. 
    \end{align}
 This requires the integrability  
    \( b \cdot \nabla u \in L^{\mathfrak{q}}_t L^{\mathfrak{p}}_x \) with \( ({\mathfrak{p}},{\mathfrak{q}})\in\sI\), 
which can  be achieved 
in dimensions $d\geq 3$ (\cite{rockner2023weak}),  
but is more difficult in dimension $d=2$.

\subsection{Ideas of the proof}   

The main novelties of our proofs 
can be summarized as follows.

\medskip 
\paragraph{\bf Meyers-type estimate} 
    The keypoint to solve the backward Kolmogorov equation \eqref{Eq:BKE1} in 
    the required space \eqref{H-2pq} 
    is obtaining a high integrability estimate of the gradient $\| \nabla u\|_{L^{\mathfrak{q}}_tL^{\mathfrak{l},1}_x}$, where $\mathfrak{l}>d$ and $\mathfrak{q}>2\mathfrak{l}/(\mathfrak{l}-d)$.  
    
    For dimensions $d \geq 3$, this estimate can be derived by proving an a priori H\"older estimate for solutions to \eqref{Eq:BKE1}, and utilizing a Gagliardo-Nirenberg type inequality together with the $L^q$-$L^p$ theory for parabolic equations. 
    See \cite{rockner2023weak} or 
    Figure $3$ in Subsection \ref{Subsec-Sob-Esti} below.

    For dimension $d=2$, 
    the primary obstacle is that the \(L^q\)-\(L^p\) theory gives only a priori estimate of 
    \(\|\nabla u\|_{L^\nu_t L^\mu_x}\) with \(\mu< 2=d\), 
    and the standard energy method gives merely an upper bound of \(\|\nabla u\|_{L^2_{t,x}}\). 
    These low integrability estimates, 
    however, are insufficient to solve the backward Kolmogorov equation \eqref{Eq:BKE1} in second order Sobolev spaces. 
    
To overcome this limitation, we introduce upgradation procedures to upgrade the integrability of the gradient of solutions. For the convenience of the readers, 
we present \Cref{Fig:2D} below to illustrate the proof strategy.

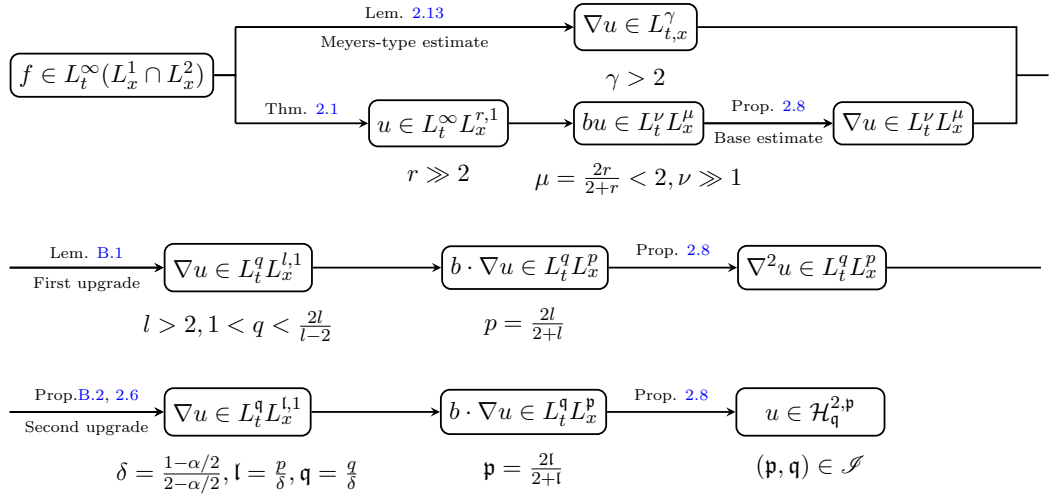
\begin{figure}[ht]
    \centering
    \resizebox{0.9\textwidth}{!}{
    \begin{tikzpicture}[myStyle] 

    \node (f) {$f\in L_t^\infty(L_x^1\cap L_x^2)$}; 
    \node (du) [right=of f, yshift=\vdist/2, xshift=1.8*\hdist] {$\nabla u\in L_{t,x}^\gamma$};
    \node [below=of du, draw=none, yshift=\vdist/1.1] {$\gamma>2$};
    \node (u)  [right=of f, yshift=-\vdist/2, xshift=\hdist/5] {$u \in L^\infty_t L^{r,1}_x$}; 
    \node [below=of u, draw=none, yshift=\vdist/1.1] {$r\gg 2$};
    \node (bu) [right=\hdist/2 of u] {$b u \in L^\nu_t L^\mu_x$};
    \node [below=of bu, draw=none, yshift=\vdist/1.1] {$\mu=\frac{2r}{2+r}<2,\nu\gg 1$};
    \node (nablau) [right=of bu] {$\nabla u \in L^\nu_t L^\mu_x$}; 

    \coordinate (blank1) at ($(f.east) + (\hdist/6, 0)$); 
    \coordinate (blank11) at (blank1 |- du); 
    \coordinate (blank12) at (blank1 |- u); 
    \draw [-] (f) -| (blank11); 
    \draw [-] (f) -| (blank12);
    \draw (blank11) -- node[above, draw=none]{\tiny Lem. \ref{Lem:Meyers}} (du);
    \draw (blank11) -- node[below, draw=none]{\tiny Meyers-type estimate} (du);
    \draw (blank12) -- node[above, draw=none] {\tiny Thm. \ref{Thm:Heat}} (u);
    \draw (u) -- (bu);
    \draw (bu) -- node[above, draw=none] {\tiny Prop. \ref{Prop:lplq}} (nablau);
    \draw (bu) -- node[below, draw=none] {\tiny Base estimate} (nablau);
 
    \coordinate (fright-center) at ($(f.east) + (6*\hdist+\hdist/4, 0)$);
    \coordinate (fright-center-1) at ($(f.east) + (6*\hdist+\hdist/2, 0)$);
    \draw [-] ($(du.east)$) -| (fright-center);
    \draw [-] ($(nablau.east)$) -| (fright-center);
    \draw [-] (fright-center) -- (fright-center-1);

    \coordinate (blank2) at ($(f.west)+(0,-\vdist*2.0)$);
    \node (nablau1) [right=\hdist*1.2 of blank2] {$\nabla u \in L^q_t L^{l,1}_x$};
    \node [below=of nablau1, draw=none, yshift=\vdist/1.1] {$l>2, 1<q<\frac{2l}{l-2}$};
    \node (bnablau) [right=\hdist of nablau1] {$b \cdot \nabla u \in L^q_t L^{p}_x$};
    \node [below=of bnablau, draw=none, yshift=\vdist/1.1] {$p=\frac{2l}{2+l}$};
    \node (nablau2) [right=\hdist of bnablau] {$\nabla^2 u \in L^q_t L^{p}_x$}; 

    \draw (blank2) --node[above, draw=none] {\tiny Lem. \ref{Lem:Inter1}}
    (nablau1);
    \draw (blank2) -- node[below, draw=none]{\tiny First upgrade} (nablau1);
    \draw (nablau1) -- (bnablau);
    \draw (bnablau) -- node[above, draw=none] {\tiny Prop. \ref{Prop:lplq}} (nablau2); 
    \coordinate (nablau2-right) at ($(nablau2.east) + (1.2*\hdist, 0)$);
    \draw [-] ($(nablau2.east)$) -| (nablau2-right);  
    \coordinate (blank3) at ($(blank2.west)-(0,\vdist*1.5)$); 
    \node (nablau3) [right=\hdist*1.2 of blank3] {$\nabla u \in L^{\mathfrak{q}}_t L^{\mathfrak{l},1}_x$};
    \node [below=of nablau3, draw=none, yshift=\vdist/1.1] {$\delta=\frac{1-\alpha/2}{2-\alpha/2}, \mathfrak{l}=\frac{p}{\delta}, \mathfrak{q}=\frac{q}{\delta}$};
    \node (bnablau1) [right=\hdist of nablau3] {$b \cdot \nabla u \in L^{\mathfrak{q}}_t L^{\mathfrak{p}}_x$};
    \node [below=of bnablau1, draw=none, yshift=\vdist/1.1] {$\mathfrak{p}=\frac{2\mathfrak{l}}{2+\mathfrak{l}}$};
    \node (w2p) [right=\hdist of bnablau1] {$\, \ \ u\in \cH^{2,\mathfrak{p}}_{\mathfrak{q}}\, \ \ $}; 
    \node [below=of w2p, draw=none, yshift=\vdist/1.1] {$(\mathfrak{p},\mathfrak{q})\in \sI$};
 
    \draw (blank3) -- node[above, draw=none] {\tiny Prop.\ref{Prop:GNI}, \ref{Prop:Holder}} (nablau3);
    \draw (blank3) -- node[below, draw=none] {\tiny Second upgrade} (nablau3); 
    \draw (nablau3) -- (bnablau1); 
    \draw (bnablau1) -- node[above, draw=none] {\tiny Prop. \ref{Prop:lplq}} (w2p);
    \end{tikzpicture}}
    \caption{Proof Strategy when \(d = 2\)}\label{Fig:2D}
\end{figure}

    We first prove that, 
    in addition to the standard energy bound, the gradient of solutions indeed obeys 
    a {\em Meyers-type estimate}. That is, the gradient of solutions has an improved integrability in \(L^\gamma_{t,x}\) for some \(\gamma > 2\), see \Cref{Lem:Meyers} below. 
    The 
    Meyers-type estimate  relies crucially on Gehring's Lemma (\Cref{Lem:Gerhing}) and 
    the reverse Hölder estimate (\Cref{Lem:RHolder}), and serves as one of the key steps in upgrading the integrability
    \begin{align*}
       \|\nabla u\|_{L^q_t L^{l,1}_x} 
       < \infty 
    \end{align*} 
    for certain exponents $l>2$ and \(1<q<2l/(l-2)\).  
    It is important that the spatial integrability exponent here can be raised above \(2\). 
    
    Then, in the second upgradation step, 
    we make use of the refined Gagliardo-Nirenberg inequality 
    (see \eqref{Eq:NE} below) 
    to further improve the integrability 
    \begin{align*}
        \|\nabla u\|_{L^\mathfrak{q}_t L^{\mathfrak{l}}_x} <\infty, 
    \end{align*} 
    where the exponents \(\mathfrak{l}>2\) and \(\mathfrak{q}>2\mathfrak{l}/(\mathfrak{l}-2)\). In both steps, it is quite delicate to 
    select appropriate integrability exponents 
    to ensure that 
    the final upgraded integrability exponents $(\mathfrak{p}, \mathfrak{q}) \in \mathscr{I}$, i.e., ${1}/{\mathfrak{p}}+{1}/{\mathfrak{q}}<1$, as required by the Krylov class.
    
    As a result, we can solve the backward Kolmogorov equation \eqref{Eq:BKE1} in 
    the desired Sobolev space 
    \eqref{H-2pq} in dimension $d=2$. 
To the best of our knowledge, 
    the solvability of equation \eqref{Eq:BKE1} in \(\cH^{2,\mathfrak{p}}_{\mathfrak{q}}\) for dimension two is new in the existing PDE literature.

\medskip 
\paragraph{\bf Construction of non-unique solutions} 
\begin{figure}[ht]
    \centering
    \includegraphics[width=4in]{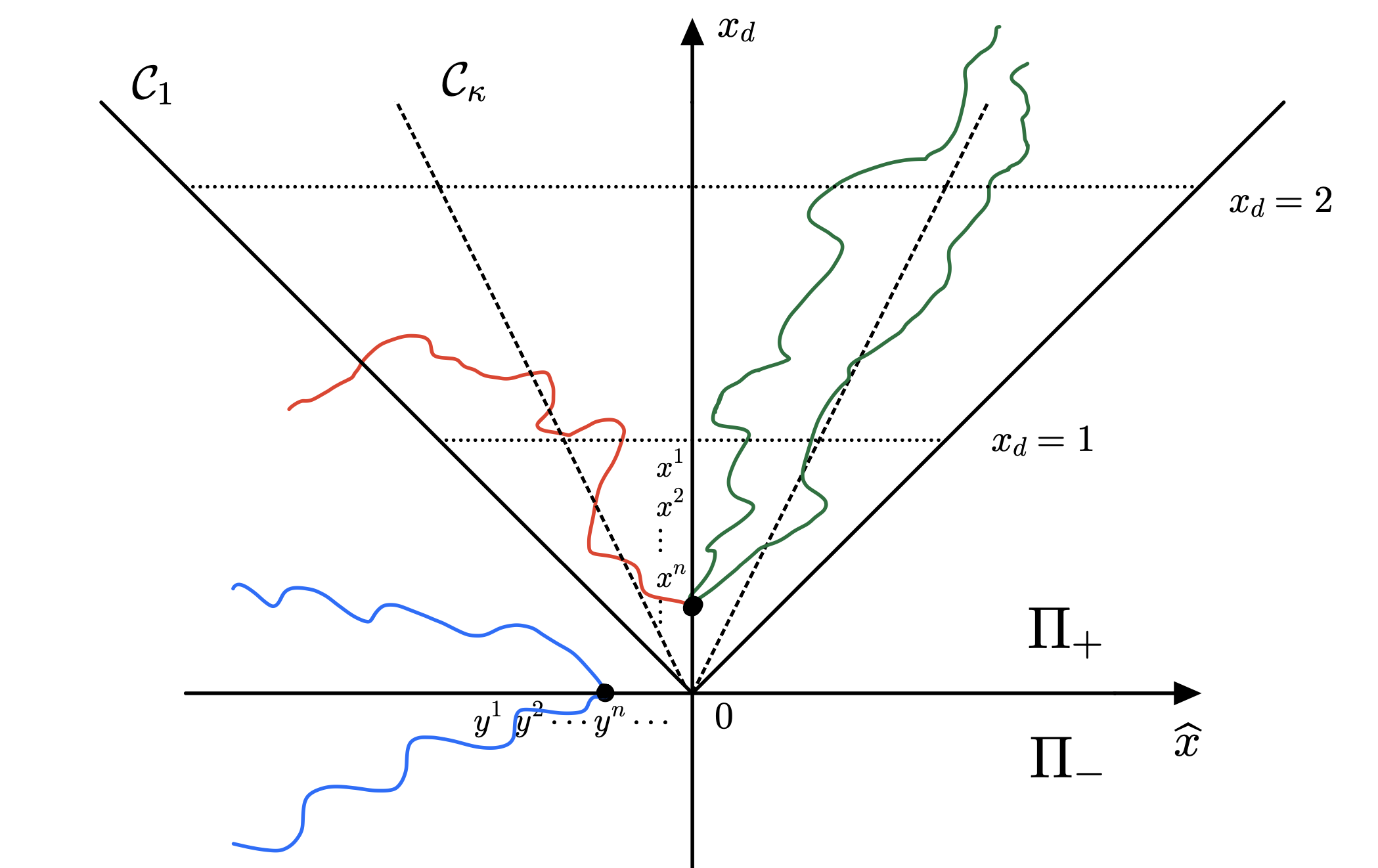}
    \caption{Solution paths}\label{Fig:path}
    \small{The solution process starting from \(x^n\) exhibits the green trajectories with high probability.} 
\end{figure}

   Concerning the non-uniqueness in Theorem \ref{Thm:example}, 
   the key observation is that
   the drift in the supercritical regime can be very singular, say, near the origin, 
   so that it suppresses the fluctuations of Brownian motions,  
   thus results in different concentrations  
   of solution paths. 

   More precisely, 
   we construct a divergence-free vector field \( b\in L_x^{p,\infty} \), 
   $p\in (d/2,d)$. 
   It is anti-symmetric with respect to the hyperplane $\Pi=\{x\in \mR^d: x_d=0\}$, 
   and when \(x\) is located in the cone \(\cC_{\kappa}=\{x: x_d>\kappa|\widehat{x}|\}\) with \(\kappa\geq 1\), 
   it can be very singular 
   in the direction of the last coordinate 
    \[
    b(x) \approx (0,\cdots, 0, \sgn (x_d) b_d(x))\approx (0,\cdots, 0,\sgn (x_d)|x|^{-d/p}).
    \] 
    
    Then, we consider two sequences 
   \( (x^n) \subseteq \{\hat{x} = 0\} \) and  \( (y^n) \subseteq \Pi \) converging to the singular point $0$ in different ways.  
    Because of the anti-symmetry of the drift,  
    the corresponding martingale solutions 
    \( \{\mathbb{P}_{y^n}\} \) are invariant under reflection across the hyperplane \(\Pi\), 
    and so is the limit \(\mathbb{P}_0 \). 
    
    In contrast, for the martingale solution \(\mathbb{P}_{x^n}\),  
    the singularity of the drift along the axis $x_d$ forces 
    the solution trajectories starting from the small cone \(\cC_\kappa\), 
    with  a uniform probability (independent of \(n\)), 
    to stay in the larger cone \(\cC_1\) and return back to \(\cC_\kappa\) 
    after certain time. 
    As a result, 
    the trajectories are more likely to 
    intersect with the hyperplane \(\{x_d = 2\}\) before exiting \(\cC_1\), and to remain in the region \(\{x_d > 1\}\) 
    for a unit time. Intuitively, as illustrated in Figure~\ref{Fig:path}, the solution process follows the green trajectories with high probability.
    Consequently, the limit \(\widetilde{\mathbb{P}}_{0} \)  of \( (\mathbb{P}_{x^n}) \) is  more concentrated on the half space 
    $\Pi_+=\{x\in \mR^d: x_d>0\}$, 
    and thus, leads to the 
    non-uniqueness  
    $  \widetilde{\mathbb{P}}_{0} \not = \mathbb{P}_{0}$.

\medskip 
\paragraph{\bf Uniqueness of solutions in the Krylov class}   
\label{Subsub-Uniq-Krylov} 

Uniqueness is usually more difficult than existence of weak solutions to McKean-Vlasov equations.  
See, for instance,  the recent work \cite{BR2024nonlinear} and \cite{barbu2023ns} for an analytic method based on  nonlinear Fokker-Planck equations and the superposition principle, and \cite[Theorem 6.3]{hao2023second} 
for the well-posedness with initial data in certain Besov spaces including \(L^{1+\eps}\cap L^1\) 
based on the semigroup method.

In the present work, we provide a direct probabilistic method to obtain the uniqueness of solutions in the Krylov class. Building upon our well-posedness results for critical SDEs, we derive an integral representation formula for the solutions in Krylov class to nonlinear Fokker–Planck equations 
\begin{align*} 
    \rho (t,y) = \int_{\mR^d} p_{0, t}^{K*\rho}(x,y) \zeta (x) \d x, 
\end{align*}
where the kernel \(p_{0,t}^{K*\rho}(x,y)\)
satisfies the Aronson-type estimate \eqref{Eq:AE},  
see \eqref{eq:rho} below. The representation formula   
is important to close Gronwall-type estimates of solutions 
in the derivation of the uniqueness.

Concerning the nonlinear Markov property, 
the one-to-one correspondence 
for a large class of nonlinear parabolic PDEs and nonlinear Markov processes was proved in \cite{rehmeier2022nonlinear}. 
See also \cite{barbu2023ns} 
and \cite{Barbu2024} in the case of singular drifts.

Let \(\cQ\subseteq \cP(\mR^d)\). In view of the criteria in \cite{rehmeier2022nonlinear}, 
two conditions are crucial for 
verifying the nonlinear Markov property for the path laws of solutions to McKean–Vlasov equations: 
\begin{itemize}  
   \item
   Flow property: the marginal distribution \((\mu^{s,\zeta})_{t\in [s,T]}\) satisfies 
   \begin{align}  \label{eq:flow} 
\mu^{s,\zeta}_t\in \cQ~\mbox{ and }~\mu^{s,\zeta}_t=\mu^{r,\mu^{s,\zeta}_r}_{t}, \ \ \forall s\leq r\leq t, \ \zeta\in \cQ\subseteq \cP(\mR^d). 
   \end{align}  
   \item  
   Extremality: $\mu^{s, \zeta}$ is an extreme point in the convex set of all weakly continuous probability solutions with the initial datum $(s, \zeta)$ to the linearized Fokker-Planck equations obtained by freezing $\rho=\mu^{s,\zeta}$ in the convolution term in \eqref{Eq:NFP}.  
\end{itemize}

Thanks to the uniqueness result of solutions to the nonlinear Fokker–Planck equation, the flow property in the present work is guaranteed. Moreover, 
our well-posedness result for SDEs ensures that $\mu^{s,\zeta}$ is an extreme point of the aforementioned solution set of the linearized Fokker-Planck equation. As a consequence, we obtain the nonlinear Markov property for \eqref{Eq:MV}. 
   
    \medskip
    
    \paragraph{\bf Organization:} 
    This paper is organized as follows: \Cref{Sec:PDE} is dedicated to establishing  the solvability of the linear backward Kolmogorov equation \eqref{Eq:BKE1} in  certain second order Sobolev spaces when $b$ satisfies \eqref{Eq:Ab}. The proofs of \Cref{Thm:SDE} and \Cref{Thm:MV} are presented in \Cref{Sec:SDE} and \Cref{Sec:MVE}, respectively. In \Cref{Sec:EX}, we provide an example demonstrating the optimality of 
    the condition \eqref{Eq:Ab}. 
    At last, 
    in \Cref{App:Lorentz} 
    we present the preliminaries of Lorentz spaces. 
    Then,  \Cref{App:Inter} 
    contains several useful interpolation estimates 
    in Lorentz spaces 
    used in this paper.

    \medskip
    
    \paragraph{\bf Notations.} 
    Let $T$ be any fixed time horizon. For $p\in [1,\infty]$,  
    $p'$ denotes its conjugate number $\frac{p}{p-1}$.  
    For $z=(t,x)\in \mR\times \mR^d$, we set $Q_R(z):= (t-R^2,t)\times B_R(x)$ and 
    $Q_R:=(-R^2,0)\times B_R$.  
    For any $p, q\in (1,\infty)$,  
    let 
    \[
    \|f\|_{L^q_t L_x^p}:= \|f\|_{L^q([0,T]; L^p(\mR^d))}.
    \]
    Moreover, let $\|f\|_{\dot{H}_x^{s,p}}:= \|\Lambda^s f\|_{L^p_x}$ and $\|f\|_{{H}_x^{s,p}}:= \|f\|_{L^p_x}+\|\Lambda^s f\|_{L^p_x}$, 
    where $\Lambda := (-\Delta)^{\frac{1}{2}}$.

\section{Linear backward Kolmogorov equations}\label{Sec:PDE}
    In this section, we study the  solvability of the backward Kolmogorov equation  corresponding to \eqref{Eq:SDE} when the drift $b$ satisfies \eqref{Eq:Ab}. We mainly focus on the most delicate case where $d=2$.  The method also applies to high dimensions \( d \geq 3 \). 

    \medskip 
    Define the Kolmogorov operator by 
    \begin{align*}
         A_t u :=\Delta u+ b(t,\cdot)\cdot\nabla u, 
    \end{align*}
     and the dual Fokker-Planck operator 
     \begin{align*}
          A^*_t u:=\Delta u-\div (b(t,\cdot) u)=\Delta u-b(t,\cdot)\cdot\nabla u. 
     \end{align*}  
     To avoid inessential issues arising from the singularity of \( b \), in the rest of this section, we assume that $b\in C_b^\infty(\mR^{d+1})$ and $b$ is divergence-free. We note that the constants in all of the following a priori estimates do not depend on the regularity of $b$, but only on the critical Lorentz-norm $\|b\|_{L^\infty_t L^{d,\infty}_x}$ , so the results in this section are valid for the drift $b$ satisfying \eqref{Eq:Ab}. 

    \subsection{Aronson-type estimate} \label{App:HKE} 
	Let us first prove the Aronson-type estimate for the fundamental solutions of $\p_t-A^*_t$. 

	\begin{theorem} [Aronson-type estimate] 
 \label{Thm:Heat}
		Let $d\geq 2$. 
        Then, $\p_t-A_t^*$ admits a fundamental solution $p_{s,t}(x,y), t\in[s, T], x,y\in \mR^d$, i.e. 
        \[
            \begin{cases}
                \p_t p_{s,t}(x,y)=[A_t^* p_{s,t}(x,\cdot)](y)\\
                \lim_{t\downarrow s}p_{s,t}(x,y)=\delta(x-y),  
            \end{cases}
        \]
        and $p_{s,t}(x,y)$ satisfies  
		\[
		    \int_{\mR^d} p_{s,r}(x,z) p_{r,t}(z,y) \d z=p_{s,t}(x,y), \quad \int_{\mR^d} p_{s,t}(x,y) \d y =1,  
		\]
		and 
		\begin{equation}\label{Eq:heat}
        \frac{1}{C} h\l({(t-s)}/{C}, x-y\r) \leqslant p_{s,t}(x,y) \leqslant C h\l(C(t-s), x-y\r), 
	\end{equation}
        where $h$ is the heat kernel for $\Delta$ on $\mR^d$ given by \eqref{Eq:H}, 
        $C$ depends only on the dimension $d$ and $\|b\|_{L^\infty_t L_x^{d,\infty}}$. Moreover, there exist constants  $\alpha\in (0,1)$ and $C>0$,  only depending on $d$ and $\|b\|_{L^\infty_t L_x^{d,\infty}}$, such that for any $t>s$ and $x, y, y'\in \mR^d$, 
		\begin{equation}\label{Eq:pHolder}
			\begin{aligned}
				&\l|p_{s,t}(x,y)-p_{s,t}(x,y')\r|\\
				\leq& C \l[ \l( \frac{|y-y'|}{\sqrt{t-s}} \r)^\alpha \wedge 1\r]
				\l[ h(C(t-s), x-y)+h(C(t-s),x-y') \r]. 
			\end{aligned}
		\end{equation}
	\end{theorem}
	
	\begin{remark}
		\begin{enumerate}[(i)]
		\item The case where $d\geq 3$  follows from \cite{kinzebulatov2022heat}. For the convenience of readers, 
        we give a unified estimate for both the  cases $d =2$ and $d\geq 3$. 
            \item Since $p_{s,t}(x,y)$ also satisfies 
                the backward equation 
                \[
                \begin{cases}
                    \p_s p_{s,t}(x,y)=[A_s p_{s,t}(\cdot,y)](x)\\
                    \lim_{s\uparrow t}p_{s,t}(x,y)=\delta(x-y),  
                \end{cases}
                \]
                and the operators \( A_t \) and \( A_t^* \) have the same form, except the opposite signs of the drift coefficients, we also have  
			\begin{equation}\label{Eq:pHolder2}
				\begin{aligned}
					&\l|p_{s,t}(x,y)-p_{s,t}(x',y)\r|\\
					\leq& C \l[ \l( \frac{|x-x'|}{\sqrt{t-s}} \r)^\alpha \wedge 1\r]
					\l[ h(C(t-s), x-y)+h(C(t-s),x'-y) \r]. 
				\end{aligned}
			\end{equation}
			\item In light of \Cref{Thm:Heat}, using standard approximation arguments 
            one also has that for any drift \( b \) satisfying \eqref{Eq:Ab}, there exists a Markov process $(\mP_{s,x}, X_t)$ such that 
			\[
			\mP_{s,x}(X_s=x)=1, \quad \mP_{s,x}(X_t\in A)= \int_{A} p_{s,t}(x,y) \d y, 
			\]
			where $p_{s,t}$ is the fundamental solution as in \Cref{Thm:Heat}.  
		\end{enumerate}
	\end{remark}

    The proof of \Cref{Thm:Heat} follows from the Nash iteration as in \cite{nash1958continuity, fabes1986new, qian2019parabolic}. Two key ingredients are Nash's inequality 
	\begin{equation}\label{Eq:Nash}
		\|u\|_{L^2_x}^{2+\frac{4}{d}} \leqslant C_d\|\nabla u\|_{L^2_x}^2\|u\|_{L^1_x}^{\frac{4}{d}}, \quad \forall u \in L^1 \cap \dot{H}^{1},  
	\end{equation}
	and the following lemma.  

\begin{lemma} [\cite{fabes1986new}]  \label{Lem:Nash}
    Suppose that $w$ is a nonnegative, nondecreasing continuous function on $[0,\infty)$. Let $p\geq 2, a,\beta,\Gamma>0$ and $\delta\in (0,1]$. Then, there exists $C$ depending only on $a$ and $\beta$ such that if
    \[
    u^{\prime}(t) \leq-\frac{a}{p} \frac{t^{p-2} \l(u(t)\r)^{1+\beta p}}{\l(w(t)\r)^{\beta p}}+p \Gamma u(t),
    \]
    then 
    \[
    t^{\frac{1}{\beta}-\frac{1}{\beta p}}u(t) \leq \left(\frac{C p^2}{\delta}\right)^{\frac{1}{\beta p}}  e^{\frac{\Gamma\delta t}{p}} w(t). 
    \]
\end{lemma}
\begin{proof}[Proof of \Cref{Thm:Heat}]
    {\bf $(i)$ Upper bound:} As in \cite{fabes1986new}, for any $\alpha\in \mR^d$, we define  
    \begin{equation*}
        \varphi(x)=\alpha\cdot x, \quad \phi=\e^{-\varphi}, \quad A_t^{\varphi*}f= \phi^{-1}A^*_t(f\phi), 
    \end{equation*}
    and 
    \[
    p^\varphi_{s,t}(x, y)=\phi^{-1}(y) p_{s,t}(x,y) \phi(x), \quad P_{s,t}^{\varphi *}f:=\int_{\mR^d} p^{\varphi}_{s,t}(x,\cdot)f(x)\d x. 
    \]
    By definition, \(p_{s,t}^{\varphi}\) is the fundamental solution of \(\partial_t-A^{\varphi*}_t\), and 
    \begin{equation}\label{eq:Aphi-f}
	\begin{aligned}
            A_t^{\varphi*} f=& \phi^{-1} \Delta (f\phi)-\phi^{-1} b\cdot \nabla (f\phi)\\
            =&\Delta f + 2\nabla f\cdot \nabla \log(\phi) + f\frac{\Delta \phi}{\phi} - b\cdot \nabla f - b\cdot\nabla \log(\phi) f\\
            =&\Delta f - (b+2 \alpha) \cdot \nabla f + (|\alpha|^2 +\alpha \cdot b) f. 
	\end{aligned}
    \end{equation}
    
    For any \(f\in L^1_x\cap L^\infty_x\), put \(f_t=P_{0,t}^{\varphi*} f\). Then \(\p_t f_t= A^{\varphi*}_t f_t\).  Using \eqref{Eq:Lady}, \eqref{Eq:Nash} and \eqref{eq:Aphi-f}, we get that for any $p\in [1,\infty)$, 
	\begin{equation}\label{eq:Df}
		\begin{aligned}
                \frac{\d }{\d t}\int f_t^{2p} =& 2p \int f_t^{2p-1} \p_t f_t = 2p \int f_t^{2p-1}  A^{\varphi*}_t f_t\\
                = &-\l(4- 2/p\r) \int  |\nabla f_t^p|^2 + 2p \int |\alpha|^2 f_t^{2p} + 2p \int \alpha\cdot b\,f_t^{2p} \\
                {\leq} &-\l(4- 2/p\r)\|\nabla f_t^p\|_{L^2_x}^2 + p|\alpha|^2 \|f_t^{p}\|_{L^2_x}^2 +C p|\alpha| \|b\|_{L_t^\infty L_x^{d,\infty}} \|f_t^p\|_{L^2_x} \|\nabla f_t^p\|_{L^2_x} \\
                \leq & - \|\nabla f_t^p\|_{L^2_x}^2 +  \l(1+C\|b\|^2_{L_t^\infty L_x^{d,\infty}}\r)p^2 |\alpha|^2 \|f_t^p\|_{L^2_x}^2\\ {\leq} & -c \frac{\|f_t^p\|_{L^2_x}^{2+\frac{4}{d}}}{\|f_t^p\|_{L^1_x}^{\frac{d}{4}}}+  \l(1+C\|b\|^2_{L_t^\infty L_x^{d,\infty}}\r)p^2|\alpha|^2 \|f_t^p\|_{L^2_x}^2,  
		\end{aligned}
	\end{equation}
    where $c$ and $C$ only depend on $d$. 
  
    Let 
		\[
		u_{p}(t) := \|f_t\|_{L^{2p}_x}, \quad   w_p (t):=\sup _{0 \leqslant s \leqslant t} s^{\frac{d}{4}-\frac{d}{2p}} \|f_s\|_{L^p_x}. 
		\]
		Then, \eqref{eq:Df} implies 
		\[
		u_p'(t) \leq - \frac{c}{p} \frac{t^{p-2} \l( u_p(t) \r)^{1+4 p / d}}{\left(w_p(t)\right)^{4 p / d}}+ C\l(1+\|b\|_{L^\infty_t L^{d,\infty}_x}^2\r) p |\alpha|^2 u_p(t). 
		\]
	Applying \Cref{Lem:Nash}  we get 
		\begin{align*}
			\|f_t\|_\infty\leq& \limsup_{k\to\infty} w_{2^k}(t) \leq C(\eps) \exp\l(\eps |\alpha|^2 t\r) w_2(t)\\
			\leq& C(\eps) \exp\l(\eps|\alpha|^2 t\r)  \sup_{s\in[0,t]}\|f_s\|_{L^2_x}. 
		\end{align*}
    Using \eqref{eq:Df} again with $p=1$ we derive that
    $$
    \|f_t\|_{L^2_x}^2 \leq \exp\l(C\l(1+\|b\|_{L_t^\infty L_x^{d,\infty}}^2\r)|\alpha|^2 t \r)\|f\|_{L^2_x}^2.
    $$
    Combining the above two estimates together one sees 
    \[
    \|P_{0,t}^{\varphi*}\|_{L^2_x \to L^\infty_x} \leq C \exp\l(C\l(1+\|b\|^2_{L_t^\infty L_x^{d,\infty}}\r)|\alpha|^2 t\r). 
    \]
    Noting that \(b\) is divergence-free, by duality, we also have 
    \[
    \|P_{t,2t}^{\varphi*}\|_{L^1_x \to L^2_x}\leq C \exp\l(C\l(1+\|b\|^2_{L_t^\infty L_x^{d,\infty}}\r)|\alpha|^2 t\r). 
    \]
    Therefore,  
    \[
	\|P_{0,t}^{\varphi*}\|_{L^1_x\to L^\infty_x}  \leq C \exp\l(C\l(1+\|b\|^2_{L_t^\infty L_x^{d,\infty}}\r)|\alpha|^2 t\r), 
    \]
    i.e., 
    \[
	p_{0,t}(x,y) \leq C \exp\l(C\l(1+\|b\|^2_{L_t^\infty L_x^{d,\infty}}\r)|\alpha|^2 t + {\alpha\cdot (y-x)}\r). 
    \]
    Letting $\alpha= (y-x)\big(2C(1+\|b\|^2_{L_t^\infty L_x^{d,\infty}}) t \big)^{-1}$, we then obtain the upper bound estimate. 
		
    \medskip 
    \paragraph{\bf $(ii)$ Lower bound:} 
    Regarding the lower bound, following \cite{fabes1986new}, we set
		\[
		\mu(x) :=(2\pi)^{-\frac{d}{2}} \exp\left(-|x|^2/2\right)
		\]
		and 
		\begin{equation*}
			G(t,y) := \int_{\mathbb{R}^d} \log p_{1-t, 1}(x,y)\,  \mu(x) \, \mathrm{d}x, \quad t\in (0,1] \mbox{ and } y\in \mR^d. 
		\end{equation*}
		By a straightforward computation with the integration-by-parts formula, we obtain
		\begin{equation*}
			\begin{aligned}
				\p_t G(t,y)=&\int_{\mathbb{R}^d} \left[\Delta_x p_{1-t, 1}(x,y)+b(1-t,x)\cdot\nabla_x p_{1-t, 1}(x,y)\right] \frac{\mu(x)}{p_{1-t, 1}(x,y)} \, \mathrm{d}x\\
				=&\int_{\mathbb{R}^d} \left[ x\cdot \nabla_x \log p_{1-t, 1}(x,y)+ |\nabla_x \log p_{1-t, 1}(x,y) |^2 \right]\mu(x)\mathrm{d} x\\
				&+\int_{\mathbb{R}^d}  b(1-t,x)\cdot \left[\nabla \mu(x) \log p_{1-t, 1}(x,y) \right] \mathrm{d}x \\
				\geq& -C \|\nabla_x \log p_{1-t, 1}(\cdot,y) \|_{L^2(\mu)}+ \|\nabla_x \log p_{1-t, 1}(\cdot,y) \|_{L^2(\mu)}^2\\
				&-C \|b\|_{L_t^\infty L_x^{d,\infty}} \|\nabla \mu \log p_{1-t, 1}(\cdot,y)\|_{L_x^{d',1}} \\
				\geq& -C- C  \left\||\cdot| \log p_{1-t, 1}(\cdot,y) \, \mu \right\|_{L_x^{d',1}} + \frac{3}{4} \|\nabla_x \log p_{1-t, 1}(\cdot,y) \|_{L^2(\mu)}^2. 
			\end{aligned}
		\end{equation*}
	In the case where $d=2$, we have  
    by \eqref{Eq:Lady}, 
		\begin{align*}
			\| |\cdot|  f\mu \|_{L_x^{2,1}}	\leq& C \||\cdot| \mu^{\frac{1}{3}}\|_{L_x^{4,2}}  \|f \mu^{\frac{2}{3}}\|_{L_x^{4,2}} \\
			{\leq} &C \|f\mu^{\frac{2}{3}}\|_{L_x^2}^{\frac{1}{2}} \|\nabla (f\mu^{\frac{2}{3}})\|_{L_x^2}^{\frac{1}{2}}\\
			\leq &  C_\eps\|f\mu^{\frac{2}{3}}\|_{L_x^2}+ \eps \|f |\cdot| \mu^{\frac{2}{3}}\|_{L_x^2}+\eps \|\nabla f \mu^{\frac{2}{3}}\|_{L_x^2} \\
			\leq &C_{\eps} \|f\|_{L^2(\mu)}+\eps \|\nabla f\|_{L^2(\mu)}^2; 
		\end{align*}
		While for $d\geq 3$, H\"older's inequality \eqref{Eq:Holder} yields 
		\begin{align*}
			\| |\cdot|  f\mu \|_{L^{d',1}_x} \leq C \||\cdot| \sqrt{\mu}\|_{L_x^{\frac{2d}{d-2}, 2}} \|f\sqrt{\mu}\|_{L_x^2}\leq C \|f\|_{L^2(\mu)}. 
		\end{align*}
		Thus, we conclude that for any $d\geq 2$, 
		\begin{equation}
			\p_tG(t,y)\geq \frac{1}{2} \|\nabla_x \log p_{1-t, 1}(\cdot,y) \|_{L^2(\mu)}^2-C \|\log p_{1-t, 1}(\cdot,y) \|_{L^2(\mu)}-C. 
		\end{equation}
		Then the proof for the lower bound can be argued in an analogous way as in \cite{fabes1986new} and \cite{qian2019parabolic}, so 
        the details are omitted here.  
		
\medskip 
  \paragraph{\bf{$(iii)$ H\"older regularity estimate}.} 
We note that the Aronson-type estimate  implies Nash’s continuity theorem (see, e.g., \cite{nash1958continuity}): suppose that $u$ is a solution to  equation 
$$
\p_t u= A^*_t u\,\  (\mbox{or}\ \  \p_t u= A_tu)
$$ 
in $Q_{R}(\tau, \xi)$, then there exist two universal constants $\alpha\in (0,1)$ and $C>0$ such that 
		\begin{equation}\label{eq:nash}
			\left|u\left(t, y\right)-u\left(t', y'\right)\right| \leq C\left(\frac{\left|t-t'\right|^{1 / 2}+\left|y-y'\right|}{R}\right)^\alpha \l( \sup_{Q_{R}(\tau, \xi)}u-\inf_{Q_{R}(\tau, \xi)} u \r) 
		\end{equation}
		for all $(t,y), (t', y')\in Q_{R/2}(\tau, \xi)$. 
		
		Below we set that for any $0\leq s<t<\infty$, 
		\[
		    R:=\sqrt{t-s}~~\mbox{ and }~~ \tau:=t+\frac{R^2}{8}=s+\frac{9R^2}{8}. 
		\] 
		\begin{figure}[ht]
			\centering
			\includegraphics[width=4in]{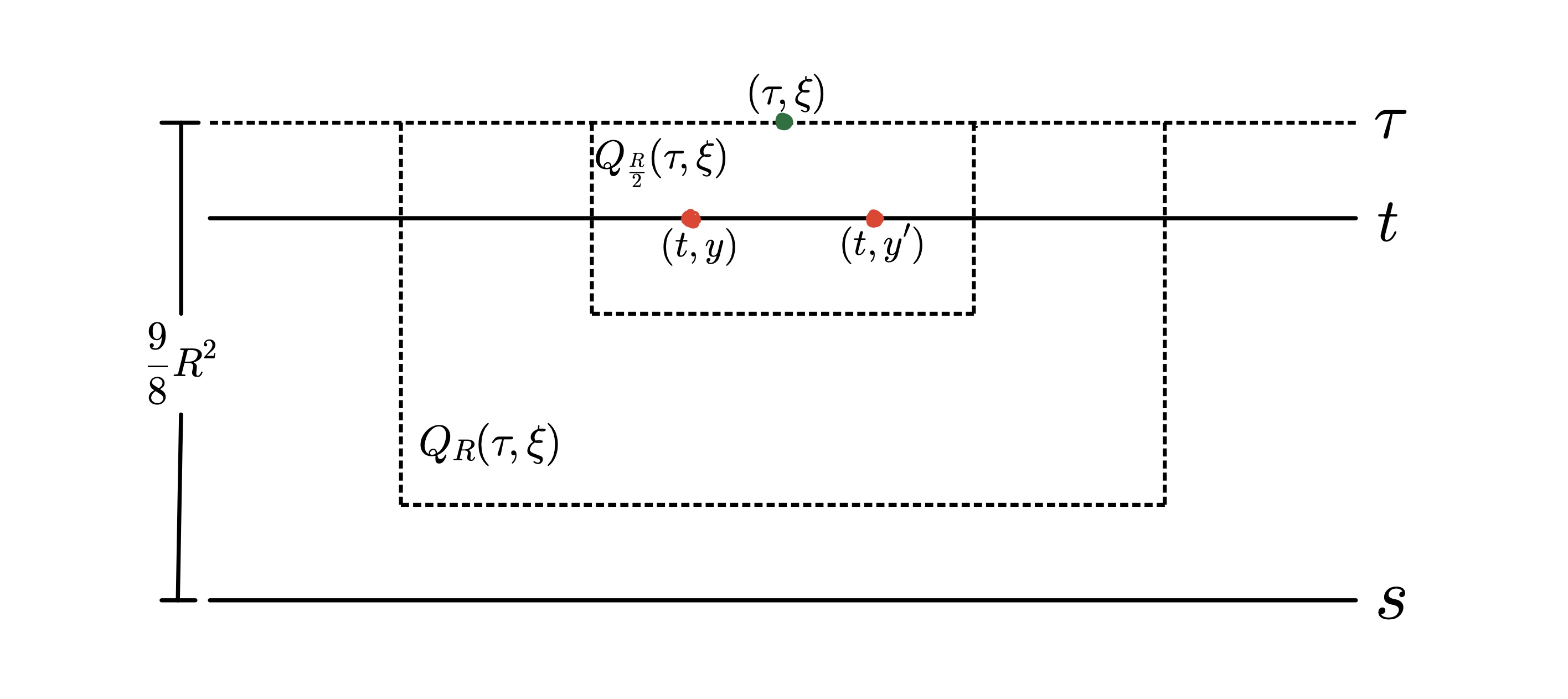}
			\caption{Cube $Q_R(\tau, \xi)$}\label{Fig-cube}
		\end{figure}  
        
        We consider two cases 
        where $|y-y'|\leq R/2$ 
        and $|y-y'| > R/2$, 
        respectively. 
        
		(iii.a) Suppose that $y,y'\in \mR^d$ satisfying $|y-y'|\leq R/2$. 
  Let $\xi=\frac{y+y'}{2}$.  Applying \eqref{eq:nash} to $u(t, y)= p_{s,t}(x,y)$, we obtain that 
		\begin{equation*}
			\begin{aligned}
				\l|p_{s,t}(x,y)-p_{s,t}(x,y')\r| \leq& C  \l( \frac{|y-y'|}{R}\r)^\alpha \sup_{\substack{z\in B_{R}(\xi); \\r\in (\tau-R^2, \tau)}}  p_{s,r}(x, z) \\
				\leq& C  |y-y'|^\alpha (t-s)^{-\frac{d+\alpha}{2}} \sup_{\substack{z\in B_{R}(\xi);\\r\in (s+R^2/8, s+(9R^2/8))}} \exp\l( \frac{-|x-z|^2}{CR^2}\r). 
			\end{aligned}
		\end{equation*}
		Since   
		\[
		    \sup_{z\in B_{R}(\xi)}\exp\l( \frac{-|x-z|^2}{CR^2}\r)\leq 1 \leq C \exp\l( \frac{-|x-y|^2}{CR^2}\r) ~\mbox{ when }~x\in B_{2R}(\xi) 
		\] 
		and 
		\[
		    \sup_{z\in B_{R}(\xi)} \exp\l( \frac{-|x-z|^2}{CR^2}\r)\asymp \exp\l( \frac{-|x-y|^2}{C'R^2}\r)~\mbox{ when }~x\in B^c_{2R}(\bar y), 
		\]
		we get 
		\begin{equation}\label{eq:leq-R}
			\l|p_{s,t}(x,y)-p_{s,t}(x,y')\r| \leq C  |y-y'|^\alpha (t-s)^{-\frac{\alpha}{2}} h\l(C(t-s), x-y\r). 
		\end{equation}
		
		(iii.b) If $|y-y'|>R/2$,  
        then we have 
		\begin{equation}\label{eq:geq-R}
			\begin{aligned}
				 \l|p_{s,t}(x,y)-p_{s,t}(x,y')\r| 
                 \leq&  \l|p_{s,t}(x,y)\r|+\l|p_{s,t}(x,y')\r|\\
				\leq & h\l(C(t-s), x-y\r)+h\l(C(t-s), x-y'\r). 
			\end{aligned}
		\end{equation}
		Combining \eqref{eq:leq-R} and \eqref{eq:geq-R} together we obtain the desired conclusion. 
	\end{proof}

\subsection{H\"older regularity estimate}
    Now, let us consider the  Kolmogorov equation 
    \begin{equation}\label{Eq:BKE1}
    	\p_t u = A_t u + c u  + f, \quad u(0)=0
    \end{equation}
    or 
    \begin{equation}\label{Eq:BKE2}
    	\p_s v + A_s v + cv = f, \quad v(T)=0. 
    \end{equation}  
    The main result of this subsection 
    is Proposition \ref{Prop:Holder} below 
    concerning the H\"older regularity estimate of solutions to the Kolmogorov equation. 

    \begin{definition}
        Let $I$ be an open interval in $\mR_+$ and $D$ a domain of $\mR^d$. Set $Q:=I\times D$. We say that $u\in L^\infty_t L^2_x(Q)\cap L^2_t H^1_x(Q)$ is a subsolution (resp. supersolution) of
        \begin{equation}\label{eq:pde}
            \p_t u = A_t u + c u  + f 
        \end{equation}
        in $Q$, if for any $\varphi \in C_c^\infty(Q)$ with $\varphi\geq 0$,  
        \begin{equation*} 
          \int_{Q} [-u \p_t \varphi +  \nabla u\cdot\nabla \varphi + (b\cdot\nabla u) \varphi  - cu \varphi] \leq(\text{resp.}\ \geq ) \int_{Q} f \varphi. 
        \end{equation*} 
        If $u\in C_tL^2_x\cap L^2_t H^1_x$ is a solution to \eqref{eq:pde} on $(0,T)\times \mR^d$ and $u(0)=0$, then we say that $u$ is a solution to \eqref{Eq:BKE1}. 
    \end{definition}

    We first have the following energy estimate, the proof is standard and thus is omitted.  
    
    \const{\Cenergy}
    \begin{proposition}[Energy estimate]\label{Prop:energy}
        Assume that $f=g+\div F$ with $g\in L^1_t L^2_x$ and $F\in L^2(\mR^d; \mR^d)$. Let $u\in C_tL^2_x\cap L^2_t H^1_x$ be the solution to \eqref{Eq:BKE1}. 
     Then 
        \begin{equation*}
            \|u\|_{L^\infty_t L_x^2} + \|\nabla u\|_{L^2_{t,x}} \leq C \l(\|g\|_{ L^1_t L^2_x}+\|F\|_{L^2_{t,x}} \r) , 
        \end{equation*}
       where $C$ only depends on $\|c\|_{L^\infty}$ and $T$. 
       When $d=2$ and $c\equiv 0$,   
       one has 
       \begin{equation*}
           \|u\|_{L^\infty_t L_x^2}^2 + \|\nabla u\|_{L^2_{t,x}}^2 \leq C_{\Cenergy} \|f\|_{L^{\frac{4}{3}}_{t,x}}^2, 
       \end{equation*}
       where $C_{\Cenergy}$ is independent of $T$. 
    \end{proposition}

    The following fact will be used frequently in this paper. 
	
    \begin{lemma}\label{lem:heat}
	 Let $h$ be the heat kernel given by \eqref{Eq:H}. Then, for any  $l, \alpha\geq 1$, it holds that  
	\begin{equation}\label{Eq:h-Lr}
		\l\| h(t,\cdot) \r\|_{L_x^{l,\alpha}}\leq C t^{-\frac{d}{2}(1-\frac{1}{l})}. 
	\end{equation}
        As a consequence, for any $(p,q)\in \sI$, $r \geq p$ and $\beta\geq 1$, 
	\begin{equation}\label{Eq:cov-fh}
         \sup_{t\in [0,T]}\l\| \int_0^t f(t-s)*h(s) \d s \r\|_{L^{r,\beta}_x}\leq C \|f\|_{L^q_t L^p_x}. 
	\end{equation}
    \end{lemma}
    \begin{proof}
    Since \( h(t, x) = t^{-\frac{d}{2}} h(1, x / \sqrt{t}) \), one has  
    \[ \|h(t, \cdot)\|_{L_x^{l, \alpha}} = t^{-\frac{d}{2}(1 - \frac{1}{l})} \|h(1, \cdot)\|_{L_x^{l, \alpha}}, 
    \]
    from which \eqref{Eq:h-Lr} can be derived. 
     
    Then, an application of  \Cref{Prop:Lorentz} (iii) and \eqref{Eq:h-Lr} yields 
    \begin{equation*}
	\begin{aligned}
		\sup_{t\in [0,T]}\l\| \int_0^t f(t-s)*h(s) \d s \r\|_{L_x^{r,\beta}} \leq C \sup_{t\in [0,T]} \int_0^{t} s^{-\frac{d}{2}(\frac{1}{p}-\frac{1}{r})} \|f(t-s, \cdot)\|_{L^p_x}\, \d s \leq C \|f\|_{L^q_t L_x^p}. 
	\end{aligned}
    \end{equation*}
    where we also used the fact that $-\frac{d}{2}(\frac{1}{p}-\frac{1}{r}) (1-\frac{1}{q})^{-1}>-1$ in the last step, due to the conditions on $p$, $q$ and $r$.  
    Thus, \eqref{Eq:cov-fh} is proved. 
    \end{proof}

	\begin{proposition} [H\"older regularity estimate] \label{Prop:Holder}
		Let $d\geq 2$ and $(p,q)\in \sI$.  Then, for any $f\in L^q_tL^p_x$, the Kolmogorov equation \eqref{Eq:BKE1} admits a weak solution $u$ such that 
		\begin{equation}\label{Eq:Holder1}
			\sup_{t\in [0,T]}\|u(t)\|_{C^{\alpha}_{x}} \leq C \|f\|_{L^q_t L_x^p}, 
		\end{equation}
		where $\alpha$ depends only on $d, p, q$ and $\|b\|_{L^\infty_t L_x^{d,\infty}}$, and $C$ depends on $d, \alpha, p, q, T$, $\|b\|_{L^\infty_t L_x^{d,\infty}}$ and $\|c\|_{L^\infty_{t,x}}$. 
	\end{proposition}
	
	\begin{proof}[Proof]
		By standard approximation arguments, it suffices to consider the 
        smooth case where  $b,c,f\in C^\infty_b(\mR^d)$
        and to prove that 
        \[
            \|v(t)\|_{C^{\alpha}_{x}} \leq C \|f\|_{L^q_t L_x^p}, \quad t\in [0,T], 
        \]
        for $v$ satisfying \eqref{Eq:BKE2}. 

For this purpose, by the Feynman-Kac formula, we derive 
		\begin{equation*}
			v(s,x)= \mE_{s,x} \int_s^T f(t, X_{t}) \exp\l(-{\int_s^t c(r, X_{r}) \d r}\r)\d t, 
		\end{equation*}
        where $(\mP_{s,x}, X_t)$ is the Markov process corresponding to $A_t$. 
        Then, the Aronson-type estimate \eqref{Eq:heat} yields 
	    \[
	        |v(s,x)| \leq C\e^{\|c\|_{L^\infty}T} \int_s^T\!\!\!\int_{\mR^d}  |f(t, y)| \, h(C(t-s), x-y)~\d y\, \d t. 
	    \]
	   Let $l=p/(p-1)$ in \eqref{Eq:h-Lr}. 
        Since $d/p+2/q<2$, we get 
	    \begin{equation}\label{eq:vbdd}
	    	\|v(s,\cdot)\|_{L^{\infty}_x} \leq C \int_0^{T-s} t^{-\frac{d}{2p}} \|f(s+t, \cdot)\|_{L^p}\, \d t \leq C \|f\|_{L^q_t L_x^p}. 
	    \end{equation}
	    Similarly, one can also derive that  
	    \begin{equation}\label{eq:v-lplq}
	        \|v\|_{L^q_t L^p_x}\leq C \|f\|_{L^q_tL^p_x}. 
	    \end{equation} 
		Now, let $g:=cv+f$. Since 
		$$
		\p_s v+ A_s v+ g=0, \quad v(T)=0, 
		$$
		  it holds that  
		\begin{equation*}
			v(s, x) = \mE_{s,x} \int_s^T g(t, X_{t}) \d t =  \int_s^T\!\!\!\int_{\mR^d}p_{s,t}(x,y) g(t, y)\, \d y \d t. 
		\end{equation*}
		Let $R=|x-x'|\ll1$.  By \eqref{Eq:pHolder2} and H\"older's inequality, there exits $\alpha\in (0,1)$ such that 
		\begin{equation*}
			\begin{aligned}
				|v(s,x)-v(s,x')|\leq& \int_s^T|g(t,y)|\l|p_{s,t}(x,y)-p_{s,t}(x',y)\r|\d t\\
                    \leq& C \int_s^{s+R^2} |g(t, y)| \l(p_{s,t}(x,y)+p_{s,t}(x',y)\r) \d y \d t\\
				&+C R^\alpha \int_{s+R^2}^{T} (t-s)^{-\frac{\alpha}{2}} \d t \int_{\mR^d} |g(t, y)| (t-s)^{-\frac{d}{2}} \exp\l( \frac{-|x-y|^2}{C(t-s)}\r)  \d y\\
				\leq&C  \int_0^{R^2} t^{-\frac{d}{2p}} \|g(s+t, \cdot)\|_{L^p_x} \d t +C R^\alpha \int_{R^2}^{T} t^{-\frac{\alpha}{2}-\frac{d}{2p}}  \|g(s+t, \cdot)\|_{L^p_x} \d t\\
				\leq& C R^{2-\frac{d}{p}-\frac{2}{q}} \|g\|_{L^q (s,T; L^p_x)} 
				+ CR^\alpha \|g\|_{L^q (s,T; L_x^p)}  \l(\int_{R^2}^T t^{-\frac{\alpha q'}{2}-\frac{dq'}{2p}} \d t \r)^{\frac{1}{q'}}. 
			\end{aligned}
		\end{equation*}
		Re-selecting the parameter $\alpha \in (0, 2-d/p-2/q)$ and combining \eqref{eq:v-lplq} with the above estimate  we thus obtain 
		\begin{equation*}
		\begin{aligned}
		    |v(s,x)-v(s,x')| \leq& C |x-x'|^\alpha \|g\|_{L^q(0,T;L_x^p)} \\
		    \leq& C |x-x'|^\alpha \l(\|v\|_{L^q_tL_x^p}+\|f\|_{L^q_tL_x^p}\r) \\
		    \leq& C |x-x'|^\alpha \|f\|_{L^q_tL_x^p}. 
		\end{aligned}
		\end{equation*}
        This together with \eqref{eq:vbdd} prove the desired assertion. 
	\end{proof}

\subsection{Sobolev regularity estimate} 
\label{Subsec-Sob-Esti}

The main result of this subsection is the 
following second order Sobolev regularity estimate for backward Kolmogorov equations, 
which is crucial in the proof of Theorem \ref{Thm:SDE}. 

    \begin{theorem} [Second order Sobolev regularity estimate] \label{Thm:LqLp}
    Let $d\geq 2$. Assume that $b$ satisfies \eqref{Eq:Ab}. Then, there exists $(\mathfrak{p},\mathfrak{q})\in \sI$, which only depends on $d$ and $\|b\|_{L^\infty_tL^{d,\infty}_x}$, such that for any $c\in L^\infty_t (L^1_x \cap L^\infty_x)$ and $f\in L^\infty_t (L^1_x\cap L^d_x)$, the weak solution $u$ to \eqref{Eq:BKE1} satisfies 
    \begin{equation}\label{Eq:D2u}
    \|u\|_{L^{\mathfrak{q}}_t W^{2, \mathfrak{p}}_x} + \|\p_t u\|_{L^{\mathfrak{q}}_t L_x^{\mathfrak{p}}} \leq C \|f\|_{L^\infty_t(L^1_x\cap L^d_x)}, 
    \end{equation}
    where $C$ only depends on $d, \mathfrak{p}, \mathfrak{q}, T$, $\|b\|_{L^\infty_t L_x^{d,\infty}}$ and $\|c\|_{L^\infty_t (L^1_x\cap L^\infty_x)}$. 
    \end{theorem}
    
    In order to prove Theorem \ref{Thm:LqLp}, 
    let us first recall the following $L^q$-$L^p$ estimate for the heat equation proved by Krylov \cite{krylov2001heat}.

	\begin{proposition} [\cite{krylov2001heat}] \label{Prop:lplq}
		Let $p,q\in (1,\infty)$, $\alpha\in \mR$. Assume that $f\in  L^q_t H^{\alpha, p}_x$. 
  Then, the  heat equation 
  \begin{equation}\label{Eq:HE}
			\p_t u-\Delta u= f \ \mbox{ in } (0,T)\times \mR^d, \quad u(0)=0, 
		\end{equation} 
  has a unique solution in $L^q_t H^{2+\alpha, p}_x$, 
which satisfies the estimate
		\begin{equation}
			\|\p_t u\|_{L^q_t H^{\alpha, p}_x}+\|u\|_{L^q_t H^{2+\alpha,p}_x}\leq C \|f\|_{L^q_t H^{\alpha,p}_x}, 
		\end{equation}
		where $C=C(d, p, q, T)$. 
	\end{proposition}

	The next result gives a parabolic version of Sobolev's embedding (see, e.g., \cite{rockner2023weak}). 
	\begin{proposition} [Parabolic-type Sobolev embedding] \label{Prop:Inter2}
		Let $p, q\in (1,\infty)$, $r\in[p,\infty)$ and $s\in [q,\infty)$. Assume $\p_t u\in L^q_t L_x^p$, $u\in L^q_t W^{2,p}_x$.  If $1<2/q+d/p=2/s+d/r+1$, then 
		\begin{equation}\label{Eq:PSob1}
			\|\nabla u\|_{L^s_t L_x^r}\leq C \l( \|\p_{t} u\|_{L^{q}_t L_x^{p}} + \|\nabla^2 u\|_{L^{q}_t L_x^{p}}\r), 
		\end{equation}
		where $C$ only depends on $d, p, q, r, s$.
	\end{proposition}

 \medskip
	\paragraph{\bf The case where $d\geq 3$:} 

 The proof is essentially presented in \cite{rockner2023weak}, see \Cref{Fig:Dgeq3} for the strategy of the proof. Hence we omit the details here, but remark that the strategy presented in \Cref{Fig:Dgeq3} can be implemented by selecting the parameters as 
\[
    r>\frac{d}{1/\delta-2}, \quad q > \frac{2}{1/\delta-d/r-2}
\]
such that
\[
    \frac{d}{\mathfrak{p}}+\frac{2}{\mathfrak{q}}=1+ 2\delta +{\delta}\l(\frac{d}{r}+\frac{2}{q} \r)<2. 
\]

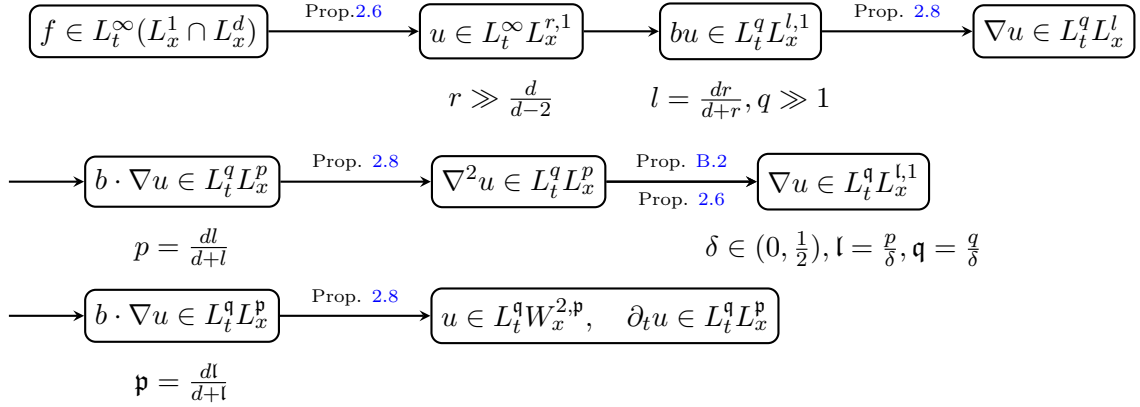
\begin{figure}[ht]
    \centering
    \begin{tikzpicture}[myStyle] 

        \node (f) {$ f \in L_t^\infty (L_x^1 \cap L_x^d) $};
        \node (u)  [right=of f] {$ u \in L_t^\infty L_x^{r,1}$};
        \node [below=of u, draw=none, yshift=\vdist/1.1] {$r\gg \frac{d}{d-2}$};
        \node (bu) [right=1cm  of u]  {$ bu \in L_t^q L_x^{l,1}$};
        \node [below=of bu, draw=none, yshift=\vdist/1.1] {$l=\frac{dr}{d+r}, q\gg 1$};
        \node (nabu) [right=of bu] {$ \nabla u \in L_t^q L_x^{l} $};

        \node (leftbgrad) [below=of f, xshift=-2cm, draw=none] {};
        \node (bgrad) [right=1cm of leftbgrad] {$ b \cdot \nabla u \in L_t^q L_x^{p} $};
        \node [below=of bgrad, draw=none, yshift=\vdist/1.1] {$p=\frac{dl}{d+l}$};
        \node (nabla2u) [right=of bgrad] {$ \nabla^2 u \in L_t^q L_x^{p} $};
        \node (nablau) [right=of nabla2u] {$ \nabla u \in L_t^{\mathfrak{q}} L_x^{\mathfrak{l},1} $};
        \node [below=of nablau, draw=none, yshift=\vdist/1.1] {$\delta\in (0,\frac{1}{2}), \mathfrak{l}=\frac{p}{\delta}, \mathfrak{q}=\frac{q}{\delta}$};

        \node (leftbgrad2) [below=of leftbgrad, draw=none] {};
        \node (bgrad2) [right=1cm of leftbgrad2] {$ b \cdot \nabla u \in L_t^{\mathfrak{q}} L_x^{\mathfrak{p}}$};
        \node [below=of bgrad2, draw=none, yshift=\vdist/1.1] {$\mathfrak{p}=\frac{d\mathfrak{l}}{d+\mathfrak{l}}$};
        \node (w2p) [right=of bgrad2] {$ u \in L_t^{\mathfrak{q}} W_x^{2,\mathfrak{p}}, \quad \partial_t u \in L_t^{\mathfrak{q}} L_x^{\mathfrak{p}} $};
        
        \draw (f) --node[draw=none,above] {\tiny Prop.\ref{Prop:Holder}} (u);
        \draw (u) -- (bu);
        \draw (bu) -- node[draw=none,above] {\tiny Prop. \ref{Prop:lplq}} (nabu);
        \draw  (leftbgrad) -- (bgrad); 
        \draw (bgrad) -- node[draw=none,above] {\tiny Prop. \ref{Prop:lplq}} (nabla2u);
        \draw (nabla2u) -- node[draw=none,above] {\tiny Prop. \ref{Prop:GNI}} (nablau);
        \draw (nabla2u) -- node[draw=none,below] {\tiny Prop. \ref{Prop:Holder}} (nablau);
        \draw (leftbgrad2) -- (bgrad2);
        \draw (bgrad2) -- node[draw=none,above] {\tiny Prop. \ref{Prop:lplq}} (w2p);
        
    \end{tikzpicture}
    \caption{Proof Strategy when \(d\geq 3\)}\label{Fig:Dgeq3}
\end{figure}

\medskip
	\paragraph{\bf The case where $d=2$:}
	As mentioned before, the 2D case is much more delicate than the high dimensional case $d\geq 3$. 
    One major obstacle is that the above strategy can not establish \( \nabla u \in L^q_t L^l_x \) with \( l > 2=d \). 
    It thus limits the application of 
     \Cref{Prop:lplq} 
     to implement the above strategy. 
	
    To overcome this problem, we 
    introduce the upgradation procedure as shown in \Cref{Fig:2D} of Section \ref{Sec-Intro}, 
    based on a new Meyers-type estimate 
    and the refined Gagliardo-Nirenberg estimate. In order to implement the upgradation strategy, 
let us first recall the famous Gerhing's Lemma. Let 
\[
    \avert f \avert_{L^p(A)} :=  \l( \fint_{A} |f|^p \r)^{\frac{1}{p}}=\l( \frac{1}{|A|}\int_{A} |f|^p \r)^{\frac{1}{p}}. 
\]
  	
  	\const{\Cgerhing}
    \begin{lemma}[Gerhing's Lemma, Proposition 1.3 of \cite{giaquinta1982partial}]\label{Lem:Gerhing}
	    Let $\theta\in [0,1)$. 
		Assume that  
		\begin{equation*}
			\avert v \avert_{L^p({Q_{R/2}(z))}}  \leq  \theta \avert v \avert_{L^p(Q_R(z))} + C_{\Cgerhing} \avert v \avert_{L^1(Q_R(z))} + C_{\Cgerhing}  \avert g \avert_{L^p(Q_{R}(z))}
		\end{equation*}
		holds for any $z=(t,x)\in \mR^{d+1}$ and $R>0$, 
        where $Q_{r}(z)=(t-R^2,t)\times B_r(x)$.  Then, 
    there exists $q= q(d,p,\theta, C_{\Cgerhing})>p$ such that for any $z\in \mR^{d+1}$ and $R>0$, 
		\begin{equation}\label{Eq:Gerhing}
			\avert v \avert_{L^{q}(Q_{R/2}(z))}\leq C \avert v \avert_{L^p({Q_{R}(z))}}+ C \avert g \avert_{L^{q}(Q_{R}(z))}, 
		\end{equation}
		where $C$ only depends on $d, p,\theta, C_{\Cgerhing}$ and $q$. 
    \end{lemma}

    \begin{lemma}
        Let $d=2$. There is a constant $C$ such that for any $R>0$ and $u\in W^{1,2}(B_R)$, 
        \begin{equation}\label{eq:Lady-Ineq}
            \|u\|_{L^4(B_R)}^4 \leq C \|u\|_{L^2(B_R)}^2 \int_{B_R} \l( R^{-2} |u|^2+|\nabla u|^2\r). 
        \end{equation}
    \end{lemma}
    \begin{proof}
        Let $E$ be the standard extension operator in the theory of Sobolev space 
        (see \cite{evans2010partial}), which continuously maps $L^2(B_1)$ and $W^{1,2}(B_1)$ to $L^2(\mR^2)$ and $W^{1,2}(\mR^2)$, respectively. By the classical Lady\v{z}henskaya inequality on $\mR^2$, we have  
	\begin{align*}
		\|u\|_{L^4(B_1)}^4 \leq \|E u\|_{L^4(\mR^2)}^4 \leq C \|\nabla E u\|_{L^2(\mR^2)}^{2}\|E u\|_{L^2(\mR^2)}^{2} \leq C \|u\|_{W^{1,2}(B_1)}^{2} \|u\|_{L^2(B_1)}^{2}.
	\end{align*}
       The desired inequality follows by the above estimate and scaling. 
    \end{proof}

    \const{\CRHolder}
    The following lemma is crucial for the proof of \Cref{Thm:LqLp}.
    
    \begin{lemma}[Reverse H\"older estimate]\label{Lem:RHolder}
        Let $d=2$, $\rho>0$, and $u$ be a weak solution to \eqref{eq:pde} in $Q_\rho$ with $c=0$. Then, there exists a constant $C_{\CRHolder}$ only depending on $\|b\|_{L^\infty_t L^{2,\infty}_x(Q_\rho)}$ such that for any $R\in (0,\rho]$ and $\eps\in (0,1)$,   
	\begin{equation*}
		\begin{aligned}
			\iint_{Q_{\frac{R}{2}}} |\nabla  u|^2  \leq & \eps\l( 1+ \|u\|_{L^\infty_tL^2_x(Q_\rho)}^2 \r)  \iint_{Q_R} |\nabla  u|^2 \\
            & + C_{\CRHolder}\eps^{-\frac{3}{2}} R^{-2} \l(\iint_{Q_R} |\nabla u|^{\frac{4}{3}}\r)^{\frac{3}{2}}+ C_{\CRHolder}\eps^{-\frac{1}{3}}\iint_{Q_R} |f|^{\frac{4}{3}}, 
		\end{aligned}
	\end{equation*}
	for any $u\in L^\infty_tL^2_x(Q_\rho)$ and $f\in L^{\frac{4}{3}}_{t,x}(Q_R)$. 
    \end{lemma}
	\begin{proof} 
 	We use the  cut-off function $\varphi\in C_c^\infty(B_1)$ satisfying $\varphi\in [0,1]$ and $\varphi\big|_{B_{\frac{1}{2}}}=1$,  and 
 let $\phi\in C_c^\infty((-1,1))$  satisfy $\phi\in [0,1]$ and $\phi\big|_{[-\frac{1}{2}, \frac{1}{2}]}=1$.   
		Let 
		\[
		    \varphi_R(x):= \varphi(\frac{x}{R}), \quad \phi_R(t):=\phi\l(\frac{t}{R^2}\r), 
		\]
		 and 
		 \[
		 \underline{u}_R:=\l( {\int_{B_R} u\varphi_R^{10}}\r) \l( {\int_{B_R} \varphi_R^{10}} \r)^{-1}, \quad \bar{u}_R:= u-\underline{u}_R. 
		 \]
	For simplicity, 
    the subscript $R$ is omitted below. Multiplying both sides of \eqref{Eq:BKE1} by $\bar {u} \varphi^{10}\phi^2$ and using the integration-by-parts formula, we have 
		\begin{equation*}
		\begin{aligned}
			\iint_{Q_R} \p_t u~\bar{u}\varphi^{10}\phi^2 =& \iint_{Q_R} \p_t (u-\underline{u}) (u-\underline{u}) \varphi^{10}\phi^2 \\
			=&  \frac{1}{2}\iint_{Q_R} \p_t (\bar{u}^2\varphi^{10}\phi^2)-\iint_{Q_R}\bar u^2\varphi^{10}\phi\phi', 
		\end{aligned}
		\end{equation*}
		and 
		\begin{equation*}
			-\iint_{Q_R} \Delta u~ \bar{u}\varphi^{10}\phi^2 =\iint_{Q_R} |\nabla u |^2 \varphi^{10}\phi^2+10 \iint_{Q_R} \bar{u} \varphi^9\phi^2~\nabla u\cdot\nabla \varphi. 
		\end{equation*}
	Since $b$ is divergence free, we get 
		\begin{equation*}
			\begin{aligned}
				\iint_{Q_R} b\cdot \nabla u~ \bar{u}\varphi^{10}\phi^2 =& \frac{1}{2}\iint_{Q_R} \varphi^{10}\phi^2~  b\cdot \nabla (\bar{u}^2)=-5 \iint_{Q_R} \bar{u}^2\varphi^9\phi^2 ~b\cdot \nabla \varphi. 
			\end{aligned}
		\end{equation*}
		Thus, 
		\begin{equation}\label{eq:IJ}
			\begin{aligned}
				&\underbrace{\sup_{t\in (-R^2, 0)} \int_{B_R} \bar u^2(t) \varphi^{10} \phi^2(t)}_{=:J_1}+ \underbrace{\iint_{Q_R} |\nabla u|^2 \varphi^{10} \phi^2}_{=:J_2} \\
				\leq &  \frac{C}{R^2}  \iint_{Q_R} \bar u^2 \varphi^{10} \phi+ \frac{C}{R} \iint_{Q_R} |\nabla u|\, \bar u \varphi^9 \phi^2 + \frac{C}{R} \iint_{Q_R}  |b|\,\bar u^2\varphi^9\phi^2 +C  \iint_{Q_R} f\bar u \varphi^{10} \phi^2\\ 
			    \leq & \frac{1}{2} \iint_{Q_R} |\nabla u|^2 \varphi^{10} \phi^2 +C \delta^{-\frac{1}{3}} \iint_{Q_R} |f|^{\frac{4}{3}}+\frac{\delta}{2} \iint_{Q_R} |\bar u|^4\quad (\forall \delta>0)\\
                    & + \underbrace{\frac{C}{R^2} \iint_{Q_R} \bar u^2\varphi^7\phi}_{=:I_1} + \underbrace{\frac{C}{R} \iint_{Q_R}  |b|\,\bar u^2\varphi^9\phi^2}_{=:I_2} . 
			\end{aligned}
		\end{equation}
		
		For $I_1$, by H\"older's inequality and the Poincar\'e-type inequality
  \eqref{Eq:Poincare2},   
  for any $\delta>0$, 
		\begin{equation}\label{eq:I1}
			\begin{aligned}
				I_1=&\frac{C}{R^2} \iint_{Q_R} \bar u^2\varphi^7\phi 
                \leq \frac{C}{R^2} \iint_{Q_R} (\bar u^{\frac{2}{3}}\varphi^{\frac{10}{3}}\phi^{\frac{2}{3}}) \bar u^{\frac{4}{3}}\\
				\leq & \frac{C}{R^2} \int_{-R^2}^0 \l( \int_{B_R} \bar u^2(t) \varphi^{10} \phi^2(t)\r)^{\frac{1}{3}} \l(\int_{B_R} \bar u^2(t) \r)^{\frac{2}{3}}\, \d t\\
				\leq&\delta \sup_{t\in (-R^2,0)} \int_{B_R} \bar u^2(t)\varphi^{10}\phi^2(t) +  \delta^{-\frac{1}{2}} \l[\int_{-R^2}^0 R^{-2} \l(\int_{B_R} \bar u^2(t) \r)^{\frac{2}{3}}\, \d t \r]^{\frac{3}{2}}\\
			{\leq} & \delta J_1 + C\delta^{-\frac{1}{2}} R^{-2} \l(\iint_{Q_R} |\nabla u|^{\frac{4}{3}}\r)^{\frac{3}{2}}. 
			\end{aligned}
		\end{equation}
		Regarding $I_2$, by the H\"older inequality \eqref{Eq:Holder}, the 
        Lady\v{z}enskaja-type inequality for Lorentz spaces \eqref{Eq:Lady} 
        and \eqref{eq:I1}, we obtain that for any $\eps>0$, 
		\begin{equation}\label{eq:I2}
			\begin{aligned}
				I_2  
                {\leq}&  \frac{C}{R} \|b\|_{L^\infty_t L^{2, \infty}_x (Q_R)}  \|\bar u \varphi^{\frac{9}{2}}\phi\|_{L^2_tL^{4,2}_x(Q_R)}^2\\
				\leq& \frac{C}{R} \|b\|_{L^\infty_t L^{2, \infty}_x (Q_R)}  \int_{-R^2}^0 \|\bar u(t) \varphi^{\frac{9}{2}}\phi(t)\|_{L^{4,2}_x(B_R)}^2 \,\d t \\
				{\leq}& \frac{C}{R}  \|b\|_{L^\infty_t L^{2, \infty}_x (Q_R)}  \int_{-R^2}^0  \|\bar u(t) \varphi^{\frac{9}{2}}\phi(t)\|_{L^2_x} \, \l\| \nabla \l(\bar u(t) \varphi^{\frac{9}{2}}\phi(t)\r)\r\|_{L^2_x} \, \d t\\
				\leq& \frac{C}{\eps R^2} \l( 1+\|b\|_{L^\infty_t L^{2, \infty}_x (Q_R)}^2 \r)  \iint_{Q_R} \bar u^2 \varphi^7 \phi +\frac{\eps}{2} \iint_{Q_R} |\nabla  u|^2\\
				{\leq}&  
                C \eps^{-1} \delta J_1
                + C\eps^{-1}\delta^{-\frac 12} R^{-2} \l(\iint_{Q_R} |\nabla u|^{\frac{4}{3}}\r)^{\frac{3}{2}}+ \frac{\eps}{2} \iint_{Q_R} |\nabla  u|^2. 
			\end{aligned}
		\end{equation}
		Combining \eqref{eq:IJ}-\eqref{eq:I2}, we get 
            \const{\crholder}
		\begin{equation*}
			\begin{aligned}
				(1-C_{\crholder}\delta/\eps)J_1+J_2\leq& \eps \iint_{Q_R} |\nabla  u|^2 + \delta \iint_{Q_R} |\bar u|^4 \\ &+ C_{\crholder}\eps^{-1}{\delta}^{-\frac{1}{2}}  R^{-2} \l(\iint_{Q_R} |\nabla u|^{\frac{4}{3}}\r)^{\frac{3}{2}}+ C_{\crholder}\delta^{-\frac{1}{3}}\iint_{Q_R} |f|^{\frac{4}{3}}, 
			\end{aligned}
		\end{equation*}
            where $C_{\crholder}$ only depends on $\|b\|_{L^\infty_tL^{2,\infty}_x(Q_\rho)}$. 

            \const{\cpoincare}
		Using \eqref{eq:Lady-Ineq} and the Poincar\'e-type  inequality \eqref{Eq:Poincare1}, we have 
		\begin{equation*}
		    \begin{aligned}
				\iint_{Q_R} |\bar u|^4 {\leq} & C \int_{-R^2}^{0}  \d t \int_{B_{R}}|\bar u|^2(t) \l(\int_{B_{R}} R^{-2}|\bar u(t)|^2+|\nabla \bar u(t)|^2\r)\\
				{\leq} & 
                \int_{-{R^2}}^{0}  \d t \int_{B_{{R}}}|\bar u|^2(t) \int_{B_{R}}|\nabla \bar u(t)|^2 \\
				\leq & C \sup _{t \in\left(-{R^2}, 0\right)}\left(\int_{B_{R}}|\bar u|^2(t)\right) \iint_{Q_R}|\nabla u|^2\\
                \leq & C_{\cpoincare} \sup _{t \in\left(-{R^2}, 0\right)}\left(\int_{B_{R}} u^2(t)\right) \iint_{Q_R}|\nabla u|^2, 
			\end{aligned}
		\end{equation*} 
            where $C$ is a constant independent of $\delta$ and $\eps$. Thus, 
            \begin{equation}
                \begin{aligned}
				(1-C_{\crholder}\delta/\eps)J_1+J_2\leq& \eps \iint_{Q_R} |\nabla  u|^2 + \delta C_{\cpoincare} \sup _{t \in\left(-{R^2}, 0\right)}\left(\int_{B_{R}} u^2(t)\right) \iint_{Q_R}|\nabla u|^2\\ &+ C_{\crholder}\eps^{-1}{\delta}^{-\frac{1}{2}}   R^{-2} \l(\iint_{Q_R} |\nabla u|^{\frac{4}{3}}\r)^{\frac{3}{2}}+ C_{\crholder}\delta^{-\frac{1}{3}} \iint_{Q_R} |f|^{\frac{4}{3}}. 
			\end{aligned}
            \end{equation}
            For any $\eps\in (0,1)$, by choosing $\delta= {\eps}/(C_{\crholder}+C_{\cpoincare})$, we obtain the desired estimate. 
	\end{proof}

    \const{\Cgerhing2}
    \begin{lemma} [Meyers-type estimate] \label{Lem:Meyers}
        Let $d=2$. There exists a universal constant $\gamma >2$ only depending on $\|b\|_{L^\infty_tL^{2,\infty}_x}$ such that for any weak solution to \eqref{Eq:BKE1} on $[0,T]\times \mR^d$ with $c=0$, it holds  
        \begin{align}\label{eq:Meyers}
            \|\nabla u\|_{L^\gamma_{t,x}} \leq C \|f\|_{L^{\frac{4}{3}}_{t,x}}^{\frac{1}{3}} \|f\|_{L^{\frac{2\gamma}{3}}_{t,x}}^{\frac{2}{3}}, 
	\end{align} 
        for any $f\in L^{\frac{4}{3}}_{t,x}\cap L^{\frac{2\gamma}{3}}_{t,x}$. Here the constant $C$ only depends on $\gamma$ and $\|b\|_{L^\infty_tL^{2,\infty}_x}$. 
    \end{lemma}
    \begin{proof}
        As before, we assume that $b\in C_b^\infty(\mR^d)$ with $\div b=0$, $f\in C_c^\infty(\mR^d)$.  Then \eqref{Eq:BKE1} admits a classical solution $u$ 
        that can be extended  continuously to $(-\infty, T]\times \mR^d$ by setting $u(t,\cdot)=0$ for all $t\leq 0$. 
        
        Let 
        \[
            \tilde{f}=f/\|f\|_{L^{\frac{4}{3}}_{t,x}}, \quad \tilde{u}=u/\|f\|_{L^{\frac{4}{3}}_{t,x}}.
        \]
        By the linearity of \eqref{Eq:BKE1}, $\tilde{u}$ also satisfies \eqref{Eq:BKE1} with $f$ replaced by $\tilde{f}$. Thanks to \Cref{Lem:RHolder} and \Cref{Prop:energy}, for any $R>0$ and $z\in (-\infty, T)\times \mR^2$, it holds that  
        \begin{equation}
            \begin{aligned}
                \iint_{Q_{\frac{R}{2}}(z)} |\nabla  \tilde{u}|^2  \leq &  \eps\l( 1+ C_{\Cenergy} \r)  \iint_{Q_R(z)} |\nabla  \tilde{u}|^2 \\
                &+ C_{\CRHolder}\eps^{-\frac{3}{2}}  R^{-2} \l(\iint_{Q_R(z)} |\nabla \tilde{u}|^{\frac{4}{3}}\r)^{\frac{3}{2}}+ C_{\CRHolder}\eps^{-\frac{1}{3}} \iint_{Q_R(z)} |\tilde{f}|^{\frac{4}{3}}. 
            \end{aligned}
        \end{equation}
        Letting $\eps :=2^{-10} \l( 1+ C_{\Cenergy} \r)^{-1}$ and 
        \[
            v :=|\nabla \tilde{u}|^{\frac{4}{3}}, \quad g :=|\tilde{f}|^{\frac{8}{9}}, \quad p :=\frac{3}{2}, 
        \]
        we obtain that 
        \[
            \avert v \avert_{L^p(Q_{{R}/{2}})} \leq \frac{1}{2} \avert v \avert_{L^p(Q_R)} + C_{\Cgerhing2}\avert v \avert_{L^1(Q_R)} +  C_{\Cgerhing2} 
            \avert g  \avert_{L^p(Q_R)},  
        \]
        where $C_{\Cgerhing2}$ only depends on $\|b\|_{L^\infty_tL^{2,\infty}_x}$.  In the light of \Cref{Lem:Gerhing}, there exists $\gamma>2$ only depending on $\|b\|_{L^\infty_t L_x^{2,\infty}}$, such that for any $R>0$, 
	\begin{align*}
            \iint_{Q_{R/2}} |\nabla \tilde{u}|^\gamma \leq& C R^{4-2\gamma} \l( \iint_{Q_R}|\nabla \tilde{u}|^2 \r)^{\gamma/2} + C \iint_{Q_R} |\tilde{f}|^{\frac{2\gamma}{3}}\\
            \leq &  C R^{4-2\gamma} + C \iint_{Q_R} |\tilde{f}|^{\frac{2\gamma}{3}}, 
	\end{align*}
        where $C$ only depends on $\gamma$ and $\|b\|_{L^\infty_t L^{2,\infty}_x}$. Letting $R\to \infty$, we obtain \eqref{eq:Meyers}. 
    \end{proof}
    
    We are now ready to prove \Cref{Thm:LqLp} in the case where $d=2$. 
	\begin{proof}[Proof of \Cref{Thm:LqLp}]
    The proof relies on three steps to upgrade the integrability of the gradient of solutions.   
            Again we focus on the smooth case where $b\in C_b^\infty(\mR^d)$ with $\div b=0$, $f\in C_c^\infty(\mR^d)$, and additionally $c=0$. Then \eqref{Eq:BKE1} admits a classical solution $u$ 
            which can be extended continuously to $(-\infty, T]\times \mR^d$ by setting $u(t,\cdot)=0$ for all $t\leq 0$. 

\medskip 
    \paragraph{\bf (i) Base estimates:}
		Let $\cT$ be the solution map for the heat equation \eqref{Eq:HE}, i.e.,  
		\[
		\cT(f)(t,x):= \int_0^t\!\!\!\int_{\mR^2} h(t-s,y) f(s,y)\,  \d y\, \d s, 
		\]
		where $h$ is the heat kernel as in \eqref{Eq:H}. 
        Then 
		\begin{equation}\label{eq:u-T}
			u=\cT f+\cT(b\cdot \nabla u)=\cT f+\nabla\cdot \cT(b u).  
		\end{equation}
		Set 
		\[
		      \norm f := \|f\|_{L^{\infty}_t L_x^1}+\|f\|_{L^{\infty}_t L_x^2}. 
		\]
		For any $\mu\in (1,2)$ and $\nu\in (2,\infty)$,  choosing 
		\begin{equation*}
			\eta=\l( \frac{1}{2\mu}+\frac{1}{2 \nu}+\frac{1}{4}\r)^{-1}\in (1,2), 
		\end{equation*}
		and using the $L^q$-$L^p$ estimate in \Cref{Prop:lplq} and the parabolic-type Sobolev embedding in \Cref{Prop:Inter2}, we have  \begin{equation*}
			 \|\nabla \cT(f)\|_{L^\nu_t L_x^\mu} \leq C \l(\|\p_t \cT(f)\|_{L^\eta_{t,x}}+ \|\nabla^2 \cT(f)\|_{L^\eta_{t,x}}\r) \leq C \|f\|_{L^\eta_{t,x}}\leq C\norm f. 
		\end{equation*}
		Moreover, using \Cref{Thm:Heat} and \Cref{lem:heat} we come to 
		\[
		    \|u\|_{L^{\infty}_t L_x^{r,1}} \leq C \norm f \quad \text{with}\quad  r=\l(\frac{1}{\mu}-\frac{1}{2}\r)^{-1}. 
		\] 
		Thus, 
		\begin{align}\label{eq:bu}
			\| b u \|_{L^\nu_t L_x^\mu } \leq C \|b\|_{L^\infty_t L_x^{2,\infty}} \|u\|_{L^\infty_t L_x^{r,1}} \leq C \norm f. 
		\end{align}
		Therefore, 
        using \eqref{eq:u-T}-\eqref{eq:bu} and \Cref{Prop:lplq} 
        again we obtain 
		\begin{equation}\label{eq:Du2}
			\|\nabla u\|_{L^\nu_t L^\mu_x } \leq C \| \nabla^2 \cT (b u) \|_{L^\nu_t L^\mu_x } + C \|\nabla \cT (f)\|_{L^\nu_t L^\mu_x }\leq  C \norm f. 
		\end{equation}

    \medskip 
     \paragraph{\bf (ii) First upgradation:} 
    We note that the integrability exponent 
    \(\mu\) in the last estimate \eqref{eq:Du2} is strictly less than two, 
    which is insufficient to derive \eqref{Eq:D2u}.  
    The key idea in this step is to 
    make use of the Meyers-type estimate 
    to upgrade the integrability of the gradient of solutions. 
    
     More precisely, let $\gamma>2$ be the integrability exponent as in \Cref{Lem:Meyers} and $\eps \in (0,\gamma^{-1})$ be a small parameter to be determined later. Set 
		\[
		    \theta=1-\eps \gamma \in (0,1), \quad \mu= \frac{1-\eps \gamma}{1/2-(\gamma/4+1/2)\eps} \in (1,2), \quad \nu>\frac{1-\eps\gamma}{(\gamma/4-1/2)\eps}\vee 1. 
		\] 
		Combining \eqref{eq:Du2} and the Meyers-type estimate \eqref{eq:Meyers} together, 
        we can upgrade the integrabilitiy of $\nabla u$ by  
		\[
		    \|\nabla u\|_{L^q_t L^{l,1}_x} \leq  \|\nabla u\|_{L^\nu_t L^\mu_x}^{\theta}\|\nabla u\|_{L^\gamma_{t,x}}^{1-\theta}\leq C \norm f , 
		\]
		where 
		\[
		    \frac{1}{l}= \frac{\theta}{\mu}+\frac{1-\theta}{\gamma}, \quad   \frac{1}{q}= \frac{\theta}{\nu}+\frac{1-\theta}{\gamma}>\eps. 
		\]
		Note that, 
        by our choice of parameters, 
        one has 
		\begin{equation}\label{eq:l-q}
			\frac{1}{l}=\frac{1}{2}-\frac{\gamma\eps}{4}+\frac{\eps}{2}<\frac{1}{2}, \quad \eps< \frac{1}{q}<\frac{\gamma\eps}{4}+\frac{\eps}{2}.
		\end{equation} 
        In particular, 
        the integrability exponent $l$ of the energy is 
        now raised above \(2=d\). Put 
        \[
        p=\frac{2l}{2+l}. 
        \]
	Then
    \begin{equation}\label{eq:bDu1}
			\| b\cdot \nabla u \|_{L^q_t L_x^{p}} 
            \leq C \|b\|_{L^\infty_t L_x^{2,\infty}} \|\nabla u\|_{L^{q}_t L_x^{l,1}} \leq C \norm f. 
		\end{equation}
		Note that 
		\begin{equation}\label{eq:p-q}
            1+\eps>\frac{1}{q}+\frac{1}{p}=\frac{1}{2}+\frac{1}{l}+\frac{1}{q}>1+\eps\l(\frac{3}{2}-\frac{\gamma}{4}\r)> 1\ \  (\mbox{if}\ \gamma<6), 
	\end{equation}
        due to \eqref{eq:l-q}. 
        
    \medskip 
    \paragraph{\bf (iii) Second  upgradation:} 
		We still need to further upgrade the integrability of $\nabla u$.  
  For this purpose, using the $L^q$-$L^p$ estimate in  \Cref{Prop:lplq} again and \eqref{eq:bDu1}, we see that 
		\[
		    \|\nabla^2 u\|_{L^{q}_t L_x^{p}} \leq C \l( \| b\cdot \nabla u \|_{L^{q}_t L_x^{p}}+ \|f\|_{L^{q}_t L_x^{p}}\r)\leq C  \norm f. 
		\]
	 By virtue of the fractional Gagliardo-Nirenberg estimate in \Cref{Prop:GNI},  
     we obtain the improved estimate 
		\[
		    \|\nabla u\|_{L^{\mathfrak{q_i}}_t L_x^{\mathfrak{l}_i}} \leq \|\nabla^2 u \|_{L^q_t L_x^{p}}^{\delta_i} \|u\|_{L^\infty_t C^{\alpha_i}}^{1-\delta_i} \leq C \norm f, 
		\]
		where $0<\alpha_1=\frac{1}{3}\alpha, \alpha_2= \alpha$, $\delta_i=\frac{1-\alpha_i}{2-\alpha_i}$ and $\alpha$ is the constant in  \Cref{Prop:Holder}, and 
		\[
		    \frac{1}{\mathfrak{l}_i}= \frac{\delta_i}{p} , \quad \frac{1}{\mathfrak{q}_i}=\frac{\delta_i}{q}. 
		\]
	Using the interpolation estimate \eqref{Eq:Inter1} again we have 
		\[
		    \|\nabla u\|_{L^{\mathfrak{q}}_t L_x^{\mathfrak{l},1}} \leq C \l( \|\nabla u\|_{L^{\mathfrak{q}_1}_t L_x^{\mathfrak{l}_1}}+ \|\nabla u\|_{L^{\mathfrak{q}_2}_t L_x^{\mathfrak{l}_2}}  \r) \leq C \norm f, 
		\]
		where 
		\[
		    \mathfrak{l}:=\frac{p}{\delta}, \quad \mathfrak{q}:=\frac{q}{\delta}  
            \mbox{\ \ with\ \ }
            0<\delta = \frac{1-\alpha/2}{2-\alpha/2}<\frac{1}{2}. 
		\]
		This further yields 
        the improved integrability estimate 
		\[
		    \|b\cdot \nabla u\|_{L^{\mathfrak{q}}_t L_x^{\mathfrak{p}}} \leq C \|b\|_{L^\infty_t L_x^{2,\infty}} \|\nabla u\|_{L^{\mathfrak{q}}_t L_x^{\mathfrak{l},1}} \leq C \norm f, 
		\]
		where 
		\[
		    \frac{1}{\mathfrak{p}}:= \frac{1}{\mathfrak{l}}+\frac{1}{2}=\frac{\delta}{p}+\frac{1}{2}. 
		\]
		Now choosing $\eps<\frac{1}{\gamma}\wedge \l(\frac{1}{2\delta}-1\r)$ and using \eqref{eq:p-q} we infer 
        that the improved integrability exponent $(\mathfrak{p}, \mathfrak{q})$ 
         exactly lies in \(\sI\), 
        that is, 
		\[
		    \frac{1}{\mathfrak{p}}+\frac{1}{\mathfrak{q}}=\frac{1}{2}+\delta\l(\frac{1}{2}+\frac{1}{l}+\frac{1}{q}\r)<\frac{1}{2}+\delta (1+\eps)<1. 
		\]
	Thus, applying the $L^q$-$L^p$ estimate in \Cref{Prop:lplq} again we obtain the desired estimate \eqref{Eq:D2u} in the case where $c=0$. 
        
    Now, the case where $c\in L^\infty_t(L^1_x\cap L^\infty_x)$ can be estimated easily 
    by applying the second order Sobolev estimate \eqref{Eq:D2u} for $c=0$ and \eqref{Eq:Holder1}: 
        \[
            \|u\|_{L^{\mathfrak{q}}_t W^{2, \mathfrak{p}}_x} + \|\p_t u\|_{L^{\mathfrak{q}}_t L_x^{\mathfrak{p}}}  \leq C\norm{cu+f}\leq C  \|u\|_{L_{t,x}^\infty} \|c\|_{L^\infty_t(L^1_x\cap L^\infty_x)}+ C \norm{f} \leq C \norm{f}, 
        \]
        where $C$ only depends on $d, \mathfrak{p}, \mathfrak{q}, T$, $\|b\|_{L^\infty_t L_x^{d,\infty}}$ and $\|c\|_{L^\infty_t (L^1_x \cap L^\infty_x)}$.
        The proof is complete. 
	\end{proof}

\section{SDEs with endpoint critical drifts}\label{Sec:SDE}

In this section, we aim to establish \Cref{Thm:SDE'} with the endpoint critical drifts. Without loss of generality, 
let us assume \(s=0\) in \eqref{Eq:SDE}.

\subsection{Existence}  \label{Subsec-Exist-SDE}
To begin with, let us start with the generalized It\^o formula. 

\begin{lemma}[Generalized It\^o's formula]\label{Lem:Ito}
    Let $(p,q)\in \sI$. Suppose that $X_t$ is a solution to \eqref{Eq:SDE} with \(s=0\) satisfying the following Krylov-type estimate 
    \[
    \bE  \int_0^T f(X_t) \d t \leq C \|f\|_{L^q_t L_x^p}, 
    \]
    and 
    $u\in L^q_t W^{2,p}_x$ with $\p_t u\in L^q_t L_x^p$.  Then, for any $0\leq t'< t\leq T$, we have \begin{equation}\label{Eq:Ito}
        u(t, X_{t}) - u(t', X_{t'}) = \int_{t'}^t (\p_r + A_r) u(r, X_r)\d r + \sqrt{2} \int_{t'}^t \nabla u(r, X_r) \cdot \d W_r. 
    \end{equation} 
\end{lemma}

\begin{proof} 
Take any mollifier $\varrho\in C_c^\infty(\mR^{d+1})$ such that $\int \varrho=1$, and let  $\varrho^n(t,x)=n^{d+2} \varrho(n^2t, n x)$.  
    Let $u^n= u*\varrho^n\in C^\infty(\mR\times \mR^d)$. 
    It\^o's formula gives that 
	\begin{equation}\label{eq:ito}
			 u^n (t, X_t)-u^n(t', X_{t'})= \int_{t'}^t (\p_r + A_r) u^n(r, X_r)\d r + \sqrt{2} \int_{t'}^t \nabla u^n (r, X_r)\cdot \d W_r, 
		\end{equation} 
        where $A_r$ is the Kolmogorov operator as in Section \ref{Sec:PDE}. 
		Since 
		\begin{equation*}
			\begin{aligned}
				\|(\p_t + A_t) (u- u^n) \|_{L^q_t L_x^p} \leq& \|\p_t u - \p_t u*\rho^n\|_{L^q_t L_x^p}+\|\Delta u- \Delta u * \varrho^n\|_{L^q_t L_x^p} \\
				&+ \|b\|_{L^\infty_t L_x^{d,\infty}} \|\nabla u -\nabla u* \varrho^n\|_{L^q_t L_x^{\frac{dp}{d-p}, p}}\\
				\leq& C \|\p_t u- \p_t u * \varrho^n \|_{L^q_t L_x^p}+C \|\nabla^2 u- \nabla^2 u * \varrho^n \|_{L^q_t L_x^p} \to 0 
			\end{aligned}
		\end{equation*}
        as $n\to \infty$, 
		by the Krylov estimate one has 
		\[
		    \int_{t'}^t (\p_r + A_r) u^n(r, X_r)\d r \rightarrow\int_{t'}^t (\p_r + A_r) u(r, X_r)\d r 
      \ \ \text{in}\ \ L^1(\Omega,\bP). 
		\]
		Moreover, noting that 
		\[
           \frac{d}{p}+ \frac{2}{q}<\frac{d}{2p}+\frac{1}{q}+1, 
		\]
	    by the parabolic-type Sobolev embedding in \Cref{Prop:Inter2}, there exists \(l>2q\) such that 
		\begin{equation*}
			\begin{aligned}
				\|\nabla u- \nabla u^n\|_{L^{l}_t L_x^{2p}} \leq C\l( \|\p_t u- \p_t u^n\|_{L^q_t L_x^p} + \|\nabla^2 u-\nabla^2 u^n\|_{L^q_t L_x^p} \r)\to 0. 
			\end{aligned}
		\end{equation*}
		This along with Doob's inequality and Krylov's estimate yields that 
		\begin{align*}
		    &\bE \sup_{0\leq t'< t\leq T} \l| \int_{t'}^t \nabla u^n(r, X_r)\cdot \d W_r- \int_{t'}^t \nabla u(r, X_r)\cdot \d W_r \r|^2 \\
            \leq & C \bE \int_{0}^T |\nabla u^n(r, X_r)-|\nabla u^n(r, X_r)|^2 \d r \\
            \leq& C \|(\nabla u^n- \nabla u)\|_{L^{2q}_t L_x^{2p}}^2 \leq C \|(\nabla u^n- \nabla u)\|_{L^{l}_t L_x^{2p}}^2\to 0,
		\end{align*} 
        where the last step is due to \(l>2q\). 
        Therefore, the right-hand side of \eqref{eq:ito} converges to that  of \eqref{Eq:Ito}.  
  Taking into account that for  any $0<\alpha< (2-\frac{d}{p}-\frac{2}{q})\wedge 1$, 
		\[
		    \|u\|_{C^\alpha} \leq C \l(\|\p_t u\|_{L^q_t L_x^p}+ \|\nabla^2 u\|_{L^q_t L_x^p}\r) 
		\]
		(see \cite{krylov2005strong}),  
  we also infer that the left-side of \eqref{eq:ito}  converges to that  of \eqref{Eq:Ito}, and thus finish the proof. 
\end{proof}

We are now in a position to provide the
\begin{proof}[Proof of the existence part of \Cref{Thm:SDE'}:] 
Let \(s=0\) for simplicity. Take any mollifier $\varrho\in C_c^\infty(\mR^{d+1})$ such that $\int \varrho=1$, and let  $\varrho^n(t,x)=n^{d+2} \varrho(n^2t, n x)$.  Set 
		\[
		b_1= b \1_{\{|b|>1\}}, ~ b_2= b\1_{\{|b|\leq 1\}} ~ \mbox{ and }~ b_i^n=b_i*\varrho^n, ~ b^n= b*\varrho^n. 
		\]
		Let  
		\[
		  p_1={3d}/{4} ~ \mbox{ and }~  p_2=2d
		\]
        and $(p_i,q_i)\in\sI$, \(i=1,2\). 
        Noting that 
        \[
            \|b_1\|_{L^{\infty}_t L^{p_1}_x}^{p_1} \leq C \|b\|_{L^\infty_tL^{d,\infty}_x}^d \int_1^\infty \lambda^{p_1-1-d} \d \lambda\leq C \|b\|_{L^\infty_tL^{d,\infty}_x}^d
        \]
        and
        \[
            \|b_2\|_{L^{\infty}_tL^{p_2}_x}^{p_2} \leq C \|b_2\|_{L^\infty_tL^{d,\infty}}^{d} \|b_2\|_{L^\infty_{t,x}}^{p_2-d} \leq C \|b\|_{L^\infty_tL^{d,\infty}}^{d}, 
        \]
        we have that $C^\infty_b \ni b_i^n \to b_i$ in $L^{q_i}_tL^{p_i}_x$. Since \( b^n \in C^\infty_b \), the classical SDE theory guarantees the existence of a unique strong solution \(X^n\) to \eqref{Eq:SDE}, with \(b\) replaced by \(b^n\), starting from \(s=0\). Moreover, by the classical parabolic PDE theory, for any $(p,q)\in\sI$, $t_1\in (0,T]$ and $f\in L^q_t L_x^p$, there exists a unique solution $v^n$ in \(L^q_tW^{2,p}_x\) satisfying 
		\[
		  \p_t v^n + \Delta v^n+ b^n\cdot \nabla v^n+f=0, \quad v^n(t_1)=0. 
		\]
	Then, by It\^o's formula and \Cref{Prop:Holder}, 
   there exist $\alpha \in(0,1/2),C>0$ such that 
	for any $0\leq  t_0<t_1\leq  T$, 
		\begin{align}\label{eq:kry1}
			\sup_n \bE \int^{t_1}_{t_0} f(t,X_{t}^n)\d t\leq \sup_{x\in\mR^d} |v^n(t_0, x)|\leq C(t_1-t_0)^{\alpha}\| f\|_{L^q_t L_x^p}.
		\end{align}
        
        Now let $\tau \in [0,T]$ be any bounded stopping time. Since 
        $$
        X^n_{(\tau+\delta)\wedge T}-X^n_{\tau}=\int^{(\tau+\delta)\wedge T}_\tau b_n(t,X^n_{t})\d t+\sqrt{2}(W_{(\tau+\delta)\wedge T}-W_\tau),\ \ \delta>0, 
        $$ 
    applying \eqref{eq:kry1} to $f=b^n_1$ we get 
		$$
		\bE \int^{(\tau+\delta)\wedge T}_\tau |b^n|(t,X^n_{t})\d t 
        \leq C\delta^\alpha \l( \| b_1^n\|_{L^\infty_t L_x^{p_1}}+\|b_2^n\|_{L^\infty_{t,x}}\r) \leq C\delta^{\alpha}. 
		$$
		So one derives that 
		\begin{align*}
			\bE \sup_{0\leq u\leq \delta}|X^n_{\tau+u}-X^n_{\tau}|&
			\leq \bE \int^{\tau+\delta}_\tau |b_n|(t,X^n_{t})\d t+\sqrt{2} \bE \sup_{0\leq u\leq \delta}|W_{\tau+\delta}-W_\tau| \leq   C\delta^{\alpha},
		\end{align*}
		where $\alpha$ and $C>0$ are  independent of $n$.  
  Hence, by \cite[Lemma 2.7]{zhang2018singular}, we obtain
		$$
		\sup_n\bE \left(\sup_{t\in[0,T]; u\in [0,\delta]}|X^n_{(t+u)\wedge T}-X^n_{t}|^{1/2}\right)\leq C\delta^{\alpha}, 
		$$
	which along with Chebyshev's inequality 
 yields that for any $\eps>0$,
		$$
		\lim_{\delta\to 0}\sup_n \bP\left(\sup_{t\in[0,T]; u\in [0,\delta]}|X^n_{(t+u)\wedge T}(x)-X^n_{t}(x)|>\eps\right)=0.
		$$
        This implies the tightness of 
        the probability laws $\{\bP\circ (X^n, W)^{-1}\}$ in $\cP(C^{\otimes2}([0,T];\mR^d))$. 
        
        Hence, by Skorokhold's representation theorem, there exist a probability space $(\Omega, \cF, \bQ)$,  and $C^{\otimes2}([0,T]; \mR^d)$-valued random variables $(Y^n, B^n)$ and $(Y,B)$ on $(\Omega, \cF, \bQ)$ such that 
		\[
		(X^n, W)\overset{d}{=}(Y^n, B^n)
		\]
        and up to a subsequence \(
		(Y^n, B^n){\to} (Y, B), ~ \bQ-a.s.\). it is clear that $B$ is a Brownian motion. 
        Moreover,  for any $(p,q)\in \sI$ and $f\in L^q_t L_x^p$, 
		\begin{equation}\label{eq:kry2}
		    \bE_{\bQ} \int^{T}_{0} f(t,Y_{t}^n) \d t, ~ \bE_{\bQ} \int^{T}_{0} f(t,Y_{t}) \d t  \leq C\| f\|_{L^q_t L_x^p},  
		\end{equation} 
        where $C$ is independent of $n$. 
  
  In order to  prove that $(Y, B)$ is a weak solution to \eqref{Eq:SDE}, we only need to show that 
		\[
		   \int_0^T |b(t, Y_t)-b^n(t, Y^n_t)| \d t \to  0, 
           \quad {\rm in}\ \bQ-{\rm probability}.
		\]
       To this end,  using \eqref{eq:kry2} we derive that for any fixed $N, n\geq 1$, 
		\begin{align*}
		   &\bE_{\bQ} \int_0^T |b(t, Y_t)-b^n(t, Y^n_t)| \d t \\  
           \leq & 
           \sum_{i=1}^2 \bE_{\bQ} \int_0^T |b_i(t, Y_t)-b_i^N(t, Y_t)|\d t
		  + \sum_{i=1}^2 \bE_{\bQ} \int_0^T |b_i^N(t, Y_t^n)-b_i^n(t, Y_t^n)|\d t\\
          & +   \bE_{\bQ} \int_0^T |b^N(t, Y_t)-b^N(t, Y_t^n)| \d t \\           
          {\leq}&  C \sum_{i=1}^2  \|b_i-b_i^N\|_{L^{q_i}_t L_x^{p_i}}  
          + 
          C \sum_{i=1}^2 \|b^N_i-b^n_i\|_{L^{q_i}_t L_x^{p_i}} 
          +  \int_0^T\bE_{\bQ} \l|b^N(t, Y^n_t)-b^N(t, Y_t)\r|\d t. 
		\end{align*}
	    Letting $n\to\infty$ and then $N\to \infty$, we thus obtain the desired existence result. 

        At last, by \Cref{Thm:Heat}, for each \(n\), \(X^n\) admits a transition density function \(p^{n}_{s,t}(x,y)\) that satisfies \eqref{Eq:AE} with the universal constant \(C\). Therefore, each limit point of $\{X^n\}$ also satisfies estimates \eqref{Eq:AE} and \eqref{Eq:Kry}. 
        \end{proof}

\subsection{Uniqueness} 
Let us continue to prove the uniqueness part of \Cref{Thm:SDE'}. 

\begin{proof}[Proof of the uniqueness part of \Cref{Thm:SDE'}:]  
Suppose that \((\Omega^1, \cF^1, \cF^1_t, \bP^1; X^1, W^1)\) and \((\Omega^2\), 
\(\cF^2, \cF^2_t, \bP^2; X^2, W^2)\) are two weak solutions to \eqref{Eq:SDE}.

Let us first consider the case where each $X^i~(i=1,2)$ satisfies 
the Krylov-type estimate \eqref{Eq:Kry}, with \((p,q)=(\mathfrak{p},\mathfrak{q})\in \sI\). Here \((\mathfrak{p},\mathfrak{q})\) is the same pair as in \Cref{Thm:LqLp}.

For any $c, f\in C_c^\infty (\mR\times \mR^d)$, 
in view of \Cref{Thm:LqLp},  
equation \eqref{Eq:BKE2} has a solution $v$ satisfying that  $v\in L^{\mathfrak{q}}_t W^{2, \mathfrak{p}}_x$ and $\p_t v \in L^{\mathfrak{q}}_t L_x^{\mathfrak{p}}$ with $d/{\mathfrak{p}}+2/{\mathfrak{q}}<2$. 
  Then, by the generalized It\^o's formula  \eqref{Eq:Ito}, 
    \begin{align*}
        &\d \l( \exp \l( \int_0^t c(s, X^i_s) \d s\r) v(t, X^i_t)  \r) \\
        =&\exp \l( \int_0^t c(s, X^i_s) \d s\r) f(t, X^i_t) \d t+ \sqrt{2} \exp \l( \int_0^t c(s, X^i_s) \d s\r)\nabla v(t, X_t) \cdot\d W^i_t. 
    \end{align*}
    Taking expectation one gets  
    \begin{align*}
        v(0, x)= \bE_{\bP^i} \int_0^T \exp \l( \int_0^t c(s, X^i_s) \d s\r) f(t, X^i_t) \d t. 
    \end{align*}
    Thus, by the density argument, this yields that for any $\lambda >0$, 
    \[
    \bE_{\bP^1} \int_0^T \exp \l( \int_0^t c(s, X^1_s) \d s\r) \e^{-\lambda t} \d t=\bE_{\bP^2} \int_0^T \exp \l( \int_0^t c(s, X^2_s) \d s\r) \e^{-\lambda t} \d t, 
    \]
    and so  
    \[
        \bP^1\circ \l( \exp \l( \int_0^t c(s, X^1_s) \d s\r) \r)^{-1}= \bP^2\circ \l( \exp \l( \int_0^t c(s, X^2_s) \d s\r) \r)^{-1} ,\quad t\in [0,T]. 
    \]
    Therefore, again by the density argument, for every \( A \in \cB(\mR^d) \) and \( t \in [0,T] \), the moment generating functions of \( \int_0^t \1_A(X^i_s) \, \d s \) are identical for \(i=1,2\), which implies the uniqueness in law \( \bP^1 \circ (X^1)^{-1} = \bP^2 \circ (X^2)^{-1} \). This together with the existence part of \Cref{Thm:SDE'} imply that  the unique solution admits a transition density function satisfying \eqref{Eq:AE}. 

\medskip 
    Now, suppose that 
    \(X\) is a weak solution to \eqref{Eq:SDE} satisfying the generalized Krylov estimate \eqref{Eq:Kry'} with \(s=0\), where the pair \((\mathfrak{p},\mathfrak{q})\) is the same as in \Cref{Thm:LqLp}. We claim that it indeed 
    satisfies \eqref{Eq:Kry} 
    for \((p,q)=(\mathfrak{p},\mathfrak{q})\) and \(s=0\).  Thus,  by the arguments in the previous case, we obtain the desired assertion.

    To this end, 
    note that for any \(\delta \in (0, T)\), 
    the process \((X_{t})_{t \in [\delta, T]}\) satisfies  
    \[
        X_t=X_\delta+\int_\delta^t b(s, X_s)\d s + \sqrt{2} (W_{t}-W_{\delta}),  
    \]
    and \eqref{Eq:Kry} with \((p,q)=(\mathfrak{p},\mathfrak{q})\in \sI\), \(C=C_\delta\) and \(s\) replaced by \(\delta\). 
    Then, as in the previous case, 
    one has 
    \[
        \bE \int_\delta^{T} f(t, X_t) \d t \leq \bE v(\delta, X_\delta) \leq \|v\|_{L^\infty_{t,x}}\leq C \|f\|_{L^{\mathfrak{q}}_tL^{\mathfrak{p}}_x}. 
    \]
    Here \(v\in L^{\mathfrak{q}}_tL^{\mathfrak{p}}_x\) is the solution to \eqref{Eq:BKE2} with \(c=0\). Note that, thanks to \Cref{Thm:LqLp}, the constant \(C\) on the right-hand side of the inequality above is universal, i.e., independent of \(\delta\). Thus, taking the limit as \(\delta \to 0\), we conclude that \(X\) satisfies the Krylov estimate \eqref{Eq:Kry} 
    with \((p,q)=(\mathfrak{p}, \mathfrak{q})\) and $s=0$. Therefore, the proof for the uniqueness part of \Cref{Thm:SDE'} is complete. 

    The proof for the Markov property of \((\mP_{s,x}, X)\) is standard (cf. \cite{stroock2007multidimensional}), so we omit this here. 
\end{proof}

In the end of this section, 
we also have the well-posedness of 
the following linear Fokker-Planck equation 
for all dimensions $d\geq 2$, 
\begin{equation}\label{Eq:LFP}
    \p_t \rho = \Delta \rho- \div (b\rho), \quad \rho|_{t=0}=\zeta\in\cM(\mR^d). 
\end{equation}

\begin{definition}
        We say that \(\rho\in C_t\cM_x\) is a distributional solution to the  Fokker-Planck equation \eqref{Eq:LFP}, 
        if \(b \in L^1(\rho)\) and for any \(\phi\in C_c^\infty([0,T]\times\mR^d)\), it holds that 
        \begin{equation}\label{eq:weak}
            \<\rho(t),\phi(t)\>-\<\zeta, \phi(0)\>= \int_0^t\<\rho(s), \p_s \phi(s)+\Delta\phi(s)+b(s)\cdot\nabla \phi(s)\> \d s, \quad t\in [0,T],  
        \end{equation} 
        where $\<\cdot, \cdot\>$ 
        denotes the integration over $\mathbb{R}^d$. 
    \end{definition}

\begin{proposition}[Well-posedness of linear FPE]\label{Prop:LFP}
    Let $d\geq 2$. Assume that the drift $b$ satisfies \eqref{Eq:Ab}. Then, 
    the Fokker-Planck equation \eqref{Eq:LFP} has a unique distributional solution $\rho \in C_t\cM_x$ satisfying 
    \[
        \rho \in L^{\mathfrak{q}'}(\delta, T; L^{\mathfrak{p}'}_x), \quad \forall \delta\in (0,T), 
    \] 
    where $(\mathfrak{p}$,  $\mathfrak{q})$ 
    is as in Theorem \ref{Thm:SDE'}, 
    and $\mathfrak{p}', \mathfrak{q}'$ are the conjugate numbers of $\mathfrak{p}$ and $\mathfrak{q}$, respectively.
\end{proposition}

\begin{proof}
    By \Cref{Thm:SDE'} and the linearity of \eqref{Eq:LFP}, the function \(\rho(t):=\int_{\mR^d} p_{0,t}(x,y)\zeta(\d x)\) is a distributional solution of  \eqref{Eq:LFP},   
    where $p_{0,t}$ is the transition density in \Cref{Thm:SDE'}. 
    
    Regarding the uniqueness, 
    as in the proof for the uniqueness part of \Cref{Thm:SDE'}, we can assume \(\rho \in C_t\cM_x \cap L^{\mathfrak{q}'}_tL^{\mathfrak{p}'}_x\). 
    Let \(v\in L^{\mathfrak{q}}_t W^{2,\mathfrak{p}}_x\) with \(\p_t v\in L^{\mathfrak{q}}_tL^{\mathfrak{p}}_x\) be the solution to \eqref{Eq:BKE2} with \(c=0\) 
    and \(f\in C_c^\infty((0,T)\times\mR^d)\). 
    Then, using a standard limiting argument, it can be shown that \eqref{eq:weak} holds 
    for \(\rho \in C_t\cM_x \cap L^{\mathfrak{q}'}_tL^{\mathfrak{p}'}_x\) and \(\phi\)  replaced by \(v\), 
    which leads to 
    \[
    -\<\zeta, v(0)\>= \int_0^T \<\rho(s), f(s)\> \d s, \quad \forall f\in C_c^\infty((0,T)\times\mR^d), 
    \]
    thereby yielding the uniqueness of \(\rho\). 
\end{proof}

\section{Construction of non-unique solutions}\label{Sec:EX}

This section is devoted to the non-uniqueness problem of SDEs and Fokker-Planck equations 
in the supercritical regime.

\subsection{SDEs with supercritical drifts}
	
Let us first prove the non-uniqueness result  in Theorem \ref{Thm:example} by constructing a divergence free drift $b\in L^{p,\infty}(\mR^d)$ with $d/2<p<d$ and $d\geq 3$, such that \eqref{Eq:SDE} have at least two distinct weak solutions starting from the origin.  

\begin{proof}[Proof of \Cref{Thm:example}]
The proof mainly proceeds in five steps. 

\medskip 
\paragraph{\bf 
Step 1: Construction of the drift} 

Let us first give the specific construction of the drift $b$ in the Lorentz space $L^{p,\infty}_x$ 
with $p\in (d/2, d)$.

 Let $g$ be a non-negative smooth function on $[0,\infty)$ such that $g'(r)\geq 0$ for all $r\geq 0$, $g(r)=0$ if $0\leq r\leq 1/2$ and $g(r)=1$ if $r>1$.  Let 
		\begin{align}  \label{alpha-def}
		    \alpha := \frac{d}{p}\in (1,2). 
		\end{align}
		For any $x=(x_1,\cdots, x_d)\in \mR^d$ with $x_d> 0$, set 
		\[
		r := \l(x_1^2+\cdots+x_{d-1}^2\r)^{1/2} ~~\mbox{ and }~~ 
         H(x) := r^{d-1} x_d^{-\alpha} g({x_d}/{r}).
		\]
	We define the drift $b$ by 
		\begin{equation}\label{Eq:bn}
			\begin{aligned}
				b_d(x):=&N  r^{2-d} \p_r H(x)\\
				=& N  (d-1) x_d^{-\alpha} g({x_d}/{r}) -N  r^{-1} x_d^{-\alpha+1}g'({x_d}/{r})  ~\mbox{ if }~x_d>0 
			\end{aligned}
		\end{equation}
		and for $1\leq i\leq d-1$, 
		\begin{equation}\label{Eq:bi}
			\begin{aligned}
				b_i(x):=& -N  x_i r^{1-d}  \p_{x_d} H(x)\\
				=&N  \alpha (x_i x_d^{-\alpha-1}) g({x_d}/{r}) -N  r^{-1}x_i x_d^{-\alpha} g'({x_d}/{r}) ~\mbox{ if }~ x_d>0. 
			\end{aligned}
		\end{equation}
	    Set 
		\[
		    b(x_1,\cdots, x_{d-1}, x_d):=- b(x_1,\cdots, x_{d-1}, -x_d) ~\mbox{ if }~ x_d<0
		\]
		and, if $x_d =0$, 
		\[
		    b(x_1,\cdots,x_{d-1},0):=0. 
		\]
		
    Next, we shall verify that 
    the constructed drift $b$ belongs to the Lorentz space $L_x^{p,\infty}$ and 
    is divergence free $\div\, b=0$. 

  To this end, we infer from 
		 \eqref{Eq:bn}-\eqref{Eq:bi} and the anti-symmetry of $b$ about the $x_d$-axis that 
		\begin{equation}\label{Eq-b-bdd}
			|b(x)|\leq C |x_d|^{-\alpha} \1_{\{r<2 |x_d|\}}. 
		\end{equation}
		Thus,  by the choice \eqref{alpha-def}, 
		\[
		    \|b\|_{ L^{p,\infty}}^p =\sup_{t>0} t^p \l| \l\{x: |b(x)|>t\r\}\r|\leq C \sup_{t>0} t^p \l|B_{t^{-\frac{1}{\alpha}}}\r|\leq C<\infty. 
		\]
		Note that, $b$ is smooth on $\mR^d\backslash\{0\}$.
		
		Moreover, by the construction, we have 
		\[
		\p_d b_d(x)= N r^{2-d} \p^2_{r x_d} H(x), ~\mbox{ for all } ~x_d>0 
		\]
		and 
		\[
		\p_i b_i(x)= -N[r^{1-d} + (1-d) x_i^2 r^{-d-1}] \p_{x_d}H(x) +N x_i^2 r^{-d} \p^2_{r x_d }H(x), 
		\]
		for all $x_d>0$ and $1\leq i\leq d-1$. Taking into account \eqref{Eq-b-bdd} and the anti-symmetry of $b$ about the $x_d$-axis, we obtain that $\div b(x)=0$ if $x\neq 0$.  
  
    In order to verify that $\div b=0$ in the sense of distribution, we need to show that  $\int b\cdot \nabla \varphi =0$ 
    for any $\varphi \in C_c^\infty(\mathbb{R}^d)$. 
    Actually, by the integration-by-parts formula and \eqref{Eq-b-bdd}, 
    for any $\rho >0$, 
		\begin{align*}
			\int b\cdot \nabla \varphi =&  \int_{B_\rho} b\cdot\nabla \varphi +\int_{\p B_\rho}b \varphi \cdot  \d \vec{\sigma} -\int_{B_\rho^c} \div b\cdot \nabla \varphi \\
			=&  \int_{B_\rho} b\cdot\nabla \varphi + \int_{\p B_\rho}b \varphi \cdot  \d \vec{\sigma} \leq  C \rho^{-\alpha+d-1}, 
		\end{align*}  
    which tends to zero as $\rho \to 0^+$, due to the fact that $ \alpha\in (1,2)$ and $d\geq 3$. 

\medskip 

    \paragraph{{\bf Step 2: Contradiction arguments}} 

    Let \(\Omega=C(\mR_+; \mR^d)\) and \(\omega_t\) be the canonical process. We consider a continuous functional $\sT$ on $\cP(\Omega)$ defined by 
    \[
    \sT(\mP) := \mE_{\mathbb{P}}  \l[ \int_{0}^\infty \e^{-t} f(\omega_t) \, \d t \r], 
     \]
    where 
    \[
    f(x):= \sgn(x_d)g(|x_d|). 
    \]
    Since $g|_{[0,\frac{1}{2}]}\equiv0$, $f$ is a continuous function on $\mR^d$. Thus, the map $\omega\mapsto \int_0^\infty \e^{-t} f(\omega_t) \d t$ is continuous from $\Omega$ to $\mR$. This yields that $\sT$ is continuous on $\cP(\Omega)$, i.e., 
    \begin{equation}
	\sT(\mP_n) \to \sT (\mP), ~\mbox{ if }~ \mP_n\Rightarrow \mP. 
    \end{equation}
    
    Thanks to \cite[Theorem 1.1]{zhang2021stochastic}, there exits at least one weak solution to \eqref{Eq:SDE} satisfying \eqref{Eq:Kry}, for any \(x\in \mR^d\) and \(s=0\). 
    
    Below we shall prove the non-uniqueness of solutions by contradiction arguments. Assume that the law of weak solutions staring from the origin is unique in the Krylov class \eqref{Eq:Kry}. Noting that \(b\) is smooth and uniformly bounded on \(B_\eps^c(0)\) for any \( \eps>0\), our assumption and \cite[Theorem 6.1.2]{stroock2007multidimensional} imply that martingale solutions to \eqref{Eq:SDE} staring from any \(x\in \mR^d\) is unique in the Krylov class \eqref{Eq:Kry}. 
		 
    Let $\mP_x$ denote the unique martingale solution starting from $x$. Our strategy is to find two sequences $(x^n)$ and $(y^n)$ in $\mR^d$ converging to zero such that   
		 \begin{equation}\label{Eq:Tx}
		 	\mP_{x^n}\Rightarrow \mP_0 ~ \mbox{ and }~ \lim_{n\to\infty} \sT(\mP_{x^n})\geq p_0>0 
		 \end{equation} 
         for some $p_0>0$, 
		 and 
		 \begin{equation}\label{Eq:Ty}
		 	\mP_{y^n}\Rightarrow \mP_0 ~ \mbox{ and }~ \lim_{n\to\infty}\sT(\mP_{y^n})=0. 
		 \end{equation}
		Then, the continuity of $\sT$ and \eqref{Eq:Tx}-\eqref{Eq:Ty}  yield that  
		 \[
		     0<p_0\leq \lim_{n\to\infty}\sT(\mP_{x^n})=\sT(\mP_0)=\lim_{n\to\infty}\sT(\mP_{y^n}) =0, 
		 \]
		 which leads to a contradiction.

   In order to construct such sequences, 
   we first note that, 
   since both $b$ and $f$ are anti-symmetric about the hyperplane $\Pi=\{x\in \mR^d: x_d=0\}$, 
   \begin{equation}\label{Eq:E0}
		\sT(\mP_y) = \mE_y  \l[ \int_{0}^\infty \e^{-t} f(\omega_t) \, \d t \r]= 0, \quad \mbox{ if }\ y_d=0. 
		\end{equation} 
    (Here we write $\mE_y$ for $\mE_{\mathbb{P}_y}$ for simplicity.) 
    
    Let $\mR^d\backslash \{0\}\ni z^n\to 0$. As  shown in the proof of \Cref{Thm:SDE'} (using \cite[Theorem 2.2]{zhang2021stochastic} instead of \Cref{Prop:Holder}), $\{\mP_{z^n}\}$ is tight in $\cP(\Omega)$ and 
    its accumulation points are martingale solutions to \eqref{Eq:SDE} with $s=0$ and $x=0$ and satisfy \eqref{Eq:Kry}. But by the uniqueness assumption, one has $\mP_{z^n}  \Rightarrow \mP_0$, $n\to \infty$. Therefore, if \(y^n\to 0\) with $(y^n)_d=0$, then \eqref{Eq:E0} implies \eqref{Eq:Ty}. 
		
    \medskip 
    It remains to construct the sequence \(x^n\to 0\) verifying \eqref{Eq:Tx}.

    \medskip 
\paragraph{\bf 
 Step 3: Reduction to exit probability estimate.}

In order to find such \(\{x^n\}\) satisfying \eqref{Eq:Tx},  
we define the cone 
    \[
    \cC_{k,h}= \left\{z\in \mR^d: k(z_1^2+\cdots+z_{d-1}^2)^{\frac{1}{2}}< z_d < h, ~  k,h>0 \right\}, \quad \cC_k=\cC_{k,\infty}. 
    \]

    Let 
    \[
    \overline{\tau}=\inf\{t>0: \omega_t\notin \cC_{1,2}\} ~\mbox{ and } ~ \tau=\inf\{t>0: \omega_t\notin \cC_{1}\} 
    \]
    be the first exit times of the canonical process from \(\cC_{1,2}\) and \(\cC_1\), respectively. Also let 
    \[
    \sigma_0=\inf\{t>0: (\omega_t)_d = 0\}, \quad \sigma_1=\inf\{t>\overline{\tau}: (\omega_t)_d = 1\}  
    \]
    be the first hitting times of \(\omega\) to the hyperplanes \(\{x_d=0\}\) and \(\{x_d = 1\}\) after time \(\overline{\tau}\), respectively. 
    Let 
    \[
    \kappa\in (1, (d-1)/\alpha).
    \]

        \medskip 
   \paragraph{\bf Claim:} 
    We have the exit probability estimate:          
        there exists a positive constant 
        $p_1$ such that 
	\begin{equation}\label{Eq:key}
		\inf_{x\in \cC_{\kappa,1}} \mP_x \big( \overline{\tau}< 1\wedge \tau, ~ \sigma_1>1+\overline{\tau} \big) \geq p_1>0.  
	\end{equation}
    Intuitively, 
    estimate \eqref{Eq:key} shows that the solution trajectories starting from \(x \in \cC_{\kappa,1}\) are likely to reach the hyperplane \(\{x_d = 2\}\) before exiting \(\cC_{1}\), and remain in the region \(\{x_d > 1\}\) for a unit time with high probability.

        Suppose that the exit probability estimate \eqref{Eq:key} is true. Then, taking into account the uniqueness assumption, $\{\mP_x\}$ forms a strong Markov process (see \cite{stroock2007multidimensional}), we infer that for any $x\in \cC_{\kappa, 1}$, 
		\begin{align*}
			\sT(\mP_x) =& \mE_x \int_0^{\sigma_0} \e^{-t} f(\omega_t) \d t+ \mE_x \left[ \e^{-\sigma_0} \left(\int_0^\infty   \e^{-t} f(\omega_t) \d t\right)\circ\theta_{\sigma_0}\right] \\
			\geq & \mE_x \left[ \1_{\big\{ \overline{\tau}< 1\wedge \tau, ~ \sigma_1>1+\overline{\tau} \big\}}  \int_{\overline{\tau}}^{\sigma_1} \e^{-t} \d t \right] \\&+ \mE_x\e^{-\sigma_0} \mE_{\omega_{\sigma_0}} \int_0^\infty   \e^{-t} f(\omega_t) \d t \geq p_1(\e^{-1}-\e^{-2})=:p_0>0, 
		\end{align*} 
		where  \(\theta_t\) is the shift operator, and we used the fact that $f(x)=1$ for all $x$ with $x_d\geq 1$ 
        and \eqref{Eq:E0}, \eqref{Eq:key} 
        in the last step. 
        This leads to the desirable limit \eqref{Eq:Tx}.

\medskip 

 Therefore, we are left to show the exit probability estimate \eqref{Eq:key}.  

 For this purpose, we let 
    \[
    B_t(\omega):= \omega_t- \omega_0 -\int_0^t b(\omega_s) \d s, \quad t\geq 0, 
    \]
    which is a Brownian motion under $\mP_x$ with variance $2t$, for any \(x\in \mR^d\).
    Set 
    \begin{align*}\label{Eq:OmegaN}
        \Omega_N:= \left\{\omega\in \Omega: |B_s(\omega)-B_t(\omega)|\leq N^{\frac{1}{2(1+\alpha)}}|s-t|^\frac{1}{1+\alpha},  s,t \in [0, 10] \right\} 
    \end{align*}
    and 
    \[
         \Omega^x:=\{\omega\in \Omega: \omega_0=x\} \mbox{ and }  \Omega_N^x:=\Omega_N\cap \Omega^x, \quad \mbox{ for all} x\in \mR^d. 
    \]
    Since $0<1/(1+\alpha)<1/2$, we can choose $N\gg 1$ such that 
    \[
        \mP_x(\Omega^x_N)\geq 1/2, \quad \mbox{ for all } x\in\mR^d. 
    \]
    
    Next we aim to show that for each \(x\in \cC_{\kappa, 1}\), all paths in \(\Omega^x_N\) reach the hyperplane \(\{y\in \mR^d: y_d=2\}\) before time \(1\wedge \tau\) when \(N\gg1\), i.e. 
    \begin{equation}\label{eq:tau2-leq-tau}
        \overline{\tau}(\omega)<1\wedge \tau(\omega), \quad \mbox{ for all }x\in \cC_{\kappa,1} \mbox{ and }  \omega\in \Omega_N^x. 
    \end{equation}
    In other words, all trajectories in \(\Omega^x_N\) correspond to the green paths depicted in \Cref{Fig:path}. 
    
    If \eqref{eq:tau2-leq-tau} holds, then 
    $$
	\mP_x (\overline{\tau}<1\wedge \tau)\geq \mP_x(\Omega_N^x)\geq  1/2, \quad x\in \cC_{\kappa,1}. 
    $$
    Since $b$ is uniformly bounded on $\{x\in \mR^d: x_d\geq 1\}$, there exists a positive constant $p_2>0$ such that 
    $$
	\inf_{\{x\in\mR^d:\, x_d=2\}} 
    \mP_x (\sigma_1>1) \geq p_2>0. 
    $$
    Using the above two estimates and the strong Markov property, we thus obtain that for all $x\in \cC_{\kappa, 1}$, 
    \begin{align*}
	\mP_x \big( \overline{\tau}< 1\wedge \tau, ~ \sigma_1>1+\overline{\tau} \big)
	=& \mP_x \big( \overline{\tau}<1\wedge \tau, ~  \sigma_1\circ\theta_{\overline{\tau}}>1 \big) \\
	\geq& \mP_x (\overline{\tau}<1\wedge \tau) \inf_{\{x\in\mR^d:\, x_d=2\}}\mP_x (\sigma_1>1)\\
	\geq& p_2/2=:p_1>0, 
    \end{align*}
    which yields \eqref{Eq:key}, thereby finishing the proof.  

    \medskip 
\paragraph{\bf 
 Step 4: Trajectories wander within the cone}  
 
Now, the proof reduces to proving \eqref{eq:tau2-leq-tau}. 
For this purpose, we shall first show 
in this step that for $N$ large enough, 
    \begin{align}\label{eq:OmegaN-good}
         \Omega_N^x \subseteq \l\{\omega\in \Omega^x:\  \omega_t\in \cC_{1},\   \forall t\in[0, t_x];\  \omega_{t_x}\in \cC_\kappa \r\},  
         \quad x\in \overline{\cC}_{\kappa,2}\backslash \{0\}, 
    \end{align}  
    where 
    \begin{equation}\label{eq:tx}
        t_x :=N ^{-1} |x_d|^{1+\alpha}, \quad x\in \mR^d.
    \end{equation} 
    The above inclusion indicates that the solution paths starting from \(x \in \cC_{\kappa, 2}\) 
stay in the larger cone \(\cC_1\) over a certain time interval \([0, t_x]\),  
and return back to the small cone  \(\cC_{\kappa}\) at time \(t_x\).

 \medskip 
In order to prove \eqref{eq:OmegaN-good}, we define 
    \[
    V_t^x(\omega) :=x_d+\int_0^t b_d(\omega_s) \d s ~\mbox{ and }~ \widehat x:=(x_1,\cdots, x_{d-1}). 
    \]
    Obviously, 
    \[
    (\omega_t)_d=V^x_t(\omega)+(B_t(\omega))_d,\quad \mbox{ for all } \omega\in \Omega^x. 
    \]

    Then, by the construction of the drift in \eqref{Eq:bn} and \eqref{Eq:bi}, we have that 
    \[
        b_d(y)=N  (d-1) y_d^{-\alpha},\quad  \mbox{ for all } y\in \cC_1 
    \]
    and for $1\leq i\leq d-1$, 
    \[
        b_i(y)= N  \alpha (y_iy_d^{-\alpha-1}), \quad  \mbox{ for all } y\in \cC_1. 
    \]
    Thus, for all \(x\in \cC_1, \omega\in \Omega^x\) and \(t\in [0, \tau(\omega)]\), \begin{align}\label{eq:Vt}
        \begin{aligned}
			V_t^x(\omega) -x_d=&\int_0^t b_d(\omega_s)\d s= N  (d-1)  \int_0^t [V^x_s(\omega)+(B_s(\omega))_d] ^{-\alpha} \d s\geq 0, \\
            \widehat \omega_t - \hat x=& N \alpha \int_0^t   [V^x_s(\omega)+(B_s(\omega))_d]^{-\alpha-1} \widehat \omega_s \d s+\widehat B_t(\omega). 
		\end{aligned}
	\end{align} 

    Then by \eqref{eq:Vt} and \eqref{eq:tx}, we have \begin{equation}\label{Eq:wt}
        \begin{aligned}
        &|B_t(\omega)-B_0(\omega)|=|B_t(\omega)| \leq N^{\frac{1}{2(1+\alpha)}}t_x^{\frac{1}{1+\alpha}} \\
        =&\eps(N) x_d   \leq \eps(N) V_t^x(\omega), \quad \mbox{ for all }x\in \overline{\cC}_{\kappa,2}\backslash \{0\},  \omega\in \Omega_N^x \mbox{ and } t\in [0, t_x\wedge \tau(\omega)], 
        \end{aligned}
    \end{equation}
    where 
    $$
        \eps = \eps(N) :=N ^{-\frac{1}{2(1+\alpha)}}\to 0,\ \ \text{as}\ \ N \to\infty.
    $$ 
    This together with \eqref{eq:Vt} yield  that 
    $$
        (1+\eps)^{-\alpha} N  (d-1) (V_t^x(\omega))^{-\alpha} \leq \frac{\d V_t^x(\omega)}{\d t} \leq (1-\eps)^{-\alpha} N  (d-1) (V_t^x(\omega))^{-\alpha}, 
    $$
    for all \(x\in \overline{\cC}_{\kappa,2}\backslash \{0\}\), \(\omega\in \Omega_N^x\) and \(t\in [0, t_x\wedge \tau(\omega)]\). By virtue of Chaplygin's Lemma, we derive \begin{align}\label{eq:est-Vt}
        \begin{aligned}
        \underline{V}_t^x(\omega)
        :=&\left[x_d^{1+\alpha}+(1+\eps)^{-\alpha}N  (d-1)(\alpha+1) t\right]^{\frac{1}{1+\alpha}}\\
        \leq& V_t^x(\omega) \leq \left[x_d^{1+\alpha}+(1-\eps)^{-\alpha}N  (d-1)(\alpha+1) t\right]^{\frac{1}{1+\alpha}}=:\overline{V}_t^x(\omega), 
        \end{aligned}
    \end{align}
    for all \(x\in \overline{\cC}_{\kappa,2}\backslash \{0\}, \omega\in \Omega_N^x\) and \(t\in [0, t_x\wedge \tau(\omega)]\). This along with \eqref{Eq:wt} yields that 
    \begin{equation}\label{Eq:Xd}
	(\omega_t)_d = V_t^x(\omega)+(B_t(\omega))_d  \geq (1-\eps)  V_t^x(\omega) \geq (1-\eps) \underline{V}_t^x(\omega), 
    \end{equation}
    for all \(x\in \overline{\cC}_{\kappa,2}\backslash \{0\}, \omega\in \Omega_N^x\) and \(t\in [0, t_x\wedge \tau(\omega)]\). 
        
    Next we aim to show that 
    \begin{equation}\label{eq:tx-tau}
        t_x < \tau(\omega), \quad \mbox{ for all } x\in \overline{\cC}_{\kappa,2}\backslash \{0\} \mbox{ and } \omega\in \Omega_N^x. 
    \end{equation}
    Actually, since $1<\kappa <{(d-1)}/\alpha$, $|\widehat x|< x_d/\kappa$ for each $x\in \cC_\kappa$, it follows from \eqref{eq:Vt} and \eqref{Eq:wt} that
    \begin{equation}\label{eq:hat-omega}
	\begin{aligned}
		|\widehat \omega_t|\leq& |\widehat x|+ N  \alpha \left|\int_0^t  \widehat \omega_s [V_s^x(\omega)+(B_s(\omega))_d]^{-\alpha-1} \d s\right|+ |\widehat{B}_t(\omega)|\\ 
		\leq & |\widehat x| + \eps V_t^x(\omega) + N  \alpha\int_0^t [V_s^x(\omega)+(B_s(\omega))_d]^{-\alpha} \d s\\
		= & |\widehat x| + \eps V_t^x(\omega) + \frac{\alpha}{d-1} (V_t^x(\omega)-x_d)\\ 
		\leq & \left(\eps + \frac{\alpha}{d-1}\right) V_t^x(\omega)+ \left(\frac1\kappa-\frac{\alpha}{d-1}\right)x_d\\ 
		\leq& \left(\eps+\frac{1}{\kappa}\right) V_t^x(\omega), \quad \mbox{ for all } x\in \overline{\cC}_{\kappa,2}\backslash \{0\}, \omega\in \Omega_N^x \mbox{ and } t\in [0, t_x\wedge\tau(\omega)], 
	\end{aligned}
    \end{equation} 
    where the last step is due to the fact that $x_d \leq V_t^x$ implied by \eqref{eq:Vt}. 
	In the second inequality above, we also used the fact that $|\widehat \omega_t| \leq (\omega_t)_d$ when $t\in [0,\tau(\omega)]$.  
    Combining the above estimate with \eqref{eq:est-Vt} and \eqref{Eq:Xd}, we come to  
		$$
		\frac{(\omega_t)_d}{|\widehat \omega_t|} \geq (1-\eps)\l(\eps+ \frac{1}{\kappa}\r)^{-1} >1, \quad \mbox{ for all }x\in \overline{\cC}_{\kappa,2}\backslash \{0\}, t\in [0, t_x\wedge \tau(\omega)] \mbox{ and } \omega\in \Omega_N^x, 
		$$
		provided that $N $ is sufficiently large. Since $\tau(\omega)$ is the first exit time of $\omega_t$ from $\cC_{1}$,  we get $t_x\wedge \tau(\omega) < \tau(\omega)$, which is \eqref{eq:tx-tau}. 
        
        Moreover, via \eqref{eq:est-Vt}, for sufficiently large $N$, \(x\in \overline{\cC}_{\kappa,2}\backslash \{0\}\) and \(\omega\in \Omega_N^x\), we also have 
        \[
        \overline{V}_{t_x}^x(\omega)\geq V_{t_x}^x(\omega)\geq \underline{V}_{t_x}^x(\omega) \geq  2^{1/3}x_d,\]
        and so by \eqref{Eq:Xd}-\eqref{eq:hat-omega}, 
		\begin{align*}
			\frac{(\omega_{t_x})_d}{|\widehat \omega_{t_x}|}
            \geq& \frac{(1-\eps) \underline{V}_{t_x}^x(\omega)}{\left(\eps + \frac{\alpha}{d-1}\right) \overline{V}_{t_x}^x(\omega)+ \left(\frac1\kappa-\frac{\alpha}{d-1}\right)x_d} \\
			\geq& \frac{(1-\eps) \underline{V}_{t_x}^x(\omega)}{\left[\eps + 2^{-1/3}\kappa^{-1}+ (1-2^{-1/3})\frac{\alpha}{d-1}\right]\overline{V}_{t_x}^x(\omega)}
            >  \kappa,  
		\end{align*}
	   where the last inequality was also due to \eqref{eq:tx}, \eqref{eq:est-Vt} and the fact that $\kappa<(d-1)/\alpha$ and that $\eps\to0$ as $N\to \infty$. Thus,  
    $$
    \omega_{t_x}\in \cC_\kappa, \mbox{ for any }  x\in \overline{\cC}_{\kappa,2}\backslash \{0\} \mbox{ and } \omega\in \Omega_N^x, 
    $$
    provided that $N\gg 1$.

    To sum up, we obtain 
    the desired inclusion \eqref{eq:OmegaN-good} for $N$ large enough. 
    
\medskip 
\paragraph{\bf 
 Step 5: Proof of 
 exit probability estimate} 
 Now we are ready to prove  \eqref{eq:tau2-leq-tau}, and so the  exit probability estimate \eqref{Eq:key}. 
    
    We begin by showing that 
    \begin{equation}\label{eq:tau-leq-1}
        \overline{\tau} < 1, \quad \mbox{ for all } x\in \cC_{\kappa,1}\mbox{ and }\omega\in \Omega_N^x. 
    \end{equation} 
    If not, then there exit \(x\in \cC_{\kappa,1}\) and \(\omega\in \Omega_N^x\) such that \((\omega_{{1}/{2}})_d<2\). On the other hand, 
    using that \(b_d(y)=N(d-1)y_d^{-\alpha}, ~y\in \cC_{1,2}\) and \eqref{eq:Vt}, we have 
    \[
        (\omega_{1/2})_d = V_{1/2}^x(\omega)+(B_{1/2}(\omega))_d \geq x_d+N(d-1)2^{-\alpha-1}- N^{\frac{1}{2(1+\alpha)}} >2. 
    \]
    This contradiction implies \eqref{eq:tau-leq-1}. 
    
    So, it remains to show that \(\overline{\tau}(\omega)<\tau(\omega)\) when \(x\in \cC_{\kappa,1}\), \(\omega\in \Omega_N^x\) (this implies the red path depicted in \Cref{Fig:path} cannot belong to \(\Omega^x_N\)). 
    Suppose, for contradiction, that this is not the case. Then there exist \(x\in \cC_{\kappa,1}\) and \(\omega\in \Omega_N^x\) such that \(\overline{\tau}(\omega)=\tau(\omega)\) and \(\omega_{\overline{\tau}}\in \partial \cC_1\). Let \(S(\omega)\) denote the last time at which \(\omega\) exits \(\cC_{\kappa,2}\) before \(\overline{\tau}(\omega)\), i.e. 
    \[
    S(\omega)=\sup\l\{0<t<\overline{\tau}(\omega): \omega_t\notin \cC_{\kappa,2}\r\}. 
    \]
    Then \(S(\omega)<\overline{\tau}(\omega)<1\), \(\omega_S\in \partial \cC_{\kappa,2}\backslash\{0\}\) and \(t_{\omega_S}<N^{-1}2^3<9\). Therefore, the inclusion in \eqref{eq:OmegaN-good} for \( x \in \overline{\cC}_{\kappa,2}\backslash\{0\} \) and \( \omega \in \Omega_N^x \) also holds, respectively, for the new starting point \( \omega_S \) and the new path \( \theta_{S} \omega \). Thus, we have 
    \[
        S(\omega)<S(\omega)+t_{S(\omega)}<\tau(\omega)=\overline{\tau}(\omega), \quad \omega_{S+t_{S}}\in \cC_{\kappa,2}, 
    \]
    which however contradicts the definition of \(S\). 
    Finally, we obtain \eqref{eq:tau2-leq-tau} and finish the proof.
\end{proof}

\medskip 
\subsection{FPE with supercritical drifts}

We close this section with 
the proof for the non-uniqueness of  linear Fokker-Planck equations 
with supercritical drifts. 

\begin{proof}[Proof of Corollary \ref{Cor-LinFPK-Nonuniq}]
    Let \(b\) be the same drift constructed in the above proof of \Cref{Thm:example}. By \cite[Theorem 1.1]{zhang2021stochastic}, there exists at least one weak solution to \eqref{Eq:SDE} satisfying the Aronson-type estimate \eqref{Eq:Kry} for any \(x \in \mathbb{R}^d\), and the one-dimensional marginal distribution of this solution solves \eqref{Eq:LFP} with the initial data \(\delta_x\). 

    Suppose that there is only one solution \(\rho\) to \eqref{Eq:LFP} with \(\rho(0)=\delta_0\) in \(C_t\cP_x\). For any \(x\in \mR^d\), let \(\mP_x\) be any martingale solution to \eqref{Eq:SDE}  with \(s=0\). We shall prove that  the one-dimensional marginal distribution of \(\mP_x\) is unique. 

For this purpose, let \(\sigma=\inf\{t>0: \omega_t=0\}\). 
Because the drift  \(b\) is smooth in \(\mR^d\backslash \{0\}\) and bounded on \(B_\eps^c(0)\) for any \(\eps>0\), 
by the standard SDE well-posedness theory, 
for any \(x\neq 0\) away from the origin,  
the restriction 
\(\mP_x\restriction_{\cF_\sigma}\) of the martingale solution $\mP_x$ 
to the 
filtration $\cF_\sigma$  
is uniquely determined. 
Then, for any \(x\neq 0\) 
and any bounded measurable function $f$, 
    \begin{align}  \label{Efwt}
\mE_x(f(\omega_t))=&\mE_x\l(f(\omega_t)\1_{\{\sigma\leq t\}}\r)+\mE_x\l(f(\omega_t)\1_{\{\sigma> t}\}\r)  \nonumber \\
        =& \mE_x\l[ \1_{\{\sigma\leq t\}} \mE_x \l(f(\omega_{t})|\cF_{\sigma}\r)\r]+\mE_x\l(f(\omega_{t\wedge \sigma})\1_{\{\sigma> t\}}\r)  \nonumber \\
        =& \mE_x\l[ \<f, \rho_o(t-\sigma)\>\1_{\{\sigma\leq t\}} \r]+\mE_x\l(f(\omega_{t\wedge \sigma})\1_{\{\sigma> t\}}\r), 
    \end{align}
    where the last identity 
    was due to the uniqueness  assumption and the fact that the regular conditional probability \(\mP_x(\cdot|\cF_\sigma)(\omega')\) is also a martingale solution to \eqref{Eq:SDE} starting from \(0\), for \(\mP_x\)-a.s. \(\omega'\). 
    Taking into account that 
    the distribution of \((\sigma, \omega_{\cdot\wedge\sigma})\) is uniquely determined under \(\mP_x\), 
    we thus infer that any one-dimensional marginal distribution of weak solutions to \eqref{Eq:SDE} 
    satisfies the same identity \eqref{Efwt}, 
    and thus,     
    is unique for all \(x\in \mR^d\). 

    Consequently, in view of \cite[Theorem 6.2.3]{stroock2007multidimensional}, 
we infer that \(\mathbb{P}_x\) is uniquely determined for every \(x \in \mathbb{R}^d\), 
which, however, contradicts the non-uniqueness result in \Cref{Thm:example}. 
\end{proof}

\section{MVE and NFPE with critical singular kernels}\label{Sec:MVE}

\subsection{Existence}

Let us start with the existence part in \Cref{Thm:MV} $(i)$. 

\begin{proof}[Proof of \Cref{Thm:MV} (i)]
    Let 
	\[
	  K_1 := K\1_{\{|K|>1\}}, ~ K_2:= K\1_{\{|K|\leq 1\}}~ \mbox{ and }~ 
      K^n := K* \varrho^n, ~  K_i^n :=K_i*\varrho^n, 
	\] 
    where $\varrho^n$ is the mollifier as in the proof of  \Cref{Thm:SDE'}. 

	Consider the approximate McKean-Vlasov equation 
	\begin{equation*}
			\l\{
		\begin{aligned}
			\d X_t^n &= K^n(t, X_t^n-y) \rho^n(t)(\d y) + \sqrt{2}\d W_t\\
			\rho^n(t) &=\mathrm{law} (X_t^n), ~\rho^n(0)=\zeta\in \cP(\mR^d). 
		\end{aligned}
		\r.
	\end{equation*}
 Its well-posedness is standard, see, for instance, \cite{rockner2021DDSDE, zhao2024existence}. Set 
	\begin{equation}\label{eq:bin}
	    b^n(t,x):= \int_{\mR^d} K^n(t, x-y) \rho^n(t, \d y) ~\mbox{ and }~ b^n_i(t,x):= \int_{\mR^d} K_i^n(t, x-y) \rho^n(t, \d y),  
	\end{equation}
	where $i=1,2$. Then 
	\begin{equation*}
		\div b^n=0, \quad \sup_n \|b^n\|_{L^\infty_t L_x^{d,\infty}}\leq C  \|K^n\|_{L^\infty_t L_x^{d,\infty}} \leq C \|K\|_{L^\infty_t L_x^{d,\infty}}<\infty. 
	\end{equation*} 
 
	Arguing as in the proof of \Cref{Thm:SDE'} 
    and applying the Skorokhold representation theorem 
    we infer that there exist a probability space $(\Omega, \cF, \bP)$ and a sequence of random maps $(Y,B)$ and  
   $\{(Y^n,B^n)\}$ such that, 
   $B^n, B$ are Brownian motions, 
	\[
	    (X^n, W^n)\overset{d}{=}(Y^n, B^n),\ \  (X, W)\overset{d}{=}(Y, B),\ \  \mbox{and}\ \ 
	    (Y^n, B^n) \to (Y, B),\ \ 
     \bP-a.s.. 
	\]
    Moreover, for any $(p,q)\in \sI$,  
    the Krylov estimate holds 
    \begin{equation}\label{eq:kry3}
        \bE \int^{T}_{0} f(t,Y_{t}^n) \d t, ~ \bE \int^{T}_{0} f(t,Y_{t}) \d t  \leq C\| f\|_{L^q_t L_x^p}.
    \end{equation}
    
    In order to prove that the limit  \((Y,B)\) is a weak solution of \eqref{Eq:MV}, 
    as in the arguments below \eqref{eq:kry2} in the proof of \Cref{Thm:SDE'}, 
    it suffices to show that 
    \begin{align} \label{bn-b-limit}
         b^n_i \to b_i:=K_i*\rho \text{ in } L^{q_i}_t L_x^{p_i},\ \ i=1,2,
    \end{align} 
    where \(\rho(t)\) is the distribution of \(Y_t\) and $(p_i, q_i)\in \sI$ is as in Subsection \ref{Subsec-Exist-SDE}. 
    
    For this purpose, given a large number $N\in \mN$, by the triangle inequality 
    \begin{equation*}
	\begin{aligned}
        |b^n_1-b_1| \leq&  |(K_1-K^N_1)*\rho|+|(K^n_1-K^N_1)*\rho^n|+|K^N_1*\rho-K^N_1*\rho^n| \\
        =:&  I_{1}+I_{2}+I_{3}.
	\end{aligned}
    \end{equation*}
    Noting that \(K_1 \in L^\infty_t L^{p_1}_x\), we have 
    \[
        \lim_{N\to\infty} \|I_{1}\|_{L^{q_1}_t L_x^{p_1}}\leq \lim_{N\to\infty} \|K_1-K^N_1\|_{L^{q_1}_t L_x^{p_1}}=0 
    \]
    and 
    \[
        \lim_{N\to\infty} \limsup_{n\to\infty} \|I_{2}\|_{L^{q_1}_t L_x^{p_1}}\leq \lim_{N\to\infty} \limsup_{n\to\infty} \|K_1^n-K^N_1\|_{L^{q_1}_t L_x^{p_1}}=0. 
    \]
    Regarding \(I_3\), since $K_1^N\in C_b^\infty(\mR^{d+1})$ and \(\sup_{t\in[0,T]}|Y_t-Y_t^n|\to 0\), \(\bP\)-a.s, it holds that  
    \begin{equation}\label{eq:KYn-to-KY}
    \begin{aligned}
        &K^N_1*\rho(t,x)-K^N_1*\rho^n(t,x)\\
        =&\mathbf{E} [K_1^N(t,x-Y_t)-K_1^N(t, x-Y_t^n)]\to 0, 
        \quad {\rm as}\ n\to\infty, \quad t\in [0,T], x\in \mR^d. 
    \end{aligned}
    \end{equation}
    Moreover, thanks to \eqref{Eq:AE}, we have 
    \begin{equation}\label{eq:KYn-leq-F}
        |K_1^N*\rho^n|(t, x)\leq C \int_{\mR^d} |K_1^N(x-y)| \d y\int_{\mR^d} h(Ct, z-y)\zeta(\d z)=:F^{N}(t,x).  
    \end{equation}
    Note that 
    \begin{align*}
        \|F^{N}\|_{L^{q_1}_tL^{p_1}_x} \leq& C \|K_1^N\|_{L^{q_1}_tL^{p_1}_x} \|h(Ct)*\zeta\|_{L^\infty_tL^1_x} \leq C \|K_1\|_{L^{\infty}_tL^{p_1}_x}\\
        \leq& C\|K\|_{L^\infty_tL^{d,\infty}_x}^{\frac{d}{p_1}} \l( \int_1^\infty \lambda^{p_1-1-d} \d \lambda \r)^{\frac{1}{p_1}}\leq C, 
    \end{align*}
    which along with \eqref{eq:KYn-to-KY}, \eqref{eq:KYn-leq-F} and the dominated convergence theorem implies that 
    \[
    \lim_{n\to\infty} \|I_{3}\|_{L^{q_1}_t L_x^{p_1}} = 0. 
    \]
    
    Combining the above estimates for \(I_i\), i=1,2,3, we obtain 
    \[
    \lim_{n\to\infty}\|b^n_1-b_1\|_{L^{q_1}_t L_x^{p_1}}\leq  \lim_{N\to\infty} \limsup_{n\to\infty} \sum_{i=1}^3\|I_i\|_{L^{q_1}_tL^{p_1}_x}=0. 
    \]
    Similar arguments also apply to the $b_2^n$ component:  
    \[
    \lim_{n\to\infty}\|b^n_2-b_2\|_{L^{q_2}_t L_x^{p_2}} =0.  
    \]
    Thus,  we obtain the desired limit \eqref{bn-b-limit} and 
    finish the proof of the existence part.

\medskip 
	Now, let $X$ be the solution to \eqref{Eq:MV} obtained by the above approximation procedure.  
    Note that, the drift satisfies the condition \eqref{Eq:Ab}, and thus, by Theorem \ref{Thm:SDE}, for each \(t\in (0,T]\) the distribution of $X_t$ has the density (still denoted by $\rho(t)$) satisfying
	\[
	    \rho(t,y)=\int_{\mR^d} p^b_{t_0, t}(x,y) \rho({t_0},x)\d x, \quad  t_0\in (0,t), 
	\]
	where  the drift $b(t,x)=K*\rho(t,x)$, 
    and $p^b_{s,t}$ is the transition density corresponding to the backward Kolmogorov operator \(\p_t+\Delta+b\cdot \nabla\), which satisfies the Aronson estimate \eqref{Eq:AE}. 
    
    Furthermore, 
    the drift is bounded on positive time, that is, 
	\[
	    \|b\|_{L^\infty([t_0, T]\times \mR^d)}\leq C {t_0}^{-\frac{1}{2}} \|K\|_{L^\infty_tL^{d,\infty}_x}<\infty, \quad t_0\in (0,T). 
	\]
	Since for each $t_0\in(0,t)$, $p^{b}_{t_0, t}= p^{b\1_{[t_0, T]}}_{t_0, t}$ and $b\1_{[t_0, T]}$ is bounded, 
 in view of \cite[Theorem 2.3]{chen2017heat}, we get 
	\[
	    |\nabla \rho(t, y)| \leq \l| \int_{\mR^d}\nabla_y p_{t_0, t}^{b\1_{[t_0, T]}}(x,y) \rho_{t_0}(\d x) \r| \leq C(d, t_0,K) |t-t_0|^{-\frac{1}{2}}, \quad t_0<t\leq T. 
	\]
    An application of inductive arguments 
    then leads to   
	\[
	    \|\nabla^k \rho(t)\|_{L_x^\infty} \leq C(k,t)<\infty, \quad k\geq 0,\quad t\in (0,T]. 
	\] 
    This gives the regularity of the corresponding density.  
\end{proof}
  
    \begin{proof}[Proof of \Cref{Thm:NFP}(i)]
    Let \(K^n\) be the same kernel as in the proof of \Cref{Thm:MV}, and let \(\rho^n\) be the solution to \eqref{Eq:NFP} with \(K\) replaced by \(K^n\). Let \(b^n\) be given by \eqref{eq:bin} and  \(p^n\) the transition density function associated with the operator \(\Delta + b^n \cdot \nabla\). 
    
    We first claim that \((\rho^n)\) is relatively compact in \(C_t\cM_x\). 
    To this end, we note that 
    \[
        \|\rho^n(t)\|_{\cM}\leq C \l\| \int_{\mR^d}p^n_{0,t}(x,\cdot) |\zeta|(\d x) \r\|_{\cM}\leq \|\zeta\|_{\cM}, \quad t\in [0,T], 
    \]
    which implies that \((\rho^n(t))\) is  compact in \(\cM\), for each \(t\in [0,T]\). According to the Arzelà-Ascoli theorem, we only need to prove that \((\rho^n)\) is equicontinuous in \(C_t\cM_x\). For this purpose, taking any \(f\in C_0(\mR^d)\), we have 
    \begin{align*}
        \l\langle \rho_{t+\delta}^n-\rho^n(t), f \r\rangle  =& \iint_{\mR^{2d}} \l[ p^n_{0,t+\delta}(x,y)-p^n_{0,t}(x,y)\r] f(y) \, \zeta(\d x) \, \d y  \\
        =& \int_{\mR} \zeta(\d x) \int_{\mR^d} p^n_{0,t}(x,y)  \l( \int_{\mR^d} p^n_{t,t+\delta}(y,z) f(z) \,\d z -f(y)\, \r) \d y. 
    \end{align*}
    Since 
    \begin{align*}
        &\l|\int_{\mR^d} p^n_{t,t+\delta}(y,z) f(z)\d z-f(y)\r|\\
        \leq& \int_{|z-y|<\sqrt[3]{\delta}} |f(y)-f(z)| p^n_{t,t+\delta}(y,z)\, \d z + 2\|f\|_{L^\infty} \int_{|z-y|>\sqrt[3]{\delta}} p^n_{t,t+\delta}(y,z)\, \d z \\
        \leq & C \int_{|z-y|<\sqrt[3]{\delta}} |f(y)-f(z)| h(C\delta, y-z)\, \d z + C \|f\|_{L^\infty} \int_{|z-y|>\sqrt[3]{\delta}} h(C\delta, y-z)\, \d z \\
        \leq & C \osc_{B_{\sqrt[3]{\delta}}(y)} f + C \delta,  
    \end{align*}
    We get 
    \[
    \sup_{n} |\l\langle \rho_{t+\delta}^n-\rho^n(t), f \r\rangle| 
    \leq C \sup_{y\in\mR^d}\osc_{B_{\sqrt[3]{\delta}}(y)} f + C \delta  \to 0, 
    \]
    which yields the equicontinuity of \((\rho^n)\) , as claimed. 
    
    Now, let \(\rho\) be the limit of \((\rho^n)\) along a subsequence, which, for simplicity, we still denote by \(n\). Noting that \(\rho^n\) is bounded in \(L^{q}_tL^{p}_x\), for any \((p,q)\in \sI\), therefore, we also have \(\rho^n\rightharpoonup \rho\) in \(L^{q}_tL^{p}_x\). Then, 
    as in the proof of \eqref{bn-b-limit}, one can verify that 
    \[
        \int_0^t \langle \rho^n(s), K^n * \rho^n(s) \cdot \nabla \phi(s) \rangle \, ds \to \int_0^t \langle \rho(s), K * \rho(s) \cdot \nabla \phi(s) \rangle \, ds, \quad \phi\in C_c^\infty ([0,T]\times \mR^d), 
    \]
    which implies that \(\rho\) is a distributional solution to \eqref{Eq:NFP}.
\end{proof}
    
\subsection{Uniqueness} 
Next, we prove the uniqueness part in \Cref{Thm:MV}. Recall that $\zeta\in \cM(\mR^d)$ can be decomposed uniquely as \(\zeta= \zeta_c+\zeta_a\), where $\zeta_c$ is the continuous part, i.e.,  $\zeta_c(\{x\})=0$ for all $x\in \mR^d$, and $\zeta_a=\sum_{i}c_i \delta_{x_i}$ is the purely atomic part. 
Also recall that $h$ is the Gaussian kernel given by \eqref{Eq:H}. 

We shall use the following estimate for Gaussian kernel proved by Giga-Miyakawa-Osada \cite{giga1988two}. 
	\const{\Cheat}
	\begin{lemma} [\cite{giga1988two}] \label{Lem:Decay} 
        For any $r>1$ and $\beta\geq 1$, 
		\begin{equation*}
			\limsup_{t\to 0}t^{\frac{d}{2r'}}\| h(t)* \zeta \|_{L_x^{r,\beta}}\leq C_{\Cheat}(d, r) |\zeta_a|,  
		\end{equation*}
		where $|\zeta_a|$ is the total variation of $\zeta_a$ of $\mathbb{R}^d$, and  $r'$ is the conjugate number of $r$.  
	\end{lemma}

\begin{proof}[Proof of the uniqueness part in \Cref{Thm:MV}:]
    
    Suppose that $Y^{(1)}$ and $Y^{(2)}$ are two weak solutions to \eqref{Eq:MV} 
    with the density $\rho^{(1)}$ and $\rho^{(2)}$, respectively, 
    and satisfy the Krylov-type estimate \eqref{eq:kry_d}. Put $b^{(i)}=K*\rho^{(i)}$, $i=1,2$. 
    Since 
    \[
        \|b^{(i)}\|_{L^\infty_t L_x^{d,\infty}} \leq C(d) \|K\|_{L^\infty_t L_x^{d,\infty}}, 
    \]
    thanks to \Cref{Thm:SDE'},  
    we have the representation formula of solutions 
    \begin{equation}\label{eq:rho}
	      \rho^{(i)} (t,y) = \int_{\mR^d} p_{0, t}^{b^{(i)}}(x,y) \zeta (x) \d x, \quad i=1,2, 
    \end{equation}
    where $p^{b^{(i)}}_{0,t}$ is the transition density of the unique solution to \eqref{Eq:SDE} with the drift $b=b^{(i)}$ starting from the initial time $s=0$. 
    Moreover, by \eqref{eq:rho},  the Aronson-type estimate \eqref{Eq:AE} and \Cref{Lem:Decay}, there exists $T_0>0$ such that
        \const{\Cheat2}
     \begin{equation}\label{eq:decay}
         \sup_{t\in [0,T_0]} t^{\frac{d}{2r'}} \|\rho^{(i)}(t)\|_{L_x^{r,\beta}} \leq C_{\Cheat2} |\zeta_a|,  \quad \forall r>1,\ \beta\geq 1.   
    \end{equation}
    Set  
    \[
    \rho(t) :=\rho^{(1)}(t)-\rho^{(2)}(t)~\mbox{ and }~ b:=b^{(1)}-b^{(2)}.
    \]
    It follows from the Duhamel formula  
    of $\rho^{(i)}$, $i=1,2$, 
    that  
    \begin{equation}\label{eq:winteg}
	\begin{aligned}
         \rho(t)=&\int_0^t \nabla\cdot h(t-s)*(b^{(1)}(s)\,\rho(s)) \d s- \int_0^t \nabla\cdot h(t-s)*(b(s)\,\rho^{(2)}(s)) \d s \\
         =&:J_1(t)-J_2(t). 
	\end{aligned}
    \end{equation} 

    In light of Lemma \ref{Lem:Decay}, 
    we set the following norm as in \cite{giga1988two}
		\[
		    \|f\|_{r,T_0}:= \sup_{t\in [0,T_0]} t^{\frac{d}{2r'}} \|f(t)\|_{L_x^r} 
		\] 
    for any $r\in (1,d')$, 
		where $r'$ is the conjugate exponet of $r$, 
    which measures the propagation of heat flows. 
    
    By \eqref{eq:decay} and Young's inequality \eqref{Eq:Young2},  
    for any $s\in [0,T_0]$, 
		\begin{align*}
			\|b^{(1)}(s)\, \rho(s)\|_{L_x^r}\leq& \|b^{(1)}(s)\|_{L_x^\infty} 
            \|\rho(s)\|_{L_x^r} 
            \leq C \|\rho^{(1)}(s)\|_{L_x^{d',1}} \|\rho(s)\|_{L_x^r} \\
			\leq & C |\zeta_a| s^{-\frac{1}{2}} \|\rho(s)\|_{L_x^r}. 
		\end{align*}
		This yields that  
		\begin{equation}\label{eq:vorticityI}
			\begin{aligned}
				\|J_1\|_{r,T_0}\leq& C |\zeta_a| \sup_{t\in [0,T]} t^{\frac{d}{2r'}} \l[\int_0^t (t-s)^{-\frac{1}{2}} s^{-\frac{1}{2}(1+\frac{d}{r'})}\|\rho\|_{r,T_0}\, \d s\r]\\
				\leq & C B\l(\frac{1}{2}, \frac{1}{2}\l(1-\frac{d}{r'}\r)\r) |\zeta_a| \|\rho\|_{r,T_0}. 
			\end{aligned}
		\end{equation}
		Regarding the second term $J_2$, 
        using \eqref{eq:decay} and Young's inequality \eqref{Eq:Young1} again we derive 
		\begin{align*}
			\|b(s)\rho^{(2)}(s)\|_{L_x^r} \leq \|b(s)\|_{L_x^l} \|\rho^{(2)}(s)\|_{L_x^{d'}} 
            \leq C  |\zeta_a| s^{-\frac{1}{2}}  \|\rho(s)\|_{L_x^r} 
		\end{align*}
		with $r\in (1, d')$ and ${1}/{l}={1}/{r}+{1}/{d}-1$. Thus, we have 
		\begin{equation}\label{eq:vorticityJ}
			\begin{aligned}
				\|J_2\|_{r,T_0} \leq & |\zeta_a| \sup_{t\in [0,T]}  t^{\frac{d}{2r'}}  \l[ \int_0^t (t-s)^{-\frac{1}{2}} s^{-\frac{1}{2}(1+\frac{d}{r'})} \|\rho\|_{r,T_0}\, \d s\r]\\
				\leq & C B\l(\frac{1}{2}, \frac{1}{2}\l(1-\frac{d}{r'}\r)\r) |\zeta_a| \|\rho\|_{r,T_0}. 
			\end{aligned}
		\end{equation}
		
    Therefore, plugging \eqref{eq:vorticityI} and \eqref{eq:vorticityJ} into \eqref{eq:winteg} 
        and using the smallness of the initial mass we thus get 
		\const{\Crho}
		\[
		\|\rho\|_{r,T_0}\leq C_{\Crho} |\zeta_a| \|\rho\|_{r,T_0}\leq \eps_0C_{\Crho}\|\rho\|_{r,T_0}, 
		\]
    where $C_{\Crho}$ only depends on $d, r$ and $\|K\|_{L^\infty_tL^{d,\infty}_x}$. 
    It follows that 
    $\|\rho\|_{r,T_0}=0$ if $\eps_0<C_{\Crho}^{-1}$.

     Note that for any $\rho \in C_t \cP_x$, the convolution $K * \rho$ belongs to the critical Lorentz space $L^\infty_t L^{d,\infty}_x$ and satisfies the divergence free condition $\div(K * \rho) = 0$. In view of the uniqueness result for \eqref{Eq:SDE} in \Cref{Thm:SDE'}, 
    for solutions in the Krylov class, the uniqueness of the path law for \eqref{Eq:MV} is a consequence of the uniqueness of the marginal distributions. 
    
    Therefore, the proof of  \Cref{Thm:MV} is complete.  
    \end{proof}

\medskip

        The uniqueness part of \Cref{Thm:NFP} can be proved in an analogous manner as that of \Cref{Thm:MV}, with the modification that \Cref{Prop:LFP} is used in place of \Cref{Thm:SDE}. So the details are omitted here.

\medskip

Next, we present the proof for \Cref{Thm:NS}. 
\begin{proof}[Proof of \Cref{Thm:NS}]  
    The uniqueness with smallness condition follows immediately from Theorem \ref{Thm:NFP}. In the case without smallness condition, by virtue of \Cref{Prop:LFP}, for any distributional solution  \(\rho\) to \eqref{Eq:NS} with \(\rho(0) = \zeta\in \cM(\mR^2)\) and satisfying  \eqref{uniq-supercritical}, one still has the following representation formula:  
    \begin{equation}\label{eq:rho-Inter2}
        \rho(t, y) = \int_{\mR^2} p^{b}_{0, t}(x, y) \zeta(x) \d x, 
    \end{equation}
    where \(b = K_{\mathrm{BS}} * \rho\), and $p^{b}_{0,t}$ is the transition density of the unique solution to \eqref{Eq:SDE} with $s=0$. Using this representation formula and the Aronson-type estimate \eqref{Eq:AE} 
we derive 

        \[
        \|\rho(t)\|_{L^\infty} \leq C \|h(Ct,\cdot)\|_{L^\infty} |\mu| \leq \frac{C}{t}|\zeta|, \quad t\in (0,T] 
        \]
        and 
        \[
        \|\rho(t)\|_{L^1} \leq C \|h(Ct,\cdot)\|_{L^1} |\zeta| \leq C |\zeta|<\infty, \quad t\in [0,T]. 
        \]
        Moreover, for each \(t>0\), \eqref{Eq:pHolder} implies that \(\rho(t)\in C_0(\mR^2)\). 
Thus, we infer that \(\rho\in C((0,T];L^1\cap L^\infty)\) and the solution lives in the uniqueness solution class in \cite{Isabelle2005uniqueness}. 
Therefore, by virtue of \cite[Theorem 1.1]{Isabelle2005uniqueness}, 
we obtain the desired uniqueness assertion.  
    \end{proof}

    Our result for the 2D vorticity NSE 
    also implies the uniqueness 
    for the original 2D NSE in the velocity formulation.  
    This is the content of Corollary \ref{Cor:NS} below. 
    
    \begin{corollary}\label{Cor:NS} 
     Let \((\mathfrak{p}, \mathfrak{q})\in \sI\) be the pair as in \Cref{Thm:NS}.
        Suppose that \(u^{(1)}\) and \(u^{(2)}\) are two distributional solutions 
        to the $2$D NSE 
        \begin{equation*}
            \partial_t u+(u \cdot \nabla) u=\Delta u-\nabla p, \quad \div u=0, \quad u|_{t=0}=u_0. 
        \end{equation*}
        and that \(u^{(i)}\in L^{\mathfrak{q}'}_t\dot{W}^{1,\mathfrak{p}'}_x\), \(u^{(i)}(t)\in L^{2,\infty}_x\) and \(\nabla\times u^{(i)}:= \partial_1 u_2^{(i)}-\partial_2 u_1^{(i)}\in C_t\cM_x\), 
        $i=1,2$. 
        Then \(u^{(1)} = u^{(2)}\). 
    \end{corollary}

    We need the following auxiliary lemma. 
    \begin{lemma}\label{Lem:rep}
        Suppose that \(F\) is a divergence-free vector field in \(L^{2,\infty}(\mR^2;\mR^2)\) and \(\nabla\times F\in \cM\). Then 
        \[
        F= K_{\mathrm{BS}}*(\nabla \times F), 
        \] 
        where $K_{\mathrm{BS}}$ is the Biot-Savart law given by \eqref{eq:BS}. 
    \end{lemma}
    
    \begin{proof}
        Let \(G:=K_{\mathrm{BS}}*(\nabla \times F)\in L^{2,\infty}\) and \(H:=F-G\). 
        Then, 
        \(\div H=0\) and \(\nabla\times H=0\) in the sense of distribution. This implies 
        \begin{align*}
            0=&\partial_1(\div H)-\partial_2(\nabla\times H) = \Delta H_1, \\
            0=&\partial_2(\div H)+\partial_1(\nabla\times H)=\Delta H_2
        \end{align*}
        in the sense of distribution. Noting that \(H\in L^{2,\infty}\subseteq L_{loc}^1\), by the mean value theorem for harmonic functions, for each \(x\in \mR^2\),  we derive 
        \begin{align*}
            |H(x)|=&\frac{1}{|B_R|}\l|\int_{B_R(x)} H(y)\d y\r|\leq C R^{-2} \int_0^\infty |\{y\in B_R(x): |H(y)|>t\}| \d t\\
                \leq& C R^{-2} \int_0^\infty (t^{-2} \wedge R^2) \d t \leq CR^{-1}\to 0 
                \quad {\rm as}\ R\to \infty, 
        \end{align*}
        which means \(F=G=K_{\mathrm{BS}}*(\nabla \times F)\). 
    \end{proof}
    
    \begin{proof}[Proof of \Cref{Cor:NS}]
        We omit the superscript \(i\) below for simplicity. 
        Let 
        \[
        \rho :=\nabla\times u\in L^{\mathfrak{q}'}_tL^{\mathfrak{p'}}_x\cap C_t\cM_x.
        \] 
        In view of \Cref{Lem:rep}, 
        we have \(u= K_{\mathrm{BS}}*\rho\), 
        where $K_{\mathrm{BS}}$ is the Biot-Savart law given by \eqref{eq:BS}. 
        
        For any \(\phi\in C_c^\infty([0,T]\times\mR^2)\), set \(\varphi=\nabla^{\perp}\phi\). Then \(\varphi\) is divergence-free. The integration-by-parts formula yields that   
        \[
        \<\rho, \phi\>=\<\nabla\times K_{\mathrm{BS}}*\rho, \phi\>=-\<u, \nabla^{\perp}\phi\>=-\<u,\varphi\> 
        \]
        and
        \[
        \<u\rho,\nabla \phi\>=\<\nabla\times u, u\cdot\nabla \phi\>=-\<u, \nabla^{\perp} (u\cdot\nabla \phi)\>=-\int u_i u_j \p_j \varphi_i. 
        \]
        Moreover, 
        by the equation of $u$, 
        for almost every \(s,t\in (0, T)\), 
        \begin{align*}
            &\<u(t),\varphi(t)\>-\<u(s), \varphi(s)\>\\
            =& \int_s^t\<u(r), \p_r \varphi(r)+\Delta\varphi(r)\>\d r+ \int_s^t \d r\int_{\mR^2} u_i(r) u_j(r)\p_j\varphi_i(r). 
        \end{align*}
        Thus, combining the above calculations together we derive 
        \[
        \<\rho(t),\phi(t)\>-\<\rho(s),\phi(s)\>= \int_s^t \<\rho(r), \p_r \phi(r)+\Delta\phi(r)+u(r)\cdot\nabla \phi(r)\>\d r, 
        \]
        for almost every \(s,t\in (0,T)\). Taking into account \(\rho\in C_t\cM_x\), we then obtain 
        \[
        \<\rho(t),\phi(t)\>-\<\nabla\times u_0,\phi(0)\>= \int_0^t \<\rho(s), \p_s \phi(s)+\Delta\phi(s)+u(s)\cdot\nabla \phi(s)\>\d s, \quad t\in [0,T]. 
        \] 
That is, $\rho$ satisfies the $2$D vorticity NSE. 
        
        Therefore, by virtue of \Cref{Thm:NS}, we obtain \(\nabla\times u^{(1)}=\nabla \times u^{(2)}\), which along with \Cref{Lem:rep} yields the desired uniqueness $ u^{(1)}= u^{(2)}$. 
    \end{proof}
    
\section{Nonlinear Markov processes} \label{Sec:NMP}

Let us first review the notions from \cite{rehmeier2022nonlinear}. 
Let \(\Omega_s:= C([s,T]; \mR^d)\) 
be the set of all continuous paths in \(\mR^d\) starting from time s equipped with the topology of locally uniform convergence, 
and \(\cB(\Omega_s\)) the corresponding Borel \(\sigma\)-algebra. Define for \(0\leq s\leq r\leq t\leq T\), 
\[
     \pi^s_t: \Omega_s\to \mR^d; \quad \pi^s_t(\omega)=\omega(t), ~ \omega\in \Omega_s,
\]
\[
\cF^s_t:= \sigma (\pi^s_r: s\leq r\leq t), 
\]
and 
\[
\theta^s_r:\Omega_s\to\Omega_r; \quad (\theta^s_r\omega)(t)=\omega(t). 
\]
\begin{definition}[Nonlinear Markov Process \cite{rehmeier2022nonlinear}]\label{Def:NMP}
    Let $\cQ \subseteq \mathcal{P}(\mR^d)$. A nonlinear Markov process is a family  of probability measures $\left(\mathbb{P}_{s, \zeta}\right)_{(s, \zeta) \in [0,T] \times \cQ}$ on $\mathcal{B}\left(\Omega_s\right)$ such that
    \begin{enumerate}[(i)]
        \item For all $0 \leq s \leq r \leq t$ and $\zeta \in \cQ$, the marginals $\mu_t^{s, \zeta}:=\mathbb{P}_{s, \zeta} \circ\left(\pi_t^s\right)^{-1}$ belong to $\cQ$.
        \item The nonlinear Markov property holds, i.e. for all $0 \leq s \leq r \leq t\leq T, \zeta \in \cQ$ 
        \begin{equation*}
            \begin{aligned}
                \mathbb{P}_{s, \zeta}\left(\pi_t^s \in A \mid \mathcal{F}_{r}^{s}\right)(\cdot)=p_{(s, \zeta),\left(r, \pi_r^s(\cdot)\right)}\left(\pi_t^r \in A\right), \quad \mathbb{P}_{s, \zeta}\text{-a.s.} \text { for all } A \in \mathcal{B}(\mathbb{R}^d)
            \end{aligned}
        \end{equation*}
        where $p_{(s, \zeta),(r, y)}, y \in \mathbb{R}^d$, is a regular conditional probability kernel from $\mathbb{R}^d$ to $\mathcal{B}\left(\Omega_r\right)$ of $\mathbb{P}_{r, \mu_r^{s, \zeta},}\left[\cdot \mid \pi_r^r=y\right], y \in \mathbb{R}^d$. 
    \end{enumerate}
\end{definition}

\begin{proof}[Proof for \Cref{Thm:NMP}]
    We only provide the proof of (i), as the proof of (ii) follows identically. 
    
    Fix \(s\geq 0\) and \(\zeta\in \cP(\mR^d)\). Let \( \mP_{s,\zeta} \) be one of the limiting distributions in \( \cP(\Omega_s) \) of the approximation processes constructed in \Cref{Sec:MVE} with initial time \( s \). Then \(\mP_{s,\zeta}\) is a martingale solution to \eqref{Eq:MV} (with initial time \(s\)). Set
    \[
       \mu^{s,\zeta}_t 
:= \mathbb{P}_{s, \zeta} \circ (\pi^s_r)^{-1}, \quad s\leq t\leq T. 
    \] 
    Then, by \eqref{eq:kry3}, we see that 
    \[
        \int_s^T\!\!\!\int_{\mR^d}f(t,x)\mu^{s,\zeta}_t(\d x)\d t = \mE_{s,\zeta} \int_s^T f(t, \pi^s_t) \d t\leq C \|f\|_{L^q(s,T; L^p_x)}, 
    \]
    for any $(p,q)\in \sI$, namely, \((\mu^{s,\zeta})\) belongs to the Krylov class. Thanks to \Cref{Thm:NFP}, $\mu^{s,\zeta}$ is uniquely determined when $\zeta\in \cP_{\eps_0}$. 

    In order to prove the desired assertions,  
    by virtue of \cite[Corollary 3.9]{rehmeier2022nonlinear}, 
    we shall verify the following two conditions:
    \begin{enumerate}[(1)]
        \item[(a)] \(\{\mu^{s,\zeta}\}_{(s,\zeta)\in [0,T]\times\cP(\mR^d)}\) satisfies the flow property  
        \eqref{eq:flow} with \(\cQ=\cP(\mR^d)\); 
        \item[(b)] 
        \(\mu^{s,\zeta}\) is the unique solution  to the linearized Fokker-Planck equation 
        \begin{equation}\label{eq:linearized}
            \partial_t\rho=\Delta\rho-\div((K*\mu^{s,\zeta})\rho), \quad \rho|_{t=s}=\zeta, 
        \end{equation}
        that satisfies
        \begin{equation}\label{eq:rho-leq-mu}
            \rho(t)\leq C\mu^{s,\zeta}_t, \quad s\leq t\leq T,  
        \end{equation}
    \end{enumerate}
    where \(C>0\) is a constant. 
    
    To this end, we note that for any $r\in [s,T]$, the path $(\mu^{s,\zeta}_t)_{t \in [r,T]}$ solves \eqref{Eq:NFP} with initial datum $(r, \mu^{s,\zeta}_r)$. In view of \Cref{Thm:NFP}, for any $r\in (s,T]$,
    $\mu^{s,\zeta}_r$ admits a bounded density which implies that $\mu^{s,\zeta}_r\in \mathcal{P}_{\varepsilon_0}$. Therefore, invoking \Cref{Thm:NFP} once again, the flow property \eqref{eq:flow} is verified. 
    
    Moreover, 
    let 
    \[
    b(t,x) := \int_{\mathbb{R}^d} K(t,x-y) \mu^{s,\zeta}_t(\mathrm{d}y), \quad s\leq t\leq T.
    \]
    Then $b \in L^\infty([s,T]; L^{d,\infty}_x)$ and $\div\, b = 0$. Since $(\mu^{s,\zeta})$ lie in the Krylov class, 
    any curve of probability measures satisfying \eqref{eq:rho-leq-mu} also belongs to the Krylov class. Thus, 
    in view of the well-posedness result in \Cref{Prop:LFP}, 
    the linear Fokker-Planck equation \eqref{eq:linearized} admits a unique probabilistic solution satisfying \eqref{eq:rho-leq-mu}, which verifies condition $(b)$. 
    
    Therefore, the proof is complete.
\end{proof}

\appendix 
\section{Lorentz spaces}\label{App:Lorentz}
\setcounter{equation}{0}
\renewcommand\theequation{A.\arabic{equation}}

This section contains several useful properties of Lorentz spaces and corresponding interpolation estimates.  
For more detailed explanations, we refer to the nice monograph \cite{grafakos2008classical}.  
 
	\begin{definition}
		The Lorentz space is the space of complex-valued measurable functions $f$ on a measure space $(X,\mu)$ such that the following quasinorm is finite
		\[
		\|f\|_{L^{p,q}(X,\mu )}= p^{\frac{1}{q}} \l\{ \int_0^\infty t^{q-1} \l[\mu(\{x: |f(x)|>t\}) \r]^{\frac{q}{p}} ~ \d t \r\}^{\frac{1}{q}}, 
		\]
		where $0< p,q \leq  \infty$.  When $q=\infty$, we set 
		$$
		\|f\|_{{L^{{p,\infty }}(X,\mu )}} =\sup _{{t>0}} t \l[\mu ( \left\{x:|f(x)|>t\right\}) \r]^{\frac{1}{p}}.
		$$
		It is also conventional to set $L^{\infty,\infty}(X, \mu) = L^{\infty}(X, \mu)$. The space  $L^{p, q}$ is contained in $L^{p, r}$ whenever  $q<r$. 
	\end{definition}
	
	\begin{proposition}\label{Prop:Lorentz}
		\begin{enumerate}[(i)]
			\item Assume that $0< p, q\leq \infty$ and $\theta>0$, then 
			\begin{equation}\label{Eq:ftheta}	\|f^\theta\|_{L^{p,q}}=\|f\|_{L^{\theta p, \theta q}}^\theta. 
			\end{equation}
			\item (H\"older's inequality) Assume that $1\leq p_1, p, p_2\leq \infty$ and $1\leq q_1, q_2\leq \infty$, then 
			\begin{equation}\label{Eq:Holder}
				\|fg\|_{L^{p,q}} \leq C \|f\|_{L^{p_1,q_1}} \|g\|_{L^{p_2, q_2}}, 
			\end{equation}
			where 
			\[
			    \frac{1}{p}=\frac{1}{p_1}+\frac{1}{p_2} ~\mbox{ and }~  \frac{1}{q}=\frac{1}{q_1}+\frac{1}{q_2}. 
			\]
			\item (Young's inequality) Assume that $1<p_1, p, p_2<\infty$, $1\leq q_1, q_2\leq \infty$, then 
			\begin{equation}\label{Eq:Young1}
				\|f*g\|_{L^{p,q}}\leq C \|f\|_{L^{p_1,q_1}} \|g\|_{L^{p_2, q_2}}, 
			\end{equation}
			where 
			\[
			    1+\frac{1}{p}=\frac{1}{p_1}+\frac{1}{p_2}, \quad  \frac{1}{q}\leq \frac{1}{q_1}+\frac{1}{q_2}; 
			\]
            and if $1<p_1, p_2<\infty, 1\leq q_1, q_2 \leq \infty$, then
            \begin{equation}\label{Eq:Young2}
                \|f * g\|_{L^{\infty}} \leq C\|f\|_{L^{p_1, q_1}}\|g\|_{L^{p_2, q_2}},
            \end{equation}
            where 
            \[
            \frac{1}{p_1}+\frac{1}{p_2}=1  ~\mbox{ and }~  \frac{1}{q_1}+\frac{1}{q_2}\geq 1. 
            \]
			\item 
			Assume that $1<p\leq \infty$ and $1\leq q\leq \infty$, then 
			\begin{equation}\label{Eq:Max}
			    \|f\|_{L^{p,q}}\leq \|\cM f\|_{L^{p,q}}\leq C \|f\|_{L^{p,q}}, 
			\end{equation}
			where $\cM f$ is the Hardy-Littlewood maximal function of $f$. 
		\end{enumerate}
	\end{proposition} 
 
   \begin{proof}
   	    The proofs of (i)-(iii) can be found in \cite{oneil1963convolution, grafakos2008classical}. For (iv), since 
   	    \begin{equation*}
   	    	\| \cM f \|_{L^{1,\infty}} \leq C \|f\|_{L^1} ~~\mbox{ and }~~ \| \cM f \|_{L^{\infty}} \leq \|f\|_{L^\infty}, 
   	    \end{equation*}
   	    by the Marcinkiewicz interpolation theorem for Lorentz spaces (see \cite{hunt1965extension}), one has 
   	    \begin{equation*}
   	    	\|\cM f\|_{L^{p,q}} \leq C_p \|f\|_{L^{p, q}} 
   	    \end{equation*} 
        for any $1<p\leq \infty$. 
   \end{proof}
	
	For any $s\in (0,d)$. The Riesz potential $I_sf$ of a locally integrable function $f$ on $\mR^d$ is defined by
	\[
		(I_{s}f)(x) :=K_s* f(x)={\frac {1}{c_{d,s}}}\int _{\mathbb {R} ^{d}}{\frac {f(y)}{|x-y|^{d-s }}}\,\mathrm {d} y. 
	\]
	We have 
	\begin{lemma} 
    [Boundedness of Riesz potential]
    \label{Lem:Sob}
		Let $0<s<d$ and $p\in (1, d/s)$. 
  There exists $C>0$ such that for any $f\in L^p$, 
		\begin{equation}\label{Eq:Sob}
			\|I_s f \|_{L^{q,p}}\leq C \|f\|_{L^p}, 
		\end{equation}
		where $\frac{1}{q}=\frac{1}{p}-\frac{s}{d}$.  
	\end{lemma}
 
	\begin{proof}
		Noting that 
		\[
		K_{s}(x)={\frac  {1}{c_{d, s}}}{\frac  {1}{|x|^{{d-s }}}}\in L^{\frac{d}{d-s},\infty} 
		\] 
	and	using Theorem 2.6 of \cite{oneil1963convolution} one has  
		\[
		\|I_s f \|_{L^{q,p}}\leq C \|K_s\|_{L^{\frac{d}{d-s},\infty}} \|f\|_{L^p}\leq C \|f\|_{L^p}. 
		\]	
	\end{proof}

\section{Interpolation Inequalities}\label{App:Inter}
\setcounter{equation}{0}
\renewcommand\theequation{B.\arabic{equation}}

Below we collect several useful  interpolation estimates in Lorentz spaces. 

\begin{lemma}\label{Lem:Inter1}
		Let $1\leq p_1<p<p_2< \infty$, $q\geq 1$ and 
		\[
		    \frac{1}{p}=\frac{\theta}{p_1}+\frac{1-\theta}{p_2}. 
		\]
		Then 
		\begin{equation}\label{Eq:Inter1}
			\|f\|_{L^{p,q}} \leq C \|f\|_{L^{p_1,\infty}}^{\theta} \|f\|_{L^{p_2,\infty}}^{1-\theta}. 
		\end{equation}
	\end{lemma}
	\begin{proof}
		Let $N_1=\|f\|_{L^{p_1,\infty}}$ and $ N_2=\|f\|_{L^{p_2,\infty}}$. Then 
		\[
		  \l|\{f>t\}\r| \leq \frac{N_1^{p_1}}{t^{p_1}} ~ \mbox{ and }~ \l|\{f>t\}\r| \leq \frac{N_2^{p_2}}{t^{p_2}}. 
		\]
		Thus, by definition, we derive 
		\begin{equation*}
			\begin{aligned}
				\|f\|_{L^{p,q}}=& p^{\frac{1}{q}} \l( \int_0^\infty t^{q-1} \l| \{|f|>t\}\} \r|^{\frac{q}{p}} ~ \d t \r)^{\frac{1}{q}}\\
				=& p^{\frac{1}{q}} \l( \int_0^\infty t^{q-1} \l(N_1^{p_1} t^{-{p_1}} \wedge N_2^{p_2} t^{-{p_2}}\r)^{\frac{q}{p}} \d t \r)^{\frac{1}{q}}\\
				\leq& p^{\frac{1}{q}} \l[ N_1^{qp_1/p} \int_{0}^{\l(\frac{N_2^{p_2}}{N_1^{p_1}}\r)^{\frac{1}{p_2-p_1}}} t^{q-1-\frac{qp_1}{p}} \d t +N_2^{qp_2/p} \int_{\l(\frac{N_2^{p_2}}{N_1^{p_1}}\r)^{\frac{1}{p_2-p_1}}}^{\infty} t^{q-1-\frac{qp_2}{p}} \d t \r]^{\frac{1}{q}}\\
				\leq & C(p,p_1,p_2,q) N_1^\theta N_2^{1-\theta}, 
			\end{aligned}
		\end{equation*}
		where $\theta=\frac{1/p-1/p_2}{1/p_1-1/p_2}$. 
	\end{proof}

	Recall that $\Lambda = (-\Delta)^{1/2}$. For any $s\geq 0$, $1\leq p<\infty$ and $q>0$, the homogeneous Lorentz-Bessel space is defined by 
	\[
	  \dot H^{s}_{p,q}:= \{f\in \sS'(\mR^d): \Lambda^s f\in L^{p,q}\},
	\]
	and 
	\[
	    \|f\|_{\dot H^{s}_{p,q}}:= \|\Lambda^s f\|_{L^{p,q}}.
	\]
	
	\smallskip
	
	The following fractional Gagliardo-Nirenberg type estimate in Lorentz spaces involving Besov-H\"older norms is useful in the proof of \Cref{Thm:LqLp}. 
	\begin{proposition}[Fractional Gagliardo-Nirenberg type estimate]\label{Prop:GNI}
    Let $1<p, p_1<\infty$, $1\leq q, q_1 \leq \infty$, $0<\alpha<\sigma\leq s<\infty$ and 
			\[
			\theta= \frac{\sigma-\alpha}{s-\alpha} = \frac{p_1}{p}=\frac{q_1}{q}\in (0,1]. 
			\]
			Then 
			\begin{equation}\label{Eq:NE}
				\begin{aligned}
					\|\Lambda^\sigma u\|_{L^{p,q}} \leq C \|\Lambda^s u\|_{L^{p_1, q_1}}^\theta \|u\|_{\dot{B}^\alpha_{\infty,\infty}}^{1-\theta}, \quad \forall u\in \dot{H}^{s}_{p_1, q_1}\cap \dot{B}^\alpha_{\infty,\infty}, 
				\end{aligned}
			\end{equation} 
			where $\dot B^\alpha_{\infty,\infty}$ is the  homogeneous Besov space (see    \cite{bahouri2011fourier}). 
	\end{proposition}
	\begin{remark}
			\Cref{Prop:GNI} implies that if $j, m\in \mN$, $0<j<m$, and $1<p_1<p<\infty$ such that $\alpha= {(jp-mp_1)}/{(p-p_1)}\in (0,1)$, then 
			\begin{equation}\label{eq:gni}
				\|\nabla^j u\|_{L^p} \leq C\|\nabla^m u\|_{L^{p_1}}^{\theta} [u]_{\alpha}^{1-\theta} ~ \mbox{ with }\theta={p_1}/{p}\in (0,1),  
			\end{equation}
            where $[u]_{\alpha}$ is the H\"older seminorm of $u$. Estimate \eqref{eq:gni} was first proved by Nirenberg in \cite{nirenberg1966extended}. 
	\end{remark}
	
    \begin{proof}[Proof of \Cref{Prop:GNI}] 
	We only need to consider the case that $\sigma<s$, i.e., $\theta\in (0,1)$. Let $\dot{\Delta}_j$ and $\dot{\Delta}'_j$ denote the homogeneous dyadic blocks (see \cite{bahouri2011fourier}) given by cutoff functions $\varphi$ and $\varphi'$, respectively, and $\varphi=\varphi \varphi'$. Then 
			\begin{align*}
				\dot{\Delta}_j \Lambda^\sigma u =  (\Lambda^\sigma \dot{\Delta}_j) \dot{\Delta}'_j u = 2^{j\sigma} \phi_j* (\dot{\Delta}'_j u), 
			\end{align*}
			where 
            \[
                \phi^\sigma:= \sF^{-1}(|\cdot|^\sigma \varphi) ~\mbox{ and }~ \phi_j^\sigma(x)=2^{dj} \phi^\sigma(2^j x).
            \]
            For the low-frequency part, by the above identity, 
			\begin{align*}
				\l| \sum_{j\leq k} \dot{\Delta}_j \Lambda^\sigma u(x) \r|  \leq C\sum_{j\leq k}  2^{j\sigma} \|\dot{\Delta}'_j u\|_\infty \leq C\sum_{j\leq k} 2^{j(\sigma-\alpha)} \|u\|_{\dot{B}^\alpha_{\infty,\infty}}. 
			\end{align*}
			For the high-frequency part, we have 
			\begin{equation}\label{eq:high}
				\begin{aligned}
					\left| \sum_{j>k} \dot{\Delta}_j \Lambda^{\sigma} u(x)\right| \leq& \sum_{j>k}\left|(\Lambda^{\sigma-s}\dot{\Delta}_{j})\Lambda^{s} u(x)\right|= \sum_{j > k} 2^{-j(s-\sigma)} \l|\phi^{\sigma-s}_j*\left(\Lambda^{s} u\right)(x)\r| \\
                    \leq& C\left(\sum_{j > k} 2^{-j(s-\sigma)}\right)\left(\cM\Lambda^{s} u\right)(x) \leq C 2^{-k(s-\sigma)}\left(\cM\Lambda^{s} u\right)(x), 
				\end{aligned}
			\end{equation}
            where we also used the fact that $|\phi^{\alpha}_j * f(x)| \leq C(\varphi, d, \alpha) \cM f(x)$.  
            Thus, we obtain  
			\begin{equation}\label{eq:lambda1}
				\left|\Lambda^{\sigma} u(x)\right| \leq C 2^{k (\sigma-\alpha)} \|u\|_{\dot{B}^\alpha_{\infty,\infty}}+C 2^{-k(s-\sigma)}\left(\cM\Lambda^{s} u\right)(x). 
			\end{equation}
			Choosing 
			\[
			k\approx (s-\alpha)^{-1}\log_2 \l( \frac{(\cM \Lambda^s u)(x)}{\|u\|_{\dot{B}^\alpha_{\infty,\infty}}}\r) 
			\] 
			and using \eqref{eq:lambda1},  we obtain that 
			\[
			|\Lambda^\sigma u(x)| \leq [(\cM\Lambda^s u)(x)]^{\theta} \|u\|_{\dot{B}^\alpha_{\infty,\infty}}^{1-\theta}, \quad \theta= \frac{\sigma-\alpha}{s-\alpha}. 
			\]
			Therefore, by \eqref{Eq:ftheta} and \eqref{Eq:Max}, and noting that $\theta p=p_1>1$ and $\theta q=q_1\geq 1$, 
			\begin{equation*}
				\begin{aligned}
					\|\Lambda^\sigma u\|_{L^{p,q}} \leq& \|\cM\Lambda^s u\|_{L^{\theta p,\theta q}}^\theta \|u\|_{\dot{B}^\alpha_{\infty,\infty}}^{1-\theta}\\
                    \leq& \|\Lambda^s u\|_{L^{p_1, q_1}}^\theta \|u\|_{\dot{B}^\alpha_{\infty,\infty}}^{1-\theta} ~ \mbox{ with }~ \theta= \frac{\sigma-\alpha}{s-\alpha}=\frac{p_1}{p}=\frac{q_1}{q}, 
			\end{aligned}
			\end{equation*}  
    which yields \eqref{Eq:NE}. 
	\end{proof}

    The following Lady\v{z}enskaja-type estimate for Lorentz space plays a crucial role in the proof of our main results when $d=2$. 
	\begin{lemma} [Lady\v{z}enskaja-type estimate]
		Let $d\geq 2$. For any $u\in W^{1,2}(\mR^d)$, it holds that 
		\begin{equation}\label{Eq:Lady}
			\|u\|_{L^{\frac{2d}{d-1},2}(\mR^d)} \leq C\|\nabla u\|_{L^2(\mR^d)}^{\frac{1}{2}} \|u\|_{L^2(\mR^d)}^{\frac{1}{2}}. 
		\end{equation}
	\end{lemma}

	\begin{proof}
		Using \eqref{Eq:Sob}, we derive 
		\begin{align*}
			\|u\|_{L^{\frac{2d}{d-1},2}(\mR^d)} {\leq}& C \|\Lambda^{\frac{1}{2}} u\|_{L^2(\mR^d)}
   {\leq}C\|\Lambda  u\|_{L^2(\mR^d)}^{\frac{1}{2}} \|u\|_{L^2(\mR^d)}^{\frac{1}{2}}, 
   \end{align*} 
   which yields  \eqref{Eq:Lady}. 
	\end{proof} 
 
    The following lemma contains the refined version of Poincar\'e's estimates. Let 
    \[
		\fint_{A} f : = \frac{1}{|A|} \int f ~\mbox{ and }~  \avert f \avert_{L^p(A)}:= \l( \fint_{A} |f|^p \r)^{\frac{1}{p}}. 
	\]
	
    \begin{lemma}[Poincar\'e-type estimates]
    	Let $0\leq \varphi \in C_c^\infty(B_1)$ such that $\int_{B_1} \varphi >0$. Given $R>0$, set $\varphi_R=\varphi(\cdot/R)$. Let 
    	\[
    	\underline{u}= \l(\int_{\mR^d} \varphi_R\r)^{-1} \int_{\mR^d} u\varphi_R~~\mbox{ and }~~  \bar{u}= u- \underline{u}. 
    	\] 
        Then, the following holds: 
        \begin{enumerate}[(i)]
            \item For any \(q\in [1,\infty)\), 
            \begin{equation}\label{Eq:Poincare1}
                \avert\bar u\avert_{L^q(B_R)} \leq C R \avert\nabla u\avert_{L^q(B_R)}.
            \end{equation}
            \item For any $q\in [1,d)$ and $p\in [q, \frac{dq}{d-q}]$, \begin{equation}\label{Eq:Poincare2}
    	    \avert\bar u\avert_{L^p(B_R)} \leq C R \avert\nabla u\avert_{L^q(B_R)}.  
    	\end{equation}
        \end{enumerate}
        Here the constant $C$ only depends on $d, p, q$ and $\varphi$.  
    \end{lemma}

    \begin{proof}
    	 By scaling, we may assume that $R=1$. For any \(q\in [1,\infty)\) and \(p\in [q, \frac{dq}{d-q}]\), using  Sobolev's estimate, we have 
    	\[
    	\|\bar u\|_{L^p(B_1)}\leq C \|\bar u\|_{W^{1,q}(B_1)}\leq C \|\bar u\|_{L^q(B_1)}+ C \|\nabla \bar u\|_{L^q(B_1)}. 
    	\]
    	Thus, it suffices to prove 
    	\[
    	\|\bar u\|_{L^q(B_1)}\leq C \|\nabla u\|_{L^q(B_1)}, \quad q\in [1,\infty). 
    	\] 
    	Arguing as in the proof of \cite[Theorem 1 in Section 5.8.1]{evans2010partial}, and assuming that the above estimate does not hold, we infer that there exists a sequence $u_k$ such that 
    	\[
    	\underline{u}_k = 0, \quad \left\| {u}_k \right\|_{{L^q}(B_1)} = 1 \ \mbox{ and }\  \|  \nabla u_k \|_{L^{q} (B_1)} \leqslant \frac{1}{k}.
    	\]
    	Thus, there exist a subsequence ${\{ {u_{{k_j}}}\} _{j \geqslant 1}} \subseteq {\{ {u_k}\} _{k \geqslant 1}}$ and $u \in L^q(B_1)$ such that $ u_{k_j} \to u$ in $L^q(B_1)$ and $\underline{u}=0$. Moreover, $\nabla u_{k_j}\rightharpoonup \nabla u$ in $L^q(B_1)$ and 
    	\[
    	\|\nabla u\|_{L^q(B_1)} \leq \liminf_{j\to\infty} \|\nabla u_{k_j}\|_{L^q(B_1)} =0.  
    	\] 
    Therefore, \(u\) is a constant in \(B_1\). 
    Taking into account $\underline{u}=\int_{B_1} u \varphi=0$, we infer that  $u=0$, which however contradicts the fact that $\|u\|_{L^q(B_1)}=1$.
    \end{proof}

\section*{Acknowledgements}
Funded by the Deutsche Forschungsgemeinschaft (DFG, German Research Foundation)-Project-ID 317210226-SFB 1283. D. Zhang is also grateful for the NSFC grants (No. 12271352, 12322108) and Shanghai Frontiers Science Center of Modern Analysis. The research of G. Zhao is also supported by the NSFC grants (Nos. 12288201, 12271352, 12201611) and National Key Research and Development Program of China (No. 2024YFA1013503). 

The authors would like to thank Zimo Hao and Dejun Luo for valuable discussions and suggestions.


\begin{thebibliography}{10}

\bibitem{bahouri2011fourier}
Hajer Bahouri, Jean-Yves Chemin, and Rapha{\"e}l Danchin.
\newblock {\em Fourier analysis and nonlinear partial differential equations},
  volume 343.
\newblock Springer Science \& Business Media, 2011.

\bibitem{Barbu2024}
Viorel Barbu.
\newblock Nonlinear {F}okker-{P}lanck equtions with singular integral drifts
  and {M}ckean-{V}lasov {SDE}s.
\newblock {\em arXiv preprint arXiv: 2410.09822}, 2024.

\bibitem{barbu2023uniqueness}
Viorel Barbu and Michael R\"{o}ckner.
\newblock Uniqueness for nonlinear {F}okker-{P}lanck equations and for
  {M}c{K}ean-{V}lasov {SDE}s: the degenerate case.
\newblock {\em J. Funct. Anal.}, 285(4):Paper No. 109980, 37, 2023.

\bibitem{BR2024nonlinear}
Viorel Barbu and Michael R\"ockner.
\newblock {\em Nonlinear {F}okker-{P}lanck flows and their probabilistic
  counterparts}, volume 2353 of {\em Lecture Notes in Mathematics}.
\newblock Springer, Cham, [2024] \copyright 2024.

\bibitem{barbu2023ns}
Viorel Barbu, Michael R{\"o}ckner, and Deng Zhang.
\newblock Uniqueness of distributional solutions to the 2{D} vorticity
  {N}avier-{S}tokes equation and its associated nonlinear {M}arkov process.
\newblock {\em arXiv preprint arXiv:2309.13910, J. Eur. Math. Soc., to appear},
  2025+.

\bibitem{beck2019stochastic}
Lisa Beck, Franco Flandoli, Massimiliano Gubinelli, and Mario Maurelli.
\newblock Stochastic {ODE}s and stochastic linear {PDE}s with critical drift:
  regularity, duality and uniqueness.
\newblock {\em Electron. J. Probab.}, 24:Paper No. 136, 72, 2019.

\bibitem{bogachev2015fokker}
Vladimir~I. Bogachev, Nicolai~V. Krylov, Michael R{\"o}ckner, and Stanislav~V
  Shaposhnikov.
\newblock {\em {F}okker-{P}lanck-Kolmogorov Equations}, volume 207.
\newblock American Mathematical Soc., 2015.

\bibitem{BCV22}
Tristan Buckmaster, Maria Colombo, and Vlad Vicol.
\newblock Wild solutions of the {N}avier-{S}tokes equations whose singular sets
  in time have {H}ausdorff dimension strictly less than 1.
\newblock {\em J. Eur. Math. Soc.}, 24(9):3333--3378, 2022.

\bibitem{BV19}
Tristan Buckmaster and Vlad Vicol.
\newblock Nonuniqueness of weak solutions to the {N}avier-{S}tokes equation.
\newblock {\em Ann. of Math. (2)}, 189(1):101--144, 2019.

\bibitem{de2022wasserstein}
Paul-Eric Chaudru~de Raynal and Noufel Frikha.
\newblock Well-posedness for some non-linear {SDE}s and related {PDE} on the
  {W}asserstein space.
\newblock {\em J. Math. Pures Appl. (9)}, 159:1--167, 2022.

\bibitem{chen2017heat}
Zhen-Qing Chen, Eryan Hu, Longjie Xie, and Xicheng Zhang.
\newblock Heat kernels for non-symmetric diffusion operators with jumps.
\newblock {\em Journal of Differential Equations}, 263(10):6576--6634, 2017.

\bibitem{CL22}
Alexey Cheskidov and Xiaoyutao Luo.
\newblock Sharp nonuniqueness for the {N}avier-{S}tokes equations.
\newblock {\em Invent. Math.}, 229(3):987--1054, 2022.

\bibitem{cheskidov2023nonuniqueness}
Alexey Cheskidov and Xiaoyutao Luo.
\newblock {$L^2$}-critical nonuniqueness for the 2{D} {N}avier-{S}tokes
  equations.
\newblock {\em Ann. PDE}, 9(2):Paper No. 13, 56, 2023.

\bibitem{evans2010partial}
Lawrence~C. Evans.
\newblock {\em Partial Differential Equations}.
\newblock The American Mathematical Society, 2010.

\bibitem{fabes1986new}
Eugene~B. Fabes and Daniel~W. Stroock.
\newblock A new proof of {M}oser's parabolic {H}arnack inequality using the old
  ideas of {N}ash.
\newblock {\em Arch. Rational Mech. Anal.}, 96(4):327--338, 1986.

\bibitem{FHM14}
Nicolas Fournier, Maxime Hauray, and St\'{e}phane Mischler.
\newblock Propagation of chaos for the 2{D} viscous vortex model.
\newblock {\em J. Eur. Math. Soc. (JEMS)}, 16(7):1423--1466, 2014.

\bibitem{galeati2025almost}
Lucio Galeati.
\newblock Almost-everywhere uniqueness of {L}agrangian trajectories for 3{D}
  {N}avier--{S}tokes revisited.
\newblock {\em J. Math. Pures Appl. (9)}, 200:Paper No. 103723, 2025.

\bibitem{Isabelle2005uniqueness}
Isabelle Gallagher and Thierry Gallay.
\newblock Uniqueness for the two-dimensional {N}avier-{S}tokes equation with a
  measure as initial vorticity.
\newblock {\em Math. Ann.}, 332(2):287--327, 2005.

\bibitem{gallagher2005uniqueness}
Isabelle Gallagher, Thierry Gallay, and Pierre-Louis Lions.
\newblock On the uniqueness of the solution of the two-dimensional
  {N}avier-{S}tokes equation with a {D}irac mass as initial vorticity.
\newblock {\em Math. Nachr.}, 278(14):1665--1672, 2005.

\bibitem{giaquinta1982partial}
Mariano Giaquinta and Michael Struwe.
\newblock On the partial regularity of weak solutions of nonlinear parabolic
  systems.
\newblock {\em Math. Z.}, 179(4):437--451, 1982.

\bibitem{giga1988two}
Yoshikazu Giga, Tetsuro Miyakawa, and Hirofumi Osada.
\newblock Two-dimensional {N}avier-{S}tokes flow with measures as initial
  vorticity.
\newblock {\em Arch. Rational Mech. Anal.}, 104(3):223--250, 1988.

\bibitem{grafakos2008classical}
Loukas Grafakos.
\newblock {\em Classical {F}ourier {A}nalysis}.
\newblock Springer, 2008.

\bibitem{grafner2024weak}
Lukas Gr{\"a}fner and Nicolas Perkowski.
\newblock Weak well-posedness of energy solutions to singular {SDE}s with
  supercritical distributional drift.
\newblock {\em arXiv preprint arXiv:2407.09046}, 2024.

\bibitem{hao2023second}
Zimo Hao, Michael R{\"o}ckner, and Xicheng Zhang.
\newblock Second order fractional mean-field {SDE}s with singular kernels and
  measure initial data.
\newblock {\em accepted by Ann. Probab.}, 2025+.

\bibitem{hao2023sdes}
Zimo Hao and Xicheng Zhang.
\newblock {SDE}s with supercritical distributional drifts.
\newblock {\em arXiv preprint arXiv:2312.11145}, 2023.

\bibitem{hunt1965extension}
Richard~A. Hunt.
\newblock An extension of the {M}arcinkiewicz interpolation theorem to
  {L}orentz spaces.
\newblock {\em Bull. Amer. Math. Soc.}, 70:803--807, 1964.

\bibitem{kinzebulatov2025strong}
Damir Kinzebulatov and Kodjo~Raphael Madou.
\newblock Strong solutions of {SDE}s with singular (form-bounded) drift via
  {R}oeckner-{Z}hao approach.
\newblock {\em Stochastics and Dynamics}, 2025+.

\bibitem{kinzebulatov2022heat}
Damir Kinzebulatov and Yuliy~A. Sem\"enov.
\newblock Heat kernel bounds for parabolic equations with singular
  (form-bounded) vector fields.
\newblock {\em Math. Ann.}, 384(3-4):1883--1929, 2022.

\bibitem{krylov2001heat}
Nicolai~V. Krylov.
\newblock The heat equation in {$L_q((0,T),L_p)$}-spaces with weights.
\newblock {\em SIAM J. Math. Anal.}, 32(5):1117--1141, 2001.

\bibitem{krylov2021stochastic2}
Nicolai~V. Krylov.
\newblock On stochastic equations with drift in {$L_d$}.
\newblock {\em Ann. Probab.}, 49(5):2371--2398, 2021.

\bibitem{krylov2021strong}
Nicolai~V. Krylov.
\newblock On strong solutions of {I}t\^o's equations with {$\sigma\in W^1_d$}
  and {$b\in L_d$}.
\newblock {\em Ann. Probab.}, 49(6):3142--3167, 2021.

\bibitem{krylov2023diffusion}
Nicolai~V. Krylov.
\newblock On diffusion processes with drift in {$L_{d+1}$}.
\newblock {\em Potential Anal.}, 59(3):1013--1037, 2023.

\bibitem{krylov2023strong}
Nicolai~V. Krylov.
\newblock On strong solutions of {I}t\^o's equations with {$D\sigma$} and {$b$}
  in {M}orrey classes containing {$L_d$}.
\newblock {\em Ann. Probab.}, 51(5):1729--1751, 2023.

\bibitem{krylov2025strong}
Nicolai~V. Krylov.
\newblock On weak and strong solutions of time inhomogeneous {I}t\^o's
  equations with {VMO} diffusion and {M}orrey drift.
\newblock {\em Stochastic Process. Appl.}, 179:Paper No. 104505, 23, 2025.

\bibitem{krylov2005strong}
Nicolai~V. Krylov and Michael R\"ockner.
\newblock Strong solutions of stochastic equations with singular time dependent
  drift.
\newblock {\em Probab. Theory Related Fields}, 131(2):154--196, 2005.

\bibitem{LQZZ24}
Yachun Li, Peng Qu, Zirong Zeng, and Deng Zhang.
\newblock Sharp non-uniqueness for the 3{D} hyperdissipative {N}avier-{S}tokes
  equations: beyond the {L}ions exponent.
\newblock {\em J. Math. Pures Appl. (9)}, 190:Paper No. 103602, 64, 2024.

\bibitem{lu2025non}
Huaxiang L{\"u}, Michael R{\"o}ckner, and Xiangchan Zhu.
\newblock Non-uniqueness of (stochastic) {L}agrangian trajectories for {E}uler
  equations.
\newblock {\em arXiv preprint arXiv:2504.16687}, 2025.

\bibitem{ma1992dirichletform}
Zhi~Ming Ma and Michael R\"{o}ckner.
\newblock {\em Introduction to the theory of (nonsymmetric) {D}irichlet forms}.
\newblock Universitext. Springer-Verlag, Berlin, 1992.

\bibitem{mckean1966class}
Henry~P. McKean, Jr.
\newblock A class of {M}arkov processes associated with nonlinear parabolic
  equations.
\newblock {\em Proc. Nat. Acad. Sci. U.S.A.}, 56:1907--1911, 1966.

\bibitem{mishura2016existence}
Yuliya Mishura and Alexander~Yu. Veretennikov.
\newblock Existence and uniqueness theorems for solutions of
  {M}c{K}ean-{V}lasov stochastic equations.
\newblock {\em Theory Probab. Math. Statist.}, (103):59--101, 2020.

\bibitem{mohammed2015sobolev}
Salah-Eldin~A. Mohammed, Torstein~K. Nilssen, and Frank~N. Proske.
\newblock Sobolev differentiable stochastic flows for {SDE}s with singular
  coefficients: applications to the transport equation.
\newblock {\em Ann. Probab.}, 43(3):1535--1576, 2015.

\bibitem{nash1958continuity}
John Nash.
\newblock Continuity of solutions of parabolic and elliptic equations.
\newblock {\em Amer. J. Math.}, 80:931--954, 1958.

\bibitem{nirenberg1966extended}
Louis Nirenberg.
\newblock An extended interpolation inequality.
\newblock {\em Ann. Scuola Norm. Sup. Pisa Cl. Sci. (3)}, 20:733--737, 1966.

\bibitem{oneil1963convolution}
Richard O'Neil.
\newblock Convolution operators and {$L(p,\,q)$} spaces.
\newblock {\em Duke Math. J.}, 30:129--142, 1963.

\bibitem{qian2019parabolic}
Zhongmin Qian and Guangyu Xi.
\newblock Parabolic equations with divergence-free drift in space
  {$L^\ell_tL^q_x$}.
\newblock {\em Indiana Univ. Math. J.}, 68(3):761--797, 2019.

\bibitem{rehmeier2022nonlinear}
Marco Rehmeier and Michael R{\"o}ckner.
\newblock On nonlinear {M}arkov processes in the sense of {M}ckean.
\newblock {\em arXiv preprint arXiv:2212.12424}, 2022.

\bibitem{rockner2021DDSDE}
Michael R\"{o}ckner and Xicheng Zhang.
\newblock Well-posedness of distribution dependent {SDE}s with singular drifts.
\newblock {\em Bernoulli}, 27(2):1131--1158, 2021.

\bibitem{rockner2025strong}
Michael R{\"o}ckner and Guohuan Zhao.
\newblock Sdes with critical time dependent drifts: strong solutions.
\newblock {\em arXiv preprint arXiv:2103.05803}, 2021.

\bibitem{rockner2023weak}
Michael R\"{o}ckner and Guohuan Zhao.
\newblock S{DE}s with critical time dependent drifts: weak solutions.
\newblock {\em Bernoulli}, 29(1):757--784, 2023.

\bibitem{stroock2007multidimensional}
Daniel~W. Stroock and {S}.{R}.~Srinivasa {V}aradhan.
\newblock {\em Multidimensional diffusion processes}.
\newblock Springer, 2007.

\bibitem{veretennikov1980strong2}
Alexander~Ju. Veretennikov.
\newblock Strong solutions and explicit formulas for solutions of stochastic
  integral equations.
\newblock {\em Mat. Sb. (N.S.)}, 111(153)(3):434--452, 480, 1980.

\bibitem{zhang2011stochastic}
Xicheng Zhang.
\newblock Stochastic homeomorphism flows of {SDE}s with singular drifts and
  {S}obolev diffusion coefficients.
\newblock {\em Electron. J. Probab.}, 16:no. 38, 1096--1116, 2011.

\bibitem{zhang2018singular}
Xicheng Zhang and Guohuan Zhao.
\newblock Singular {B}rownian diffusion processes.
\newblock {\em Commun. Math. Stat.}, 6(4):533--581, 2018.

\bibitem{zhang2021stochastic}
Xicheng Zhang and Guohuan Zhao.
\newblock Stochastic {L}agrangian path for {L}eray's solutions of 3{D}
  {N}avier-{S}tokes equations.
\newblock {\em Comm. Math. Phys.}, 381(2):491--525, 2021.

\bibitem{zhao2024existence}
Guohuan Zhao.
\newblock Existence and uniqueness for {M}c{K}ean-{V}lasov equations with
  singular interactions.
\newblock {\em Potential Anal.}, 62(3):625--653, 2025.

\bibitem{zvonkin1974transformation}
Alexander~K. Zvonkin.
\newblock A transformation of the phase space of a diffusion process that will
  remove the drift.
\newblock {\em Mat. Sb. (N.S.)}, 93(135):129--149, 152, 1974.

\end{thebibliography}
\end{document}